\documentclass[11pt]{article}

\usepackage{graphicx} % Required for inserting images
\graphicspath{ {./images/} } %declares directory where the images are saved
\usepackage{xcolor}
\pdfoutput=1 
\usepackage[margin=1in]{geometry}
\usepackage{amsthm,amssymb,enumerate, bbm ,graphicx,color,caption,upgreek, float, tikz, subcaption,booktabs,longtable, appendix,graphics, pdfpages,rotating,mathtools, array, bm, blkarray,setspace,textcomp,ytableau, tabularx}
\usepackage{oubraces}

\usepackage[pdftex]{hyperref}

\usepackage{relsize}
\usepackage{verbatim}
\usepackage{amsmath}
\usepackage{amsthm}
\usepackage{amssymb}
\usepackage{makecell}
\usepackage{multirow}
\usepackage{tcolorbox}
\usepackage{hyperref}

\tcbuselibrary{skins}   % prime numbers
\newcommand{\N}{\ensuremath{\mathbb{N}}}   % natural numbers
\newcommand{\Q}{\ensuremath{\mathbb{Q}}}   % rational numbers
\newcommand{\R}{\ensuremath{\mathbb{R}}}   % real numbers
\newcommand{\K}{\ensuremath{\mathbb{K}}}   % field K
\newcommand{\C}{\ensuremath{\mathbb{C}}} %complex numbers
\newcommand{\T}{\ensuremath{\mathbb{T}}} % complex numbers of norm 1

\newcommand{\Perm}{\mathfrak G}

\newcommand{\MU}{{\mathcal U}}

\newcommand{\MS}{{\mathcal S}}

\newcommand{\MW}{{\mathcal W}}
\newcommand{\MK}{{\mathcal K}}
\newcommand{\MT}{{\mathcal T}}
\newcommand{\MDU}{{\mathcal {DU}}}
\newcommand{\MDO}{{\mathcal {DO}}}
\newcommand{\Herm}{\mathrm{Herm}}
\newcommand{\Cnn}{\C^n\ot \C^n}
\newcommand{\Cnnt}{\C^n\ot (\C^n)^{\ot t}}
\newcommand{\Cnt}{(\C^n)^{\ot t}}
\newcommand{\MKt}{\MK^{(t)}}
\newcommand{\sS}{\mathbb S}

\newcommand{\PCP}{\text{\rm PCP}}
\newcommand{\TCP}{\text{\rm TCP}}
\newcommand{\SEP}{\text{\rm SEP}}
\newcommand{\CP}{\text{\rm CP}}
\newcommand{\COP}{\text{\rm COP}}
\newcommand{\CLDUI}{\text{\rm CLDUI}}
\newcommand{\LDUI}{\text{\rm LDUI}}
\newcommand{\LDOI}{\text{\rm LDOI}}
\newcommand{\DPS}{\text{\rm DPS}}
\newcommand{\DPSt}{\mathrm{DPS}^{(t)}}
\newcommand{\SEPBS}{\SEP^{\BS}}
\newcommand{\tilDPS}{\widetilde {\DPS}}
\newcommand{\tilDPSt}{\widetilde{\DPS}^{(t)}}
\newcommand{\tilF}{\widetilde F}
\newcommand{\tilMK}{\widetilde\MK}

\newcommand{\PCOP}{\text{\rm PCOP}}

\newcommand{\Tr}{\text{\rm Tr}}
\newcommand{\cone}{\text{\rm cone}}
\newcommand{\conv}{\text{\rm conv}}
\newcommand{\ou}{\overline u}

\newcommand{\ox}{\overline x}
\newcommand{\oy}{\overline y}
\newcommand{\diag}{\text{\rm diag}}
\newcommand{\Diag}{\text{\rm Diag}}
\newcommand{\ot}{\otimes}
\newcommand{\PiLDUI}{\Pi_{\LDUI}}

\newcommand{\PiLDUIt}{\Pi_{\LDUI^{(t)}}}

\newcommand{\BS}{\text{\rm BS}}

\newcommand{\supp}{\text{\rm supp}}
\newcommand{\sfT}{T}

\newcommand{\ui}{\underline{i}}
\newcommand{\uj}{\underline{j}}
\newcommand{\uk}{\underline{k}}
\newcommand{\uh}{\underline{h}}

\newcommand{\ovU}{\overline U}
\newcommand{\ovx}{\overline x}
\newcommand{\ovy}{\overline y}

\newcommand{\Rea}{\text{\rm Re}}
\newcommand{\Ima}{\text{\rm Im}}
\newcommand{\bfi}{\mathbf i}
\newcommand{\edge}{\stackrel{1}{\simeq}}
\newcommand{\edgetwo}{\stackrel{2}{\simeq}}
\newcommand{\tftwo}{{\lfloor (t+1)/2\rfloor}}

\usepackage{amsmath}
\usepackage{amsfonts}
\usepackage{amssymb}
\usepackage{amsthm} 

\theoremstyle{plain}
\newtheorem{satz}{Satz}[section]
\newtheorem{lemma}[satz]{Lemma}
\newtheorem{corollary}[satz]{Corollary}
\newtheorem{definition}[satz]{Definition}
\newtheorem{theorem}[satz]{Theorem}
\newtheorem{remark}[satz]{Remark}
\newtheorem{example}[satz]{Example}
\newtheorem{conjecture}[satz]{Conjecture}

\makeatletter
\renewenvironment{proof}[1][\proofname]{%
  \par
  \pushQED{}% remove automatic qed
  \normalfont
  \topsep6\p@\@plus6\p@
  \trivlist
  \item[\hskip\labelsep\itshape #1\@addpunct{.}]% <-- italic heading
}{%
  \popQED
  \endtrivlist
}
\makeatother
\newcommand{\ML}{\textcolor{blue}}

\newcommand{\ignore}[1]{}

 \title{Semidefinite hierarchies for diagonal unitary invariant bipartite quantum states}
\author{Jonas Britz \thanks{Centrum Wiskunde \& Informatica (CWI), Amsterdam, and LAAS-CNRS, Toulouse \url{Jonas.Britz@cwi.nl}}
\and Monique Laurent \thanks{Centrum Wiskunde \& Informatica (CWI), Amsterdam, and Tilburg University, \url{Monique.Laurent@cwi.nl}}
}

\date{\today}

\begin{document}

\maketitle

\begin{abstract}
We investigate questions about the cone $\SEP_n$ of separable bipartite states, consisting of the Hermitian matrices acting on   $\C^n\ot\C^n$ that can be written as conic combinations of rank one matrices of the form $xx^*\ot yy^*$ with $x,y\in\C^n$. % in the complex unit sphere. 
Bipartite states that are not separable are said to be  entangled. Detecting quantum entanglement is a fundamental task in quantum information and a hard computational problem. We explore the Doherty-Parrilo-Spedaglieri (DPS) hierarchy of semidefinite conic approximations for $\SEP_n$ when the bipartite states have some additional structural properties: first, (i)  for states with diagonal unitary invariance, and second (ii) for states with Bose symmetry. In   case (i) we show that the DPS hierarchy can be block diagonalized, which, combining with its moment reformulation,  leads to a substantially more efficient implementation. In case (ii), we give a characterization of the dual hierarchy, in terms of %real-sums-of-squares for 
sums of squares of Hermitian complex polynomials, extending a known result in the generic case. It turns out that the completely positive cone $\CP_n$, its dual cone $\COP_n$, and their sums-of-squares based conic approximations $\MK^{(t)}_n$, play a central role in these two settings (i),(ii). We clarify these connections and test the block diagonal relaxations on   classes of examples.
\end{abstract}

%\keywords{quantum bipartite states \and entanglement \and separable states \and DPS hierarchy \and  \and sum-of-squares polynomial \and moment approach   \and completely positive matrix \and copositive matrix \and semidefinite programming \and sparsity pattern \and symmetry}
%\subclass{90C22; 90C26; 90C30}

\tableofcontents

\section{Introduction}\label{sec:introduction}

In quantum information theory, a  bipartite quantum state corresponds to a Hermitian positive semidefinite matrix\footnote{We consider here {\em unnormalized} states, i.e., we do not require that $\rho_{AB}$ has trace one. This restriction, which is  common  in quantum information, is however just a matter of scaling.}  $\rho_{AB}$ acting on a tensor product space $\C^{n_A}\ot\C^{n_B}$.
Each of the two Hilbert spaces $\C^{n_A}$ and $\C^{n_B}$ corresponds to a part  of a bipartite quantum system, also referred to as its A- and B-registers;
 the notation $\rho_{AB}$ is meant to reflect this fact and is commonly used  in quantum information theory.
 Then, the state $\rho_{AB}$ is said to be  {\em separable} when it can written as a sum of tensor product states, of the form $X\ot Y$ with $X\in \Herm^{n_A}_+$ and $Y\in \Herm^{n_B}_+$, in which case $\rho_{AB}$ acts `separately' on each register of the quantum system. When $\rho_{AB}$ is not separable it is said to be {\em entangled}. Entanglement is a physical resource that can be used to carry out more efficiently on a quantum computer certain tasks in computation, communication, teleportation, etc. We refer, e.g., to the books \cite{Nielsen_Chuang,Watrous} for background in quantum information theory.  From now on we assume  $n_A=n_B=n$, but we may sometimes still use the letters $n_A$, $n_B$ to stress the roles of the two A- and B-registers in the system.
 The set of separable states in $\Herm(\Cnn)$ 
 can  alternatively be described as
\begin{align}\label{eq:SEP}
\begin{split}
\SEP_n 
= \cone\{ xx^*\ot yy^*: x,y\in\C^n, \|x\|=\|y\|=1\}\subseteq \Herm(\Cnn). 
\end{split}
\end{align}
Testing whether a given bipartite state is separable or entangled is a fundamental problem in quantum information theory, which has been shown to be  computationally hard  (Gurvits \cite{Gurvits}).  

Doherty, Parrilo and Spedaglieri \cite{DPS_2002,DPS_2004} introduced a hierarchy of semidefinite relaxations $\DPS_n^{(t)}$ (for $t\ge 1$) for the separable cone $\SEP_n$, having the property that it is {\em complete}: 
%equality $\bigcap_{t\ge 1}\DPSt_n=\SEP_n$ holds. In other words,
if a state $\rho_{AB}$ is entangled, then $\rho_{AB}\not\in\DPSt_n$ for some order $t$, i.e., entanglement is detected at some finite level of the DPS hierarchy. Testing membership in the relaxation $\DPSt_n$ amounts to testing feasibility of a  semidefinite program involving $t+1$ matrices of size  $n^{t+1}$, thus rapidly  increasing  with the relaxation order $t$. This  motivates investigating more economical relaxations when dealing with classes of bipartite states that enjoy some  structural properties.

In this paper we investigate the following structural properties: when the bipartite state is {\em invariant under some diagonal unitaries}, leading to classes of bipartite states coined in the literature as $\CLDUI_n$, $\LDUI_n$  and $\LDOI_n$, and when the bipartite state is {\em Bose symmetric}. 
For separable bipartite states enjoying some of these structural properties, there are interesting connections  with the cone $\CP_n$ of completely positive matrices,  a well-studied matrix cone in matrix theory and optimization.  % (and the dual cone $\COP_n$ of copositive matrices) 
Some natural extensions of $\CP_n$ have been introduced in quantum information theory, known as $\PCP_n$ and $\TCP_n$, consisting of all pairwise (triplewise) completely positive matrices (exact definitions follow below). %in Section \ref{sec:intro-PCP}.
Via conic duality there are also intimate links to positive polynomials and their characterizations in terms of sums of squares of polynomials.
%, that we briefly review in Section \ref{sec:intro-dualobjects}. 
This leads in particular to connections with the copositive cone $\COP_n$, the dual cone of $\CP_n$, and some of its semidefinite approximation hierarchies.

\subsubsection*{Diagonal unitary invariant bipartite  states }\label{sec:intro-PCP}

Following Johnston and MacLean \cite{JML-PCP} and Singh and Nechita \cite{SN-CLDUI}, we say  that a bipartite state $\rho_{AB}\in \Herm(\Cnn)$ is  {\em conjugate local diagonal unitary invariant} \footnote{The name reflects the fact that each $U$ is `diagonal unitary' and `conjugate local'  refers to the fact that $U$ acts on the A-register while its conjugate $\ovU$ acts on the B-register.}
(CLDUI, for short) 
if  
\begin{align}\label{eq:CLDUI}
\rho_{AB}= (U\ot \ovU) \rho_{AB} (U\ot \ovU)^* \ \text{ for all } U\in\MDU_n.
\end{align}
Here,  $\MDU_n$ is the set of diagonal unitaries, which consists of the matrices $U=\Diag(u)$ for $u\in \T^n$, where $\T=\{z\in\C: |z|=1\}$.
%i.e., with $u\in\C^n$ and $|u_i|=1$ for $i\in [n]$. 
Analogously, $\rho_{AB}$ is {\em  local diagonal unitary invariant} (LDUI) if
\begin{align}\label{eq:LDUI}
\rho_{AB}= (U\ot U) \rho_{AB} (U\ot U)^* \ \text{ for all } U\in\MDU_n
\end{align}
and $\rho_{AB}$ is {\em  local diagonal orthogonal invariant} (LDOI) if it satisfies 
\begin{align}\label{eq:LDOI}
\rho_{AB}= (O\ot O)\rho_{AB} (O\ot O)^* \ \text{ for all } O=\Diag(o) \text{ with } o\in \{\pm 1\}^n.
\end{align}
% for all diagonal orthogonal matrices $U$, i.e., of the form $U=\Diag(u)$ with $u\in \{\pm 1\}^n$. 
Let $\CLDUI_n$, $\LDUI_n$ and $\LDOI_n$ denote the  sets of bipartite states with the relevant invariance properties, so that $\CLDUI_n\cup \LDUI_n\subseteq \LDOI_n$.
If, in (\ref{eq:CLDUI}) and (\ref{eq:LDUI}), one would ask invariance under {\em all unitaries} (instead of all diagonal unitaries), then one  gets the class of (CLDUI) Werner states \cite{Werner} and (LDUI) isotropic states \cite{HH-isotropic}. %, see \cite{SN-CLDUI}.}
As observed in \cite{JML-PCP,SN-CLDUI} and recalled in (\ref{eq:suppCLDUI})-(\ref{eq:suppLDOI}) below, CLDUI, LDUI, and LDOI states have a very special sparsity pattern. 

\smallskip
A matrix $\rho_{AB}\in\Herm(\Cnn)$  is indexed by $[n]^2$, whose elements are  the sequences $(i,j)$ in $[n_A]\times [n_B]$, simply denoted as $ij$ for compact notation; so\footnote{The letters $n_A$, $n_B$ are used to stress the different roles of the indices, but recall   $n_A=n_B=n$.}, $\rho_{AB}=((\rho_{AB})_{ij,kl})_{i,k \in [n_A], j,l\in [n_B]}$. Its {\em support} $\supp(\rho_{AB})$    is defined as the set of positions corresponding to nonzero entries of $\rho_{AB}$, thus consisting of the pairs $(ij, kl)\in [n]^2\times [n]^2$ with $(\rho_{AB})_{ij,kl}\ne 0$.
Define the sets
\begin{align*}
\Omega_n=\{ (ij,ij): i, j\in [n]\}\cup \{(ii,jj): i,j\in [n]\},\\
\Phi_n=\{(ij,ij): i, j\in [n]\}\cup \{(ij,ji): i,j\in [n]\}.
\end{align*}
 For a matrix $\rho_{AB}\in \Herm(\Cnn)$, the following  equivalences hold:
\begin{align}
\rho_{AB} \text{ is CLDUI} &\Longleftrightarrow \supp(\rho_{AB})\subseteq
 \Omega_n,
 %=\{ (ij,ij): i, j\in [n]\}\cup \{(ii,jj): i,j\in [n]\},
 \label{eq:suppCLDUI}\\
\rho_{AB} \text{ is LDUI}  &\Longleftrightarrow \supp(\rho_{AB})\subseteq \Phi_n,
%=\{(ij,ij): i, j\in [n]\}\cup \{(ij,ji): i,j\in [n]\},
\label{eq:suppLDUI}\\
\rho_{AB} \text{ is LDOI}  &\Longleftrightarrow \supp(\rho_{AB})\subseteq
\Omega_n\cup\Phi_n.\label{eq:suppLDOI}
\end{align}
As an illustration,  we  give the (easy) argument for (\ref{eq:suppCLDUI}). Let $U=\Diag(u)$ with $u\in\T^n$, and let $i,j,k,l\in [n]$.  Then, we have
$((U\ot\ovU) \rho_{AB} (U\ot \ovU)^*)_{ij,kl}= u_i\overline{u_j}\ \overline {u_k}u_l (\rho_{AB})_{ij,kl}$, which is equal to $(\rho_{AB})_{ij,kl}$  if $(i,k)=(j,l)$ or $(i,j)=(k,l)$. At any other position, there exists $u\in\T^n$ such that $u_i\overline{u_j}\ \overline {u_k}u_l \ne 1$, which implies $(\rho_{AB})_{ij,kl}=0$.
This shows (\ref{eq:suppCLDUI}) and the argument is similar for (\ref{eq:suppLDUI})-(\ref{eq:suppLDOI}).

Hence,   the nonzero entries of an LDOI matrix $\rho_{AB}\in\Herm(\Cnn)$   are fully captured by a triplet of matrices $(X,Y,Z)\in \R^{n\times n} \times \Herm^n\times \Herm^n$, obtained by setting
\begin{align}\label{eq:rho-ABC}
X=((\rho_{AB})_{ij,ij})_{i,j=1}^n,\quad Y=((\rho_{AB})_{ii,jj})_{i,j=1}^n,\quad Z=((\rho_{AB})_{ij,ji})_{i,j=1}^n;
\end{align}
note that  $\diag(X)=\diag(Y)=\diag(Z)$ holds. 
Conversely,  given $(X,Y,Z)\in \R^{n\times n} \times \Herm^n\times \Herm^n$ with
$\diag(X)=\diag(Y)=\diag(Z)$, one can define a matrix 
$\rho_{AB} \in \Herm(\Cnn)$ via relation (\ref{eq:rho-ABC}), with zero entries at all other positions;  following \cite{SN-CLDUI}, this matrix is denoted $\rho_{(X,Y,Z)}$.
For the reader's convenience,  its explicit description reads
\begin{align}\label{eq:rhoXYZ0}
\begin{split}
&\rho_{(X,Y,Z)} \\
&=\sum_{i,j=1}^n X_{ij} e_ie_i^*\ot e_je_j^* +\sum_{i\ne j\in [n]} Y_{ij} e_ie_j^*\ot e_ie_j^*
+\sum_{i\ne j\in [n]} Z_{ij} e_ie_j^* \ot e_je_i^*\\
& =\sum_{i,j=1}^n X_{ij} (e_i\ot e_j) (e_i\ot e_j)^* +\sum_{i\ne j\in [n]} Y_{ij} (e_i\ot e_i)(e_j\ot e_j)^*\\
& \quad +\sum_{i\ne j\in [n]} Z_{ij} (e_i\ot e_j)(e_j\ot e_i)^*.
\end{split}
\end{align}
Hence,  $\rho_{(X,Y,Z)}$ is LDOI by construction. Set $d=\diag(X)=\diag(Y)=\diag(Z)\in\R^n$. When $Y$ is diagonal, set $\rho_{(X,\cdot,Z)}=\rho_{(X,\Diag(d),Z)}$, which is LDUI by construction, and 
when $Z$ is diagonal, set $\rho_{(X,Y)}=\rho_{(X,Y,\Diag(d))}$, which is CLDUI by construction. 
In other words, $\rho_{(X,Y,Z)}$ can be written (after suitably permuting its rows/columns) as a block-diagonal matrix of the form
\begin{align}\label{eq:rhoXYZ-block}
\rho_{(X,Y,Z)}\simeq Y \oplus \displaystyle\oplus_{1\le i<j\le n} \begin{pmatrix} X_{ij} & Z_{ij} \cr Z_{ji} & X_{ji}\end{pmatrix}.
\end{align}
Hence, $\rho_{(X,Y,Z)}\succeq 0$ if and only if $Y\succeq 0$ and each of the above $2\times 2$ matrices is positive semidefinite. Furthermore, $\Tr [\rho_{(X,Y,Z)}]=1$ (i.e. $\rho_{(X,Y,Z)}$ is a normalized quantum state) if and only if $\sum_{i,j\in[n]}X_{i,j}=1$.

Observe that going from  the CLDUI sparsity pattern to the LDUI sparsity pattern amounts to ``taking the partial transpose $T_B$ with respect to the B-register" (see (\ref{eq:transpose1})). Roughly, this means switching the second indices,  i.e., mapping position $(ih,jk)$ to position $(ik,jh)$. In particular, 
\begin{align}\label{eq:rhoXYZ-block-PT}
\rho_{(X,Y,Z)}^{T_B}=\rho_{(X,Z,Y)},\quad \rho_{(X,Y)}^{T_B}= \rho_{(X,\cdot,Y)}.
\end{align}
So, the CLDUI and LDOI sparsity structures can be seen as the main players, but additional symmetry structure can be added to LDUI states as we see below.
%: $\rho_{(X,\cdot,X)}=\rho_{(X,X)}^{T_B}$ is indeed both LDUI and Bose symmetric.  \ML{See Example \ref{Ex?} below for an illustration for the case $n=3$.}
 
\subsubsection*{(Pairwise/triplewise)   completely positive matrices}

Following \cite{JML-PCP,SN-CLDUI}, define  the cone $\PCP_n$ (resp., the cone $\TCP_n$) consisting of all matrix pairs $(X,Y)$ for which the bipartite state $\rho_{(X,Y)}$ is separable (resp., all matrix triplets $(X,Y,Z)$ for which $\rho_{(X,Y,Z)}$ is separable).  
 It turns out that these two cones can be seen as  generalizations of the classical {\em completely positive cone} $\CP_n$, defined as 
  \begin{align}\label{eq:CP}
 \CP_n=\conv\{xx^\sfT: x\in\R^n_+\}.
 \end{align}
Indeed, the following links between the cones $\CP_n$, $\PCP_n$ and $\TCP_n$   are shown in \cite{JML-PCP,SN-CLDUI} (recalled in Lemma \ref{lem:CP-PCP-TCP} with a short proof).  For a matrix $X\in \MS^n$, 
 \begin{align}\label{eq:CP-PCP-TCP}
\begin{split}
X\in \CP_n & \Longleftrightarrow  (X,X)\in \PCP_n \ \ (\text{\rm i.e., } \rho_{(X,X)}\in\SEP_n)  \\& \Longleftrightarrow (X,X,X)\in \TCP_n 
\ \ (\text{\rm i.e., } \rho_{(X,X,X)}\in\SEP_n).  
\end{split}
\end{align}
Accordingly, elements of $\PCP_n$ ($\TCP_n$) are called {\em pairwise (triplewise) completely positive}. 

\subsubsection*{Bose symmetric bipartite states}

 Let $\BS(\Cnn)\subseteq\Herm(\Cnn)$ denote the set of {\em Bose symmetric} (or simply {\em symmetric}) bipartite  states $\rho_{AB}$, i.e., satisfying the following permutation invariance property: 
 $$(\rho_{AB})_{ij,kl}=(\rho_{AB})_{ij,lk} \text{  for all } i,j,k,l\in [n].$$
Hence, any $\rho_{AB}\in\BS(\C^n\ot \C^n)$ leaves $S^2(\C^n)$ invariant and vanishes on $(S^2(\C^n))^\perp$, so that $\BS(\C^n\ot \C^n)\simeq\Herm(S^2(\C^n))$. Clearly, any state $xx^*\ot xx^*$ is Bose symmetric and, in fact,  these states  generate  the cone  $\SEPBS_n$, defined as the  set of separable Bose symmetric states:
\begin{align}\label{eq:SEP-BS}
\SEPBS_n:=\SEP_n\cap\BS(\Cnn)=\cone\{xx^*\ot xx^*: x\in \sS^{n-1}\}.
\end{align}
(See Lemma \ref{lem:SEP-BS} for a short argument.)
%Note that, for any   $X\in\MS^n$ and $Y\in \Herm^n$ with $\diag(X)=\diag(Y)$, the associated LDOI state $\rho_{(X,Y,X)}$ is Bose symmetric. 
Note that Bose symmetric states with LDOI pattern arise in the following way, as directly follows using relation (\ref{eq:rho-ABC}): for a 
bipartite state $\rho_{AB}\in\Herm(\C^n\ot \C^n)$,
\begin{align}\label{eq:LDUI-BS}
\begin{split}
&\rho_{AB}\in \LDOI_n \cap \BS(\C^n\ot \C^n)\Longleftrightarrow  \\
&\rho_{AB}=\rho_{(X,Y,X)} \text{ for some } X\in\MS^n, Y\in\Herm^n \text{ with } \diag(X)=\diag(Y).
\end{split}\end{align}
%Indeed, if $\rho_{AB}\in\LDUI_n$, then $\rho_{AB}=\rho_{(X,\cdot,Z)}$ for some $X\in\R^{n\times n}$, $Z\in\Herm^n$ with $\diag(X)=\diag(Z)$, and, moreover,   $X=Z$ if $\rho_{AB}$ is Bose symmetric. 
 We will give a class of examples in Section \ref{sec:exampleab} as  an illustration for the case $n=3$.

\subsubsection*{Our contributions and organization of the paper}
 The above connections suggest the following natural questions:\\ % that we address in this paper:\\
$\bullet$  Is it possible to adapt the DPS hierarchy in order to test separability of CLDUI (or LDUI, LDOI) bipartite quantum states in a more economical way? \\
$\bullet$  Is there a   link between the DPS hierarchy for separable bipartite states and the well-known hierarchy of semidefinite relaxations for the cone $\CP_n$ (and its dual, the copositive cone $\COP_n$)? \\
 These questions form the main motivation for our work.  We address the first question in Section~\ref{sec:DPS-CLDUI} and the second one in Section \ref{sec:DPS-symmetric}.

In  Section~\ref{sec:DPS-CLDUI}  we show that the sparsity structure of CLDUI (LDUI, LDOI) states can be transported to every level of  the DPS hierarchy. In a nutshell, we show that the matrices involved in the semidefinite program modeling the DPS relaxation at any given  level $t$ inherit the sparsity structure of the original bipartite state (Theorem  \ref{theo:CLDUItDPSCert}) and, moreover, that they  enjoy a block-diagonal structure (Theorem \ref{theo:LDUILDOItDPSCert}).   As an application, solving these semidefinite programs is much more efficient, which makes it possible to compute the relaxations at higher order.
In Corollary \ref{cor:sizeclique},  we give an estimate on the %gain in size 
 size of the matrices involved in the semidefinite programs modeling the level $t$ DPS relaxation in the tensor setting and, in Lemma \ref{lem:LDOIBlockRatio},  we indicate the maximum block size when reformulating the relaxation in the moment framework. We also refer to Tables \ref{tab:GenericBlockSize}, \ref{tab:LDOIBlockSize}, \ref{tab:CLDUIBlockSize} that show the comparative explicit matrix sizes for small $n$ and $t$ (when reformulating the SDP in the polynomial optimization framework). We report about computational experiments in  Section~\ref{sec:runtime}.

  It turns out that the second question -- about links between the hierarchy of conic approximations for $\COP_n$ and the DPS hierarchy -- was  considered  in the recent paper \cite{GNP_2025}, of which we learned while completing   the present manuscript. We revisit this question in 
Section \ref{sec:DPS-symmetric}, where we additionally offer some sharpenings and more compact arguments. The analysis relies in a crucial manner on the dual objects, the dual cone of  $\SEP_n$ and the copositive cone $\COP_n$  that tie naturally with positive polynomials and sums of squares. In particular, we show a characterization  of the dual of the variation   (\ref{eq:DPSt-BS})   of the DPS hierarchy adapted to Bose symmetric states     (Theorem \ref{theo:FF-BS}), which fully mirrors an earlier result of Fang and Fawzi \cite{FF-sphere} for general bipartite states. Based on this characterization, one can then easily establish a correspondence between the hierarchy of inner conic approximations   $\MK^{(t)}_n$   for the copositive cone $\COP_n$,  and the dual of the symmetry-adapted DPS hierarchy  (Theorem \ref{theo:DPS-KCOP}), thereby essentially recovering a recent result of \cite{GNP_2025}.
 
\medskip
 The paper is organized as follows. In Section \ref{sec:background} we introduce the main concepts and provide background facts on the main players in the paper, including  the  cones $\CP_n$, $\SEP_n$, $\DPS^{(t)}_n$, $\PCP_n$ and their duals, and we review basic properties of bipartite states with diagonal unitary invariance. In Section \ref{sec:DPS-CLDUI}, we investigate how to exploit the sparsity structure of CLDUI (LDUI, LDOI) states in order to design an adapted, more economical  variation of the DSP hierarchy; we show that this leads to a reformulation involving matrices with a block-diagonal structure. Section \ref{sec:DPS-symmetric} is devoted to Bose symmetric bipartite states; we first give a characterization of the dual symmetry-adapted DPS hierarchy, which we then use to establish a correspondence between the hierarchy of cones $\MK^{(t)}_n$ and the symmetry-adapted DPS hierarchy.
 In Section  \ref{sec:implementation} we discuss several issues regarding the practical implementation of the DPS hierarchy. We first indicate how to reformulate it in the `moment setting', which is the right setting permitting to remove the redundancies that are inherent to the `tensor setting' in which the DPS hierarchy is originally defined. This applies to  general bipartite states, as well as to bipartite states with some diagonal unitary invariance. After that, we report on numerical experiments that have been carried out on classes of bipartite states with LDOI sparsity structure, motivated by the so-called  PPT$^2$ conjecture from quantum information theory.

\section{Preliminaries and background results}\label{sec:background}

\subsection{Preliminaries on notation }
We begin with introducing  notation on matrices, tensors, and polynomials,
used throughout the paper.   Thereafter, we group basic definitions and preliminary results about the various objects that are the main players in the paper. This includes  the DPS hierarchy and the associated dual cones, a semidefinite  hierarchy for the copositive cone and the associated dual cones for (pairwise) completely positive matrices, and basic facts on bipartite states with diagonal unitary invariance.
So, this section offers a road map through the main players of the paper and can serve as a gentle introduction to the subject.

\subsubsection*{Matrices and tensors}

For  integers $1\le k\le n$, set $[k:n]=\{k,k+1,\ldots, n-1,n\}$ and $[n]=[1:n]$, and let $\Perm_n$ denote the group of permutations of the set $[n]$. Let $\T\subseteq \C$ denote the set of complex numbers $z$ with squared modulus $|z|^2={\overline z z}=1$. We use the notation $\bfi=\sqrt{-1}\in\T$. For a complex number $z\in \C$, $\Rea(z)=(z+\overline z)/2$ denotes its real part; the notation extends to vectors and matrices by taking the entrywise real part. Let $e_1,\ldots,e_n$ denote the standard unit vectors that form an orthonormal basis of $\K^n$, for $\K=\R$, $\C$. For   $u\in\K^n$, $\|u\|= \sqrt {u^*u}$ is the Euclidean norm, where $u^*=\overline u^T$ denotes the conjugate transpose of $u$. Let $\sS^{n-1}=\{u\in \C^n: \|u\|=1\}$ denote the unit sphere in $\C^n$.

We use the trace inner product $\langle X,Y\rangle =\Tr(X^*Y)$ for matrices $X,Y\in \K^{n\times n}$, with $X^*=\overline X^T$ denoting the conjugate transpose of $X$. A matrix $X$ is Hermitian if $X^*=X$ and positive semidefinite, denoted as $X\succeq 0$, if all its eigenvalues are nonnegative. Let $\Herm^n$ (resp., $\MS^n$) denote the set of $n\times n$ Hermitian (resp., real symmetric) matrices and $\Herm^n_+$ (resp., $\MS^n_+$) its subcone of positive semidefinite matrices. 
For a matrix $X\in \C^{n\times n}$, $\ker X$ denotes its kernel (all vectors $w\in\C^n$ such that $Xw=0$).
Moreover, $\diag(X)=(X_{ii})_{i=1}^n$ denotes the vector consisting of its diagonal entries and, for a vector $x\in \C^n$, $\Diag(x)\in\C^{n\times n}$ denotes the diagonal matrix with $X_{ii}=x_i$ for $i\in [n]$, and $D_X=\Diag(\diag(X))$ denotes the diagonal matrix with same diagonal entries as $X$. We let $I_n$, $J_n$  denote, respectively, the  identity matrix and the all-ones matrix of size $n\times n$.  

The support of a matrix $X\in \Herm^n$ is the set $\supp(X)=\{(i,j)\in [n]^2: i\ne j, X_{i,j}\ne 0\}$ consisting of the off-diagonal positions corresponding to nonzero entries. Then,  $X$ is said to be supported on a graph $G=([n],E)$ if $\supp(X)\subseteq E$, i.e., $X_{i,j}=0$ if $(i,j)$ is not an edge of $G$.

We use the symbol `$\circ$' to denote the Hadamard product and `$\otimes$' to denote the tensor product. Given two vectors $u,v \in \C^n$, their Hadamard product is $u\circ v=(u_iv_i)_{i=1}^n\in\C^n$. The definition applies analogously to matrices: $A\circ B=(A_{ij}B_{ij})_{i\in [m], j\in [n]}$ for $A,B\in\C^{m\times n}$. Given vectors $u\in \C^{n_1}$, $v\in \C^{n_2}$, their tensor product is
$u\ot v\in \C^{n_1}\ot \C^{n_2}$ with entries $(u\ot v)_{i_1i_2}=u_{i_1}v_{i_2}$ for $i_1\in [n_1],i_2\in [n_2]$. Analogously, given matrices $A\in\C^{m_1\times n_1}, B\in \C^{m_2\times n_2}$, their tensor product is
$A\ot B\in \C^{m_1  \times n_1 }\ot \C^{m_2\times n_2}$ with entries 
$(A\ot B)_{(i_1,i_2),(j_1,j_2)}=A_{i_1,j_1}B_{i_2,j_2}$ for $(i_1,i_2)\in [m_1]\times[m_2]$ and $(j_1,j_2)\in [n_1]\times [n_2]$.
Given $A\in \C^{n\times n}$, the following (easy) observation will be often used later:
\begin{align}\label{eq:Au}
A\circ uu^* = \Diag(u) A \Diag(u)^* \ \text{ for any } u\in \C^n.
\end{align}
Given $X\in \Herm^n$, $Y\in \Herm^m$, 
$$X\oplus Y =\begin{pmatrix}X & 0\cr 0 & Y\end{pmatrix}\in \Herm^{n+m} $$
is their direct sum, which has a block-diagonal form. Such block-diagonal structure is useful since   $X\oplus Y\succeq 0$ if and only if $X,Y\succeq 0$, where the latter is easier to test as it involves smaller matrices. 

\smallskip
We will  deal with bipartite states, i.e., matrices that act on a tensor product space $\Cnn$. While $\Cnn$ can be identified with $\C^{n^2}$,  we also use the notation $\Herm(\Cnn)$, $\Herm(\Cnn)_+$ (instead of $\Herm^{n^2}$, $\Herm^{n^2}_+$) when it is important  to stress the tensor product structure. This notation will be especially useful when considering states that act 
 on a tensor product space of the form $\Cnt$ (for some $t\ge 1$). Then, it  is   convenient (and common in quantum information) to refer to the various copies of $\C^n$  as the  {\em registers} of the quantum system $\Cnt$.

\smallskip
Consider the tensor   space $\Cnt$ (for   $t\ge 1$). 
A $t$-tensor $v\in \Cnt$ is indexed by sequences $\ui=(i_1,\ldots,i_t)\in [n]^t$, so $v=(v_{\ui})_{\ui\in [n]^t}$. For  compact notation, we also use the notation $\ui=i_1\ldots i_t$ for sequences in $[n]^t$. 
For an integer $1\le s\le t$, we let $\ui \uj\in [n]^t$ denote the concatenation of   $\ui\in [n]^s$ and $\uj\in [n]^{t-s}$. 
Given a sequence   $\ui=i_1\ldots i_t \in[n]^t$, we let $\{\ui\}=\{i_1,\ldots,i_t\}\subseteq [n]$ denote the associated multiset, and define $\alpha(\ui)=(\alpha(\ui)_k)_{k=1}^n\in \N^n$, where $\alpha(\ui)_k$ denotes the number of occurrences of the symbol $k$ in the multiset $\{\ui\}$. Then, $|\alpha(\ui)|=\sum_{k=1}^n \alpha(\ui)_k=t$. Moreover, two sequences $\ui,\uj\in[n]^t$ correspond to the same multiset, i.e., $\{\ui\}=\{\uj\}$, precisely when $\alpha(\ui)=\alpha(\uj)$.

\smallskip
Any permutation $\sigma\in \Perm_t$ of the set $[t]$  acts on $[n]^t$ by setting
%\footnote{\ML{\textcolor{red}{The notation must be changed to $i^\sigma$, $v^\sigma$ in order not to confuse later with the action of permutation $\theta\in \Perm_n$, check throughout again..}}}
 $\ui^\sigma=i_{\sigma(1)}\ldots i_{\sigma(t)}$ for $\ui\in[n]^t$,  it also acts on $\Cnt$ by setting $v^\sigma= (v_{\sigma(\ui)})_{\ui\in [n]^t}$ for $v\in\Cnt$, and on $\Herm((\C^n)^{\ot t})$ by setting 
 $\rho^\sigma= (\rho_{\ui^\sigma,\uj^\sigma})_{\ui,\uj\in[n]^t}$ for $\rho\in \Herm((\C^n)^{\ot t})$.
Then, $v\in\Cnt$ is called a {\em symmetric tensor} if $v^\sigma=v$ for all $\sigma\in\Perm_t$.
Let $\MS^t(\K^n)$ denote the subspace of symmetric $t$-tensors and 
$\Pi_t: (\K^n)^{\ot t} \to \MS^t(\K^n)$ denotes the orthogonal projection onto the symmetric subspace $\MS^t(\K^n)$.

%\smallskip
A matrix  $\rho\in \Herm(\Cnt)$  is said to be {\em Bose symmetric}
(also called  {\em bosonic} in the quantum information literature) 
if $\rho^\sigma=\rho$ for all $\sigma\in\Perm_t$, i.e.,  $\rho\  e_{i_1}\ot \ldots\ot e_{i_t}= \rho \ e_{i_{\sigma(1)}}\ot\ldots \ot e_{i_{\sigma(t)}}$ for all $\sigma\in\Perm_t$ and $\ui\in [n]^t$. Equivalently, $\rho$ is Bose symmetric if $\rho \Pi_t= \rho$ (or $\Pi_t\rho=\rho$, or $\Pi_t\rho\Pi_t=\rho$, since $\rho$ is Hermitian).
%, $\Pi_t\rho=\rho$ or $\Pi_t\rho\Pi_t=\rho$. 
%In other words, $\rho\  e_{i_1}\ot \ldots\ot e_{i_t}= \rho \ e_{i_{\sigma(1)}}\ot\ldots \ot e_{i_{\sigma(t)}}$ for all $\sigma\in\Perm_t$ and $\ui\in [n]^t$. 
Let $\BS(\Cnt)\subseteq \Herm(\Cnt)$ denote the set of Hermitian Bose symmetric matrices. So,  any $\rho\in \BS_t(\Cnt)$ leaves $\MS^t(\C^n)$ invariant and vanishes on its orthogonal complement. 
This fact justifies the identifications $\BS(\Cnt) \simeq \Herm(S^t(\C^n))$ and
$(\BS(\Cnt)^\perp\simeq \Herm(  (S^t(\C^n))^\perp)$.
In the same way,
%\footnote{\ML{If we like this notation we can use it  to replace $\MV^{(t)}_n$ in the DPS hierarchy. I find it 'more easy to parse' than $\MV^{(t)}_n$, so I started already using it   }}, 
$\Herm(\C^n\ot S^t(\C^n))$ denotes the subspace of Hermitian matrices in $\Herm((\C^n)^{\ot (t+1)})$ that leave $\C^n\ot S^t(\C^n)$ invariant and vanish on its orthogonal complement $\C^n\ot (S^t(\C^n))^\perp$.

\smallskip
For a matrix acting on a tensor product space, one can define   the {\em partial trace} and the {\em partial transpose}   with respect to a subset of  registers. We first introduce these notions for the case of a tensor space $\C^{n_A}\ot \C^{n_B}$.  
Let $\rho_{AB}\in \Herm(\C^{n_A}\ot \C^{n_B})$. Then, {\em tracing out the B-register} produces the matrix $\Tr_B(\rho_{AB})\in\Herm(\C^{n_A})$, with entries
\begin{align}\label{eq:trace1}
\Tr_B(\rho_{AB})_{i,j}=\langle \Tr_B(\rho_{AB}), e_ie_{j}^T\rangle = 
\langle \rho_{AB}, e_ie_{j}^T \ot I_{n_B}\rangle = \sum_{k\in [n_B]} (\rho_{AB})_{ik,jk} 
%\ \text{ for } i,j\in [n_A].
\end{align}
for $i,j\in [n_A]$. {\em Taking the partial transpose with respect to the B-register} produces 
$\rho_{AB}^{T_B}\in\Herm(\C^{n_A}\ot \C^{n_B})$, with entries
\begin{align}\label{eq:transpose1}
(\rho_{AB}^{T_B})_{ih,jk}= (\rho_{AB})_{ik,jh}\ \text{ for } i,j\in [n_A], h,k\in [n_B].
\end{align}

\begin{example}\label{ex:Dicke}
For $t=2$, the vectors  $D_{ii}=e_i\ot e_i$ ( for $i\in [n]$) and $D_{ij}=(e_i\ot e_j+ e_j\ot e_i)/\sqrt 2$
(for $1\le i< j\le n$) form an orthonormal basis of $\MS^2(\C^n)$, known as the {\em Dicke basis}.
The  {\em Dicke states} $D_{ij}D_{ij}^*$ ($1\le i\le j\le n$) are clearly Bose symmetric. Moreover, they span the space $\BS(\Cnn)\cap \LDUI_n$, also known as the set of {\em mixed Dicke states} \cite{GNP_2025} (called {\em diagonally symmetric states} in \cite{Marconi_2021,Tura_2018}). This follows from (\ref{eq:LDUI-BS}) combined with the  fact that, for any $X\in\MS^n$, $\rho_{(X,\cdot,X)}$ can be decomposed as
%As an illustration, for any $X\in \MS^n$, the  state $\rho_{(X,\cdot,X)}$ belongs to $ \LDUI_n\cap \BS(\C^n\ot \C^n)$ (recall (\ref{eq:LDUI-BS})) and it can be decomposed as
$$\rho_{(X,\cdot,X)}= 
\sum_{i=1}^n X_{ii} D_{ii}D_{ii}^*+\sum_{1\le i<j\le n} X_{ij} D_{ij}D_{ij}^*.
$$
Each  state $D_{ii}D_{ii}^*= e_ie_i^*\ot e_ie_i^*$ is clearly separable. However, for $i<j$,  the state $D_{ij}D_{ij}^*$ is entangled, and it is known as the {\em maximally entangled state}. That the state $D_{ij}D_{ij}^*$ is entangled follows   from the fact that its partial transpose is not positive semidefinite, as shown below for $n=2$. Then, $D_{12}=(e_1\ot e_1+e_2\ot e_2)/\sqrt 2$,  
$$
D_{12}D_{12}^*={1\over 2} \begin{pmatrix} 0 & 0 & 0 & 0\cr 0 & 1 & 1 & 0\cr 0 & 1 & 1 & 0\cr 0 & 0 & 0 &0\end{pmatrix},\quad 
(D_{12}D_{12}^*)^{T_B}={1\over 2} \begin{pmatrix} 0 & 0 & 0 & 1\cr 0 & 1 & 0 & 0\cr 0 & 0 & 1 & 0\cr 1 & 0 & 0 &0\end{pmatrix}\not\succeq 0,
$$
where the matrices are indexed by $11, 12, 21, 22$ (in that order).
Note that $D_{12}D_{12}^*$ is LDUI  and its partial transpose $(D_{12}D_{12}^*)^{T_B}$ is CLDUI (as expected). Moreover, $D_{12}D_{12}^*$ is Bose symmetric, but its partial transpose $(D_{12}D_{12}^*)^{T_B}$ is not (since its columns indexed by 12 and 21 are distinct). 
Note also that $\Tr_B(D_{12}D_{12}^*)={1\over 2}I_2.$
\end{example}

%\begin{example}\label{example1}
%For $n=2$,  consider the  bipartite state $\rho_{AB}\in\Herm(\C^2\ot\C^2)$, its partial transpose and its partial trace with respect to the B-register: 
%$$
%\rho_{AB}=\left(\begin{matrix} 2 & 0 & 0 & 1/2\cr 0 & 2& 2 & 0\cr  0 & 2& 2 & 0\cr
%1/2 & 0 & 0 & 2\end{matrix}\right),\ 
%\rho_{AB}^{T_B}= \left(\begin{matrix} 2 & 0 & 0 & 2\cr 0 & 2& 1/2 & 0\cr  0 & 1/2& 2 & 0\cr
%2 & 0 & 0 & 2\end{matrix}\right), \ \Tr_B(\rho_{AB})=\left( 
%\begin{matrix} 4 & 0 \cr 0 & 4\end{matrix}\right).
%$$
%indexed by $11, 12, 21, 22$ (in that order). Note that  $\rho_{AB}$ is Bose symmetric (since its columns indexed by 12 and 21 coincide) and that $\rho_{AB}$ is LDOI (since its support is contained 
%in $\Omega_2\cup\Phi_2$, recall (\ref{eq:suppCLDUI})-(\ref{eq:suppLDOI})). \ML{Maybe this example can now be omitted?}
%\end{example}

The partial trace and partial transpose extend naturally to a tensor product space of the form $\Cnnt$ (with $t\ge 1$), as will be the setting for the DPS hierarchy in Section \ref{sec:DPS} below. Consider  $\rho_{AB[t]}\in \Herm (\Cnnt)$, where the notation is meant to stress   the role of the first A-register $\C^n$ combined with $t$ copies of the B-register $\C^n$.
% (recall $n=n_A=n_B$).
Then, one can trace out or take the partial trace with respect to any subset of the $t+1$ registers. In the application to the DPS hierarchy, we will do this for a subset of the $B$-registers. Assume 
we are given an integer $s\in [t]$ and we apply these operations w.r.t. the last $s$ B-registers; this  amounts to consider $\Cnnt$ as the bipartite system $\C^{n_A}\ot\C^{n_B}$ with $\C^{n_A}\sim \C^n\ot (\C^n)^{\ot (t-s)}$ and $\C^{n_B}\sim (\C^{n})^{\ot s}$, in which case relations (\ref{eq:trace1}) and (\ref{eq:transpose1}) specialize as follows. 
Then, $\Tr_{B[t-s+1:t]}(\rho_{AB[t]})\in \Herm(\C^n\ot (\C^n)^{\ot (t-s)})$ denotes the matrix obtained from $\rho_{AB[t]}$ by {\em tracing out the last $s$ B-registers}, thus with entries 
 \begin{align}\label{eq:partialtrace}
 (\Tr_{B[t-s+1:t]}(\rho_{AB[t]}))_{i_0\ui,j_0\uj}=\sum_{\uk\in [n]^{s}} (\rho_{AB[t]})_{i_0\ui \uk,j_0 \uj \uk}
 % \quad \text{  for } (i_0,\ui), (j_0,\uj)\in [n]\times [n]^{t-s}.
  \end{align}
for $(i_0,\ui), (j_0,\uj)\in [n]\times [n]^{t-s}.$
Similarly, $  \rho_{AB[t]}^{T_{B[t-s+1:t]}}\in \Herm(\Cnnt)$ is the matrix obtained from $\rho_{AB[t]}$ 
by taking the 
 {\em 
 partial transpose with respect to the last $s$ B-registers}, thus with entries
\begin{align}\label{eq:partialtranspose}
(\rho_{AB[t]}^{T_{B[t-s+1]}}) _{i_0\ui \uh, j_0\uj \uk} =( \rho_{AB[t]})_{i_0 \uj \uh,j_0\ui \uk}
 \text{ for } (i_0,\ui,\uh), (j_0,\uj,\uk)\in [n]\times [n]^{t-s}\times [n]^s.
% i_0,j_0\in [n],\ \ui,\uj\in [n]^s,\ \uh,\uk\in [n]^{t-s}.
\end{align}
%For an integer  $s\in [t]$, one can write sequences in $[n]^{1+t}$ as $(i_0,\ui, \uk)$ (or $i_0\ui\uk$) with $i_0\in [n]$, $\ui\in [n]^s$ and $\uk\in [n]^{t-s}$,  to reflect the splitting into $1+ s + (t-s)$ of the $t+1$ registers. 
%Then,  we let $\Tr_{B[s+1:t]}(\rho_{AB[t]})\in \Herm((\C^n\ot \C^n)^{\ot s})$ denote the matrix obtained from $\rho_{AB[t]}$ by {\em tracing out the last $t-s$ B-registers}, with entries 
% \begin{align}\label{eq:partialtrace}
% (\Tr_{B[s+1:t]}(\rho_{AB[t]}))_{i_0\ui,j_0\uj}=\sum_{\uk\in [n]^{t-s}} (\rho_{AB[t]})_{i_0\ui \uk,j_0 \uj \uk} \quad \text{  for } i_0,j_0\in [n],\ \ui,\uj\in [n]^s.
% \end{align}
%Moreover, $  \rho_{AB[t]}^{T_{B[s]}}\in \Herm(\Cnnt)$ is the matrix obtained from $\rho_{AB[t]}$  by taking the 
% {\em partial transpose with respect to the first $s$ B-registers}, with entries
%\begin{align}\label{eq:partialtranspose}
%(\rho_{AB[t]}^{T_{B[s]}}) _{i_0\ui \uh, j_0\uj \uk} =( \rho_{AB[t]})_{i_0 \uj \uh,j_0\ui \uk}
%\quad \text{ for }  i_0,j_0\in [n],\ \ui,\uj\in [n]^s,\ \uh,\uk\in [n]^{t-s}.
%\end{align}
Hence, $\Tr_B=\Tr_{B[1:1]}$ and $T_B=T_{B[1:1]}$ if $s=t=1$. For instance, if $\rho_{AB[t]}=xx^*\ot (yy^*)^{\ot t}$, then
% with $x,y\in\C^n$, then 
\begin{align}\label{eq:basicPT}
\begin{split}
\Tr_{B[t-s+1:t]}(\rho_{AB[t]})= xx^*\ot (yy^*)^{\ot (t-s)} \|y\|^{2s}, \\
\rho_{AB[t]}^{T_{B[t-s+1:t]}} = xx^*\ot (y y^*)^{\ot (t-s)}\ot (\ovy \ovy^*)^{\ot s}.
\end{split}\end{align}

%\subsubsection*{Notation on polynomials and the moment approach\footnote{\ML{First draft only}}}
\subsubsection*{Polynomials}

We will use polynomials in real variables and in complex variables, so we introduce both settings. Throughout, $\R[x]=\R[x_1,\ldots,x_n]$ denotes the polynomial ring in $n$ variables $x=(x_1,\ldots,x_n)$, with real coefficients. For an integer $d\ge 0$, $\R[x]_{d}$ denotes the space of homogeneous polynomials  of degree $d$, of the form $p=\sum_{\alpha\in\N^n, |\alpha|=d}p_\alpha x^\alpha$, and $\R[x]_{\le d}=\cup_{k=0}^d\R[x]_k$ is the set of polynomials with degree at most $d$. Here, $|\alpha|=\sum_{i=1}^n\alpha_i$ and $x^\alpha=\prod_{i=1}^n x_i^{\alpha_i}$ for $\alpha\in\N^n$.
A polynomial $p\in \R[x]$ is said to be a {\em sum of squares} ({\em sos}, for short) if $p=\sum_{\ell=1}^m q_\ell^2$ for some $q_\ell \in\R[x] $ and $m\ge 1$. Then, $p$ has even degree $2d$, $q_\ell$ has degree at most $d$, and each $q_\ell$ is homogeneous of degree $d$ if $p$ is homogeneous. 
%Let $\Sigma$ denote the cone of sums of squares and set $\Sigma_d=\Sigma\cap \R[x]_{2d}$. 
As is well-known,   %membership of $p$ in $\Sigma_d$ 
a polynomial $p$ of degree $2d$ is sos if and only if  the following   semidefinite program is feasible: there exists   a positive semidefinite matrix $Q\in \MS^N$ ($N={n+d\choose d}$) such that the polynomial identity $p=\langle Q, (x^{\alpha+\beta})_{\alpha,\beta\in\N^n, |\alpha|,|\beta|\le d}\rangle$ holds. 

%begin ignore
\ignore{
 Clearly, if $p\in \Sigma$, then $p$ is globally nonnegative. Restrictive converse statements hold, which are known as {\em Positivstellens\"atze} (most notably by  Putinar \cite{Putinar}, Schm\"udgen \cite{Schmudgen}, see e.g. \cite{Las09,Laurent09}). Particularly relevant to our treatment is the following Positivstellensatz.\footnote{\ML{Or state it for the complex sphere?}}

\begin{theorem}\cite{Schmudgen}
If $p\in \R[x]$ is strictly positive on the real sphere $\sS^{n-1}$, i.e., $p(x)>0$ for all $x\in \sS^{n-1}\cap \R^n$, then $p=\sigma+q (1-\sum_{i=1}^nx_i^2)$ for some sos $\sigma\in \Sigma$ and some $q\in \R[x]$.
\end{theorem}

The dual space of $\R[x]$ is $\R[x]^*$ that consists of the linear functionals $L:\R[x]\to\R$. 
A typical instance of such $L$ arises by considering integration w.r.t. a measure $\mu$ on a closed set $S$, given by
$$L(x^\alpha)=\int_S x^\alpha d\mu(x)\ \text{ for all }\alpha\in \N^n.$$ 
The quantities $\int_Sx^\alpha d\mu(x)$ are known as the moments of the measure $\mu$. Deciding whether a linear functional $L$ admits such a representing measure is the classical moment problem (see, e.g., \cite{}). This problem is well solved for compact sets (see \cite{Putinar,Schmudgen}). For the case of the sphere, the result reads: 
\begin{theorem}\cite{Schmudgen}
Let $L\in\R[x]^*$. Then, $L$ admits a representing measure supported on the real sphere $\sS^{n-1}$ if and only if $L \ge 0$ on $\Sigma$ and $L=0$ on the ideal $\{q(1-\|x\|^2):q\in\R[x]\}$.
\end{theorem}
}
% end ignore

\smallskip We will also consider the ring $\C[x,\ovx]$ of polynomials in $n$ complex variables $x=(x_1,\ldots,x_n) $ and their conjugates
$\ovx=(\ovx_1,\ldots,\ovx_n)$. 
So, $p\in \C[x,\ovx]$ is of the form $p=\sum_{\alpha,\alpha'\in \N^n} p_{\alpha,\alpha'}x^\alpha \ovx ^{\alpha'}$ (with finitely many nonzero coefficients). The largest value of $|\alpha|+|\alpha'|$ such that $p_{\alpha,\alpha'}\ne 0$ is the {\em total degree} of $p$.
One can also consider the degree in $x$ and the degree in $\ovx$. For integers $d,d'\ge 0$,  
$\C[x,\ovx]_{d,d'}$ consists of the homogeneous polynomials $p$ with degree   $d$ in $x$ and  $d'$ in $\ovx$, i..e, with $p_{\alpha,\alpha'}=0$ if $(|\alpha|,|\alpha'|)\ne (d,d')$. 
A polynomial $p\in \C[x,\ovx]$ is said to be {\em Hermitian} if $p=\overline p$ holds. For instance, $p=x+\ovx$ is Hermitian, but $q=\bfi x+\ovx$ is not Hermitian since $\overline q= -\bfi\ovx+x\ne q$.

A (Hermitian) polynomial $p\in \C[x,\ovx]$ is said to be a {\em real-sum of squares} (abbreviated as {\em r-sos}) if 
$p=\sum_{\ell=1}^m q_\ell^2$ for some Hermitian polynomials $q_\ell\in \C[x,\ovx]$. Equivalently, $p$ is r-sos if and only  if $p=\sum_\ell h_\ell \overline{h_\ell}$ for some polynomials $h_\ell\in\C[x,\ovx]$.
For instance, $p=x\ovx=\big({x+\ovx\over 2}\big)^2+({x-\ovx\over 2\bfi})^2$ is r-sos.
Here, we stress `real'-sos in order to distinguish from the other (more restricted) notion of {\em complex-sum of squares} ({\em c-sos}), where $p\in \C[x,\ovx]$ is c-sos if $p=\sum_\ell h_\ell \overline {h_\ell}$ for some $h_\ell\in \C[x]$.
Again, one can check whether a polynomial $p\in \C[x,\ovx]$ is r-sos by means of a semidefinite program.

Any polynomial $p(x,\ovx)\in \C[x,\ovx]$ can be transformed into a pair of polynomials in $\R[x_\Rea,x_\Ima]$ in the real variables $x_\Rea, x_\Ima\in \R^n$ via the transformation
$x=x_\Rea+ \bfi x_\Ima$. %, where $\bfi$ is the complex root of $-1$.
In this way, each $p\in\C[x,\ovx]$ corresponds to a unique pair $(p_\Rea,p_\Ima)\in \R[x_\Rea,x_\Ima]^2$  via 
\begin{align*}
p_\Rea(x_\Rea,x_\Ima)= \Rea (p(x_\Rea +\bfi x_\Ima, x_\Rea -\bfi x_\Ima)), \\
p_\Ima(x_\Rea,x_\Ima)= \Ima (p(x_\Rea +\bfi x_\Ima, x_\Rea -\bfi x_\Ima),
\end{align*}
and we have  the identity
\begin{align}\label{eq:pCR}
p(x,\ovx)= p(x_\Rea +\bfi x_\Ima, x_\Rea -\bfi x_\Ima)
= 
p_\Rea(x_\Rea,x_\Ima) + \bfi p_\Ima(x_\Rea,x_\Ima).
\end{align}
So, $p_\Rea=\Rea(p)$ is the real part of $p$ and $p_\Ima=\Ima(p)$ is its imaginary part.
Then, $p$ is Hermitian (i.e., $p=\overline p$) precisely when $\Ima(p)=0$.
In particular, the polynomial $p \overline p$ is Hermitian and equality
\begin{align}\label{eq:rsos-sos}
\Rea(p \overline p) = p_{\Rea}^2 + p_{\Ima}^2
\end{align}
 holds. Based on this one can   verify the following well-known correspondence: for a Hermitian polynomial  
$p\in\C[x,\ovx]$,
 \begin{align}\label{eq:rsos-sos}
% \text{ A Hermitian polynomial } p\in\C[x,\ovx]  
 p \text{ is r-sos }\Longleftrightarrow \text{ its real part } p_\Rea
 \in\R[x_\Rea,x_\Ima]  \text{ is sos}
 \end{align}
(which explains the terminology `real'-sos). See, e.g., \cite[Appendix A]{GLS_2021} for a detailed exposition.

\smallskip
For the application to the DPS hierarchy, one needs a second group $y=(y_1,\ldots,y_n)$ of $n$ complex variables  and their conjugates 
$\ovy=(\ovy_1,\ldots,\ovy_n)$. So, we will use the ring  $\C[x,\ovx,y,\ovy]$ of polynomials in these variables, where a polynomial 
$p\in \C[x,\ovx,y,\ovy]$ is of the form
$p=\sum_{\alpha, \alpha',\beta,\beta'\in\N^n} p_{\alpha,\alpha',\beta,\beta'}x^\alpha \ovx^{\alpha'}y^{\beta}\ovy^{\beta'}$ (with finitely many nonzero coefficients) and one can consider the separate degrees in each of $x,$ $\ovx,$ $y$ and $\ovy$.  
%The largest value of $|\alpha|+|\alpha'|+|\beta|+|\beta'|$ for which $p_{\alpha,\alpha',\beta,\beta'}\ne 0$ is the {\em total degree} of $p$. Similarly, one can consider the degree in each of the variables $x,\ovx,y$ or $\ovy$. 
%So, one can grade the polynomial ring as $\C[x,\ovx,y,\ovy]_{r,r',s,s'}$, consisting of the polynomials $p$ where variable $x$ occurs only with degree $r$, $\ovx$ with degree $r'$, $y$ with degree $s$, and $\ovy$ with degree $s'$; that is, $p_{\alpha,\alpha',\beta,\beta'}\ne 0$ only if $(|\alpha|,|\alpha'|,|\beta|,|\beta'|)=(r,r',s,s')$. A polynomial $p\in   \C[x,\ovx,y,\ovy]$ is said to be {\em Hermitian} if $p=\overline p$ holds, let $\C[x,\ovx,y,\ovy]^h$ denote the set of Hermitian polynomials.
For instance, for a matrix $M\in \Herm(\Cnn)$, the polynomial $F_M=\langle M, (x\ot y)(x\ot y)^*\rangle$ (as in (\ref{eq:FM}) below) is Hermitian and 
%\begin{align}\label{eqFM}
% F_M=(x\ot y)^* M x\ot y= \langle M, (x\ot y)(x\ot y)^*\rangle
% \end{align}
 belongs to $ \C[x,\ovx,y,\ovy]_{1,1,1,1}$.

\ignore{
A Hermitian polynomial $p$ is said to be a 
{\em real-sum of squares} (r-sos, for short) if $p=\sum_{\ell=1}^m q_\ell  \overline {q_\ell}$ for some $q_\ell \in  \C[x,\ovx,y,\ovy]$  or, equivalently, if $p=\sum_{\ell=1}q_\ell^2$ for some Hermitian polynomials $q_\ell$ and $m\ge 1$ (see Lemma \ref{lem:rsos-sos}). Here, one stresses  `real'-sos in order to distinguish from the other notion of {\em complex sum of squares} (c-sos) where one would restrict to $q_\ell\in \C[x,y]$. Being r-sos is a much stronger notion than c-sos (and it corresponds to requiring that both the real and imaginary parts are sums of squares of real polynomials).
See Appendix \ref{sec:proof-lemma} for  more details.

\medskip
$L\in  \C[x,\ovx,y,\ovy]^*$ is nonnegative \ML{To do later - see what is needed.}
 }
 
\subsection{The DPS hierarchy for bipartite separable states}\label{sec:DPS}
%{\color{red}introduction via state extension}\cite{DPS}

We  define  the hierarchy of semidefinite relaxations, denoted\footnote{Here, we   follow the notation   used in \cite{FF-sphere}, also for extended states $\rho_{AB[t]}$ and the partial transpose $T_{B[s]}$.
In the quantum information literature, the DPS relaxation is also denoted by (a variation of)
\smash{ $\text{\rm PPTBExt}^{(t)}_n$}; see, e.g., \cite{GNP_2025}, where the authors refer to the notion of {\em star-extendibility}, connecting to \cite{ACGNO_2023}.}
as $\DPSt_n$,  that were introduced by Doherty, Parrilo and Spedaglieri \cite{DPS_2002,DPS_2004} to approximate the cone $\SEP_n$ of separable states (as in (\ref{eq:SEP})). 
Consider a bipartite state $\rho_{AB}\in \Herm(\Cnn)$. Assume that $\rho_{AB}$ is separable, i.e., there exist vectors $x_\ell,y_\ell\in \sS^{n-1}$ and scalars $\lambda_\ell>0$, satisfying 
%$\sum_{\ell=1}^m \lambda_\ell=1$ and 
$$\rho_{AB}=\sum_{\ell=1}^m \lambda_\ell\ x_\ell x_\ell^*\ot y_\ell y_\ell^*.$$
For any integer $t\ge 1$, one can define the following ``extended state" 
$$\rho_{AB[t]}= \sum_{\ell=1}^m \lambda_\ell\ x_\ell x_\ell^*\ot (y_\ell y_\ell^*)^{\ot t},%\in\Herm(\Cnnt),
$$
obtained by replacing each factor $y_\ell y_\ell^*$ by its $t$-th tensor product $(y_\ell y_\ell^*)^{\ot t}$.
Then, $\rho_{AB[t]}$ belongs to   $\Herm(\Cnnt)$ and it satisfies the following conditions.
First, it is invariant under permuting the  $t$ B-registers:
\begin{align}\label{eq:DPS1}
\begin{split}
& (I_n\ot \Pi_t) \ \rho_{AB[t]} \ (I_n\ot \Pi_t) = \rho_{AB[t]}; \text{  equivalently}, \  \rho_{AB[t]} \ (I_n\ot \Pi_t) = \rho_{AB[t]}, \\
&\text{i.e., } (\rho_{AB[t]})_{i_0\ui,j_0\uj}= (\rho_{AB[t]})_{i_0 \ui,j_0\uj^\sigma} \text{ for all } \sigma \in \Perm_t,\ i_0,j_0\in [n],\ \ui,\uj\in [n]^t;\\
& \text{in other words, } \rho_{AB[t]}\in \Herm(\C^n\ot S^t(\C^n)).
\end{split}\end{align}
%(where the equivalence follows since $\rho_{AB[t]}$ is Hermitian). 
Second, taking the partial transpose of any subset of the $t$ B-registers preserves positive semidefiniteness (recall (\ref{eq:basicPT})). In view of the permutation invariance property  (\ref{eq:DPS1}), one may restrict to taking the partial transpose with respect to the first $s$ B-registers, denoted as $T_{B[s]}$, leading to
\begin{align}\label{eq:DPS2}
\rho_{AB[t]}^{T_{B[s]}}\succeq 0  \ \text{ for all } s=0,1,\ldots,t.
\end{align}
%for all $s=0,1,\ldots,t$. 
%known as the {\em Positive Partial Transpose} property (PPT) \cite{}.
If $s=0$, then  no partial transpose is taken; so, (\ref{eq:DPS2}) for $s=0$ amounts to asking $\rho_{AB[t]}\succeq 0$.
Finally, if we trace out the last $t-1$  B-registers, then we recover the initial state $\rho_{AB}$:
\begin{align}\label{eq:DPS3}
\Tr_{B[2:t]}(\rho_{AB[t]}) = \rho_{AB}.
\end{align}
 %Following \cite{DPS_2002,DPS_2004} we  n
 Now, define the relaxation $\DPSt_n$ of order $t$ as
 \begin{align}\label{eq:DPS}
 \begin{split}
 \DPSt_n=\{\rho_{AB}\in \Herm(\Cnn):  &\  \exists \rho_{AB[t]}\in \Herm (\Cnnt) \\
 &  \text{  satisfying }  (\ref{eq:DPS1}),  (\ref{eq:DPS2}), (\ref{eq:DPS3})\}.
  \end{split}
  \end{align}
Any $\rho_{AB[t]}\in\Herm(\Cnnt)$  satisfying  (\ref{eq:DPS1})-(\ref{eq:DPS3}) certifies membership of $\rho_{AB}$ in the set $\DPSt_n$, we call it  a {\em $\DPSt$-certificate} for $\rho_{AB}$. 

For $t=1$, no `state extension' takes place and $\rho_{AB}\in\DPS^{(1)}_n$ if and only if $\rho_{AB}\succeq 0$ and $\rho_{AB}^{T_B}\succeq 0$; the latter condition is known as the {\em PPT (positive partial  transpose)  condition} (see \cite{Horodecki_1996,Horodecki_1997}).
From the above discussion,  $\SEP_n\subseteq \DPSt_n$ for all $t\ge 1$, and
$\DPS^{(t+1)}_n\subseteq \DPSt_n$ (since tracing out the last register in a $\DPS^{(t+1)}$-certificate  produces a $\DPSt$-certificate for $\rho_{AB}$). In \cite{DPS_2002,DPS_2004} it is shown that the hierarchy is {\em complete}; that is, 
\begin{align}\label{eq:DPScomplete}
\bigcap_{t\ge 1}\DPSt_n=\SEP_n.
\end{align}
 An alternative proof of completeness,  based on the  moment approach in polynomial optimization, can be found in \cite{GLS_2021}. In Section \ref{sec:implementation},  we will recall how this alternative approach also leads to a more economical  description of the DPS hierarchy.
Also the weaker hierarchy that requires (\ref{eq:DPS2}) only for $s=0$ (i.e., $\rho_{AB[t]}\succeq 0$)  is complete  (see \cite{CKMR_2007}, with  a proof   based on  quantum de Finetti theorems).  

Observe that if $\rho_{AB[t]}$ is a $\DPSt$-certificate for $\rho_{AB}$, then its complex conjugate $\overline{\rho_{AB[t]}}$ is a $\DPSt$-certificate for the conjugate $\overline{\rho_{AB}}$ of $\rho_{AB}$. The next lemma follows then as an easy application.
%Observe that if $\rho_{AB}$ belongs to $\DPSt_n$ with $\DPSt$-certificate $\rho_{AB[t]}$, then also its conjugate $\overline{{\rho}_{AB}}$ belongs to $\DPSt_n$ with  $\DPSt$-certificate $\overline{{\rho}_{AB[t]}}$. From this follows the next lemma.
%Denote with $\Re(\rho_{AB})=\frac{\rho_{AB}+\overline{\rho}_{AB}}{2}$. Because of the conic structure of $\DPSt_n$ we get the following lemma.

\begin{lemma}\label{lem:real-certificate}
Assume that  $\rho_{AB}$ belongs to $\DPSt_n$  with $\DPSt_n$-certificate $\rho_{AB[t]}$. Then, its real part   $\Re(\rho_{AB})= {1\over 2}( \rho_{AB}+ \overline{\rho_{AB}})$ belongs to $\DPSt_n$ with real-valued $\DPSt_n$-certificate $\Re(\rho_{AB[t]})$.
\end{lemma}
\noindent
Note that the analogous result does {\em not} hold for $\SEP_n$: while a real-valued separable state $\rho_{AB}$ admits a real-valued $\DPSt$-certificate for every $t$, it can be that any separable decomposition $\rho_{AB}=\sum x_\ell x_\ell^*\ot y_{\ell}y_{\ell}^*$  requires some complex atoms $x_\ell ,y_\ell \in\C^n\setminus \R^n$. We will see such an example  in  Section \ref{sec:exampleab} (see Remark \ref{rem:exampleab}).

\medskip
 For Bose symmetric separable  states, one can naturally  strengthen  the DPS hierarchy. We begin with giving a short argument for the description (\ref{eq:SEP-BS}) of the set $\SEPBS_n$.
  
 \begin{lemma}\label{lem:SEP-BS}
The set of Bose symmetric separable bipartite states is given by
$$\SEPBS_n:=\SEP_n\cap\BS(\C^n\ot \C^n)=\cone \{xx^*\ot xx^*: x\in\sS^{n-1}\}.$$
\end{lemma}

\begin{proof}
Assume $\rho_{AB}\in \SEP_n\cap \BS(\Cnn)$. Say,
$\rho_{AB}=\sum_{\ell=1}^m\lambda_\ell \ x_\ell x_\ell^*\ot y_\ell y_{\ell}^*$ for some $\lambda_\ell>0$ and $x_\ell,y_\ell\in \sS^{n-1}$. 
As $\rho_{AB}$ is Bose symmetric, the vector $e_i\ot e_j-e_j\ot e_i$ belongs to the kernel of $\rho_{AB}$ for all $i\ne j\in [n]$. Hence, this vector also belongs to the kernel of $x_\ell x_\ell^*\ot y_\ell y_{\ell}^*$ for each $\ell$. This implies that 
$(x_\ell)_i(y_\ell)_j=(x_\ell)_j (y_{\ell})_i$ for all $i\ne j\in [n]$, and thus $x_\ell=\pm y_\ell$ (as both are unit vectors).
\qed\end{proof}

\noindent
Hence,   any separable Bose symmetric state $\rho_{AB}=\sum_\ell \lambda_\ell (x_\ell x_\ell^*)^{\ot 2}
$ admits a $\DPSt$-certificate 
$\rho_{AB[t]}= \sum_\ell \lambda_\ell (x_\ell x_\ell^*  )^{\ot (t+1)} $ satisfying a stronger symmetry property, namely, invariance under permuting {\em all} $t+1$ registers, including the A-register, instead of just the $t$ B-registers as  in relation (\ref{eq:DPS1}); that is, one may assume  
$\rho_{AB[t]}\in  \BS((\C^n)^{\ot (t+1)})\simeq \Herm(S^{t+1}(\C^n))$. 
 Note that for any $\rho_{AB[t]}\in \BS((\C^n)^{t+1})$, using the fact that the transpose of a Hermitian positive semidefinite matrix is positive semidefinite,
 condition  (\ref{eq:DPS2}) is equivalent to
 \begin{align}\label{eq:DPS2-BS}
 \rho_{AB[t]}^{T_{B[s]}}\succeq 0 \quad \text{ for } s=0, 1,\ldots,  \lfloor (t+1)/2\rfloor.
 \end{align}
This leads to the (symmetric)  variation\footnote{This relaxation corresponds to the set 
\smash{ $\text{\rm PPTBExt}^{K_{t+1}}_n$} from \cite[Section 8.2]{GNP_2025}, consisting of  the
{\em $K_{t+1}$-PPT bosonic extendible states}, corresponding to {\em complete graph extendibility} in \cite{ACGNO_2023}.}
 of the set $\DPSt_n$, which we denote $\tilDPSt_n$, defined by
\begin{align}\label{eq:DPSt-BS}
\begin{split}
\tilDPSt_n=\{\rho_{AB}\in\Herm(\Cnn):\ & \exists \rho_{AB[t]}\in \BS((\C^n)^{\ot (t+1)}) \text{ satisfying } (\ref{eq:DPS2-BS})   \\
% \exists \rho_{AB[t]}\in\Herm(\Cnnt) \text{ satisfying } %(\ref{eq:DPS1-BS}), (\ref{eq:DPS2-BS}) \\
& \text{ and }   \rho_{AB}=\Tr_{B[2:t]}(\rho_{AB[t]})\}.
\end{split}
\end{align}
For any order  $t\ge 1$, we have
%\footnote{\ML{For order $t=1$, the inclusion $\tilDPSt_n\subseteq \DPSt_n\cap \BS(\Cnn)$ is an equality. However, for $t\ge 2$, it is not clear whether equality holds. Probably not ??.}}
\begin{align}\label{eq:DPS-BS-inclusion}
\SEPBS_n\subseteq \tilDPS^{(t+1)}_n\subseteq \tilDPSt_n\subseteq \BS(\Cnn),\
\tilDPSt_n\subseteq \DPSt_n\cap \BS(\Cnn).
\end{align}
 The right-most  inclusion is an equality for $t=1$  (since no state extension takes place); it remains an open question to show that the inclusion is strict for $t\ge 2$ and $ n\ge 5$. 
%\textcolor{red}{in Section \ref{sec:implementation} we will give an example showing  strict inclusion for $t=2$ and $n=5$.}
Completeness of the DPS hierarchy for $\SEP_n$ (recall (\ref{eq:DPScomplete})) implies completeness of the $\tilDPS$ hierarchy for $\SEPBS_n$:
$$\bigcap_{t\ge 1}\tilDPSt_n =\SEPBS_n.$$  
 Next, we turn to the description of the dual objects for $\CP_n$, $\SEP_n$ and $\DPSt_n$, and we will consider the dual of $\tilDPSt_n$ in Section \ref{sec:DPS-symmetric}.

\subsection{\smash{Dual cones for $\CP_n$, $\SEP_n$ and $\DPSt_n$}}\label{sec:intro-dualobjects}
%\subsection[Dual cones for $\mathsmaller{\CP_n}$, $\mathsmaller{\SEP_n}$ and $\mathsmaller{\DPSt_n}$]{Dual cones for $\CP_n$, $\SEP_n$ and $\DPSt_n$}\label{sec:intro-dualobjects}

The dual cone of the completely positive cone $\CP_n$ is the copositive cone $\COP_n$,   defined as
\begin{align}\label{eq:COP}
\begin{split}
\COP_n& =\{A\in \MS^n: x^T Ax\ge 0 \text{ for all } x\in \R^n_+\}\\
&=\{A\in\MS^n: (x^{\circ 2})^TAx^{\circ 2}\ge 0 \text{ for all } x\in\R^n\},
\end{split}
\end{align}
where   $x^{\circ 2}=x\circ x= (x_i^2)_{i=1}^n$ for $x\in\R^n$. So, any matrix $A\in\MS^n$ defines a  homogeneous $n$-variate polynomial with degree 4, 
\begin{align}\label{eq:fA}
f_A=(x^{\circ 2})^TAx^{\circ 2}=\sum_{i,j=1}^nx_i^2x_j^2 A_{ij}\in\R[x]_4,
\end{align}
and this polynomial is globally nonnegative precisely when $A$ is a copositive matrix. A sufficient condition for $f_A$ to be globally nonnegative is when 
the polynomial $\|x\|^{2t}f_A$ is sos for some integer $t\ge 0$. Motivated by this, 
Parrilo \cite{Parrilo_2000} and de Klerk and Pasechnik \cite{dKP_2002} define the cones
\begin{align}\label{eq:KtCOP}
\MKt_n=\big\{A\in\MS^n: \|x\|^{2t}f_A=\|x\|^{2t}\cdot  (x^{\circ 2})^TA x^{\circ 2} \text{ is sos}\big\}.
\end{align}
These cones form a hierarchy of semidefinite inner approximations  for $\COP_n$ and  satisfy 
$$\MK^{(t)}_n\subseteq \MK^{(t+1)}_n\subseteq \COP_n, \quad 
\text{\rm int}(\COP_n)\subseteq \bigcup_{t\ge 0} \MKt_n\subseteq \COP_n,$$
where the inclusion \smash{$\text{\rm int}(\COP_n)\subseteq \bigcup_{t\ge 0} \MKt_n$} follows from  Reznick \cite{Reznick}. As a consequence, the dual cones \smash{$(\MKt_n)^*$} form a complete hierarchy of semidefinite outer approximations for   $\CP_n$, i.e.,  
\begin{align}\label{eq:CPcomplete}
\CP_n\subseteq  (\MK^{(t+1)}_n)^*\subseteq (\MKt_n)^*,\quad \CP_n=\bigcap_{t\ge 0}  (\MKt_n)^*.
\end{align}
 Here are some known  facts about the cones $\MKt_n$ and their duals, as background information. For $t=0$, we have \smash{$(\MK^{(0)}_n)^*= \MS^n_+\cap \R^{n\times n}_+$}. Hence, \smash{$(\MK^{(0)}_n)^*$}   consists of the matrices that are positive semidefinite and entrywise nonnegative, known as {\em doubly nonnegative} matrices. Equivalently, $\MK^{(0)}_n$ consists 
of the matrices that can be written as sum of a positive semidefinite matrix and an entrywise nonnegative matrix (see \cite{Parrilo_2000,dKP_2002}). For $n\le 4$,  equality  \smash{$\CP_n=(\MK^{(0)}_n)^*$} holds, or, equivalently, \smash{$\COP_n=\MK^{(0)}_n$}   \cite{Diananda}.  For $n\ge 5$, %it is known that 
the inclusion $\MKt_n\subseteq\COP_n$ is strict for any $t\ge 0$ \cite{DDGH_2013}. Moreover, equality 
\smash{$\bigcup_{t\ge 0}\MKt_5= \COP_5$} holds \cite{Schweighofer-Vargas_2024} and the inclusion 
\smash{$\bigcup_{t\ge 0}\MKt_n\subseteq \COP_n$} is strict for     $n\ge 6$ 
\cite{Laurent-Vargas_2023}.

\medskip
We now turn to the dual cone of the separable cone $\SEP_n$. For a matrix $M\in \Herm(\Cnn)$, consider the associated polynomial $F_M$ in $2n$ complex variables $x=(x_1,\ldots,x_n)$, $y=(y_1,\ldots,y_n)$ and their conjugates $\ovx=(\ovx_1,\ldots,\ovx_n)$, $\ovy=(\ovy_1,\ldots,\ovy_n)$, defined as
 \begin{align}\label{eq:FM}
 F_M= %(x\ot y)^* M (x\ot y)=
  \langle M, (x\ot y)(x\ot y)^*\rangle=\langle M, xx^*\ot yy^*\rangle \in \C[x,\ovx,y,\ovy]_{1,1,1,1}.
 \end{align}
Then, the  dual of $\SEP_n$ is defined as
 \begin{align}\label{eq:SEPdual}
 \SEP_n^*=\{M\in\Herm(\Cnn): F_M= (x\ot y)^* M (x\ot y) \ge 0 \  \forall x,y\in\C^n\}.
 \end{align}
Note indeed that, since $F_M$ is homogeneous (of degree 1) in each of $x,\ovx,y,\ovy$, $F_M$ is globally nonnegative if and only if $F_M$ is nonnegative on the bi-sphere $\sS^{n-1}\times \sS^{n-1}$.  
 
 Any nonzero matrix $M\in \Herm(\Cnn)$ provides a hyperplane in the space $\Herm(\Cnn)$, with  normal   $M$. 
 The corresponding half-space, consisting of the states $\rho_{AB}\in\Herm(\Cnn)$ satisfying 
  $\langle M, \rho_{AB} \rangle \ge 0$, contains   $\SEP_n$ precisely when $M\in\SEP_n^*$. 
For this reason,  $M$ is also called an {\em entanglement witness}   since it separates some entangled states, i.e., 
$\langle M,\rho_{AB}\rangle <0$ for some $\rho_{AB}\in \Herm(\Cnn)\setminus\SEP_n$.

  A sufficient condition for $F_M$ to be globally nonnegative is when the polynomial
  $\|y\|^{2t} F_M$ is r-sos for some integer $t\ge 0$. Fang and Fawzi \cite{FF-sphere} characterize the matrices $M$ satisfying this property for a given order $t$ and show that this corresponds precisely to the dual of the DPS hierarchy. To state their result,  
  %consider the set 
 % \begin{align}\label{eq:Vnt}
 % \MV^{(t)}_n
% \ML{ \Herm(\C^n\ot S^t(\C^n))}=\{\rho_{AB[t]}\in \Herm(\Cnnt): \rho_{AB[t]} \text{ satisfies condition } (\ref{eq:DPS1})\},
 %   \end{align}
let $\MW^{(t)}_{n,s}$ denote the set of matrices satisfying   condition (\ref{eq:DPS2}) for given $s$, i.e.,
    \begin{align}\label{eq:Wst}
  \MW^{(t)}_{n,s}=\{\rho_{AB[t]}\in \Herm(\Cnnt): \rho_{AB[t]}^{T_{B[s]}}\succeq 0\}\quad \text{ for } s=0,1,\ldots,t.
  \end{align}
    
  \begin{theorem}\label{theo:FF}\cite{FF-sphere}
Consider  a matrix $M\in\Herm(\Cnn)$ and an integer $t\ge 1$. The following assertions are equivalent.
  \begin{itemize}
  \item[(i)] $M\in (\DPSt_n)^*$.
  \item[(ii)] The polynomial $\|y\|^{2(t-1)} F_M\in\C[x,\ovx,y,\ovy]_{1,1,2t,2t}$ is r-sos.
  \item[(iii)] There exist matrices $B,W_0,W_1,\ldots,W_t\in\Herm(\Cnnt)$ such that 
  \begin{align*}
M\ot I_n^{\ot (t-1)}= B+\sum_{s=0}^t W_s
%, & \text{ with  }  B\in \Herm(\C^n\ot S^t(\C^n))^\perp, W_s\in \MW^{(t)}_{n,s} \text{ for }s=0,1,\ldots,t.
\end{align*}
with $B\in \Herm(\C^n\ot S^t(\C^n))^\perp$ and $W_s\in \MW^{(t)}_{n,s}$  for  $s=0,1,\ldots,t$.
  \end{itemize}
  \end{theorem}
  
  Note the analogy between the description of the approximation hierarchy $\MKt_n$  for the copositive cone $\COP_n$ and the approximation hierarchy $(\DPSt_n)^*$ for the dual separable cone $\SEP_n^*$; in particular, compare relation (\ref{eq:KtCOP}) and Theorem \ref{theo:FF}(ii).
 As we will see in Section \ref{sec:DPS-symmetric}, there is a formal relationship between these objects, which can be shown when restricting to Bose symmetric   states.

\subsection{Diagonal unitary invariant states}\label{sec:CLDUI}

We   return to the class of CLDUI states. Recall the two equivalent characterizations through the invariance property  \eqref{eq:CLDUI} and through the sparsity pattern   \eqref{eq:suppCLDUI}. Following \cite{JML-PCP}, we now consider a third characterization via the projection\footnote{This map is also known as the CLDUI-twirling, see, e.g., \cite{GNP_2025}.}
 $\Pi_{\CLDUI_n}: \Herm(\C^n\ot\C^n)\to\CLDUI_n$ given by 
\begin{align}\label{eq:PiCLDUI}
    \Pi_{\CLDUI_n}(\rho_{AB})=\int_{\MDU_n}(U\ot\overline{U})\rho_{AB}(U\ot\overline{U})^*dU,
\end{align}
where we integrate with regards to the Haar measure. Since this integration can be seen as forming the group average of the operation in \eqref{eq:CLDUI}, this map is indeed the projection onto the $\CLDUI_n$ subspace. Now, clearly we have
\begin{align*}
    \rho_{AB}\in\CLDUI_n\Longleftrightarrow\Pi_{\CLDUI_n}(\rho_{AB})=\rho_{AB}.
\end{align*}
Two easy but important properties of this projection, collected in the next lemma, are that  it preserves positive semidefiniteness and separability.

\begin{lemma}\label{lem:projection-psd-sep}
If $\rho_{AB}$ is positive semidefinite (resp., separable), then $\Pi_{\CLDUI}(\rho_{AB})$ is positive semidefinite (resp., separable).
\end{lemma}

\begin{proof}
 If $\rho_{AB}\succeq 0$, then $\Pi_{\CLDUI}(\rho_{AB})\succeq 0$ because 
 %The %fact that positive semidefiniteness is preserved
%first property  follows from the fact that 
it  is block-diagonal with its diagonal blocks being equal to certain selected principal submatrices of $\rho_{AB}$ 
% the CLDUI sparsity pattern is block diagonal
 (see relation \eqref{eq:rhoXYZ-block}). 
Given $x,y\in\C^n$, we now show that  $\Pi_{\CLDUI}(xx^*\ot yy^*)$ is separable, using  (\ref{eq:PiCLDUI}).
%the . For this, we use the definition of $\Pi_{\CLDUI}$ from  (\ref{eq:PiCLDUI}). 
 For any $U\in \MDU_n$, we have  
 $$(U\ot\overline{U}) (xx^*\ot yy^*) (U\ot\overline{U})^*=Ux(Ux)^*\ot\overline{U}y(\overline{U}y)^*\in\SEP_n.$$ 
  % Now suppose the projection is applied to $xx^*\ot yy^*, x,y\in\C^n$. Then 
%  Hence, the integrand in (\ref{eq:PiCLDUI}) is separable. 
  %for every $U\in\MDU_n$ because 
% $$(U\ot\overline{U})xx^*\ot yy^*(U\ot\overline{U})^*=Ux(Ux)^*\ot\overline{U}y(\overline{U}y)^*\in\SEP_n.$$ 
Since integration can be seen as a limit of finite sums and  $\SEP_n$ is a closed cone, it follows that $\Pi_{\CLDUI}\left(xx^*\otimes yy^*\right) \in \SEP_n$.
\qed\end{proof}

The projection $\Pi_{\CLDUI}$ also allows to connect two different ways of characterizing the cone of pairwise completely positive matrices, as  done in \cite{JML-PCP}.

\begin{theorem}\label{theo:PCPDec}\cite{JML-PCP}
    Let $(X,Y)\in\R^{n\times n}\times \Herm^n$ and assume that  $\diag(X)=\diag(Y)$. Then, $(X,Y)\in\PCP_n$ (i.e., the associated CLDUI matrix $\rho_{(X,Y)}$ belongs to $\SEP_n$) if and only if 
    \begin{align}\label{eq:PCPDec}
        (X,Y)\in\cone\{((x\circ \overline{x})(y\circ \overline{y})^*,(x\circ y)(x\circ y)^*):x,y\in\C^n\}.
    \end{align}
\end{theorem}

\begin{proof} The underlying key fact is that, for any $x,y\in\C^n$ and  $i,j\in [n]$,  we have
\begin{align}\label{eq:relrhoxy}
\begin{split}
(xx^*\ot yy^*)_{ij,ij}= x_i\ovx_i y_j\ovy_j=((x\circ \ovx)(y\circ \ovy)^*)_{i,j},\\
(xx^*\ot yy^*)_{ii,jj}=x_i\ovx_jy_i\ovy_j= ((x\circ y) (x\circ y)^*)_{ii,jj}.
%\rho_{ij,ij}= (xx^*\ot yy^*)_{ij,ij}= x_i\ovx_i y_j\ovy_j=((x\circ \ovx)(y\circ \ovy)^*)_{i,j},\\
%\rho_{ii,jj}= (xx^*\ot yy^*)_{ii,jj}=x_i\ovx_jy_i\ovy_j= (x\circ y) (x\circ y)^*)_{ii,jj}.
\end{split}
\end{align}
Assume $(X,Y)\in\PCP_n$, i.e., $\rho_{(X,Y)}\in\SEP_n$. Then, $\rho_{(X,Y)}=\sum_{\ell=1}^{m}x_\ell x_\ell^*\ot y_\ell y_\ell^*$ for some $x_\ell,y_\ell\in\C^n$.
Using relation \eqref{eq:rho-ABC} combined with (\ref{eq:relrhoxy}), we obtain
    \begin{align}\label{eq:X,YinPCPdec}
      X=\sum_{\ell=1}^m (x_\ell\circ \overline{x}_\ell)(y_\ell\circ \overline{y}_\ell)^*,\ Y=\sum_{\ell=1}^m(x_\ell\circ y_\ell)(x_\ell\circ y_\ell)^*,
    \end{align}
as desired.    Conversely, assume that \eqref{eq:X,YinPCPdec} holds for some $x_\ell,y_\ell\in\C^n$. 
Then,  using  relations (\ref{eq:suppCLDUI}), (\ref{eq:rho-ABC}) and (\ref{eq:relrhoxy}), we obtain that 
 $\Pi_{\CLDUI_n}\left(\sum_{l=1}^{m}x_lx_l^*\ot y_ly_l^*\right)=\rho_{(X,Y)}.$
Since the projection preserves separability (by Lemma \ref{lem:projection-psd-sep}), 
 this shows    $\rho_{(X,Y)}\in\SEP_n$ and thus $(X,Y)\in \PCP_n$, as desired.
\qed\end{proof}

The alternate characterization (\ref{eq:PCPDec}), which is in fact the original definition of $\PCP_n$ in \cite{JML-PCP}, further illustrates %solidifies 
the connection between the two cones $\PCP_n$ and $\CP_n$. Clearly, every completely positive matrix is both entrywise nonnegative and positive semidefinite. These two properties are now `split up' between the two arguments of a pairwise completely positive matrix tuple: 
%If $(X,Y)\in\PCP_n$, then $X$ is entrywise nonnegative and $Y$ is positive semidefinite,
\begin{align*}
(X,Y)\in \PCP_n \Longrightarrow X\ge 0 \text{ and } Y\succeq 0,
\end{align*}
 which  follows  from relation (\ref{eq:PCPDec}).
% (Lemma \ref{lem:necCondPCP} below).
Another necessary condition for $(X,Y)\in \PCP_n$ is 
\begin{align}\label{eq:propiii}
|Y_{ij}|^2\leq X_{ij}X_{ji} \ \text{  for all } i,j\in[n]
\end{align}
(coined in \cite{JML-PCP} as ``$X$ is almost entrywise larger then $Y$").
%
%Since deciding whether a state $\rho_{(X,Y)}\in\CLDUI_n$ is separable amounts to deciding whether $(X,Y)\in\PCP_n$, it is desirable to have necessary conditions on $X$ and $Y$ for membership in $\PCP_n$, which is the content of the following lemma. 
%
%\begin{lemma}\label{lem:necCondPCP}\cite{JML-PCP}
%   Assume  $(X,Y)\in \PCP_n$. Then, the following holds.
 %   \begin{itemize}
  %      \item[(i)] $X$ is entrywise nonnegative, i.e., $X\geq0$.
%        \item[(ii)] $Y$ is positive semidefinite, i.e., $Y\in\Herm^n_+$.  
%        \item[(iii)] $X$ is `almost' entrywise larger then $Y$, i.e., $|Y_{ij}|^2\leq X_{ij}X_{ji}$ for all $i,j\in[n]$.
%    \end{itemize}
% \end{lemma}
%
Johnston and MacLean \cite{JML-PCP} show this using the description (\ref{eq:PCPDec}) of $\PCP_n$ (and Cauchy-Schwartz inequality). Alternatively, one can show it at the `state level' since $(X,Y)\in\PCP_n$ if and only if $\rho_{(X,Y)}\in \SEP_n$. This implies %$\rho_{(X,Y)} \in\DPS^{(1)}_n$ and thus 
$\rho_{(X,Y)}^{T_B}=\rho_{(X,\cdot,Y)}\succeq 0$, which, by relation \eqref{eq:rhoXYZ-block}, gives (\ref{eq:propiii}) (and $X\ge 0$, $Y\succeq 0$).
%these properties using the description of $\PCP_n$ in Theorem~\ref{theo:PCPDec} (the first two conditions are immediate and the third one is shown using the Cauchy-Schwarz inequality). Alternatively, these properties can be shown at 
%the `state' level since  $(X,Y)\in\PCP_n$ is equivalent to  $\rho_{(X,Y)}\in\SEP_n.$ 
%Then, (i),(ii)  follow   from the fact that $\rho_{(X,Y)}$ is positive semidefinite (by relation \eqref{eq:rho-ABC}). Property (iii) follows from the fact  that $\rho_{(X,Y)}\in\DPS^{(1)}_n$, implying   that  $\rho_{(X,Y)}^{T_B}=\rho_{(X,\cdot,Y)}$ is positive semidefinite, which, by relation \eqref{eq:rho-ABC}, is precisely property (iii).

\medskip
The above results and arguments can analogously be cast for the LDUI and   LDOI cases. The corresponding projections are given by
\begin{align}
    \Pi_{\LDUI_n}(\rho_{AB})=\int_{\MDU_n}(U\ot U)\rho_{AB}(U\ot U)^*dU,\label{eq:PiLDUI}\\
    \Pi_{\LDOI_n}(\rho_{AB})=\frac{1}{2^n}\sum_{O\in\MDO_n}(O\ot O)\rho_{AB}(O\ot O)^*.\label{eq:PiLDOI}
\end{align}
Note that the identity  $(\Pi_{\CLDUI_n}(\rho_{AB}))^{T_B}=\Pi_{\LDUI_n}(\rho_{AB}^{T_B})$ indeed holds. Moreover, as $|\MDO_n|=2^n$, the projection onto the $\LDOI_n$ subspace is given by a finite sum.  

\begin{remark}\label{rem:finite-sum}
Also for  the projections onto the $\CLDUI_n$ and $\LDUI_n$ subspaces, one could replace the integral over $\MDU_n$ by a finite sum over a suitably selected finite subset. This follows from a general result by Tchakaloff \cite{Tchakaloff}. We return to this question in Remark \ref{rem:finite-sum2}, where we give a finite representation for the projection onto a generalization of $\CLDUI_n$. 
\end{remark}

Following \cite{SN-CLDUI}, we have the following alternative characterization  for the cone $\TCP_n$, whose proof is analogous to that of Theorem \ref{theo:PCPDec}.% and thus omitted.

\begin{theorem}\label{theo:TCPDec}\cite{SN-CLDUI}
    Let $(X,Y,Z)\in\R^{n\times n}\times \Herm^n\times \Herm^n$ with $\diag(X)=\diag(Y)=\diag(Z)$. Then, $(X,Y,Z)\in\TCP_n$ (i.e., the associated LDOI matrix $\rho_{(X,Y,Z)}$ belongs to $\SEP_n$) if and only if 
    \begin{align}\label{eq:TCPDec}
        (X,Y,Z)\in\cone\{((x\circ \overline{x})(y\circ \overline{y})^*,(x\circ y)(x\circ y)^*,(x\circ \ovy)(x\circ \ovy)^*):x,y\in\C^n\}.
    \end{align}
\end{theorem}

%The proof 
%runs analogue to the proof given 
%is analogous to that given above for Theorem \ref{theo:PCPDec} and thus omitted. 

%Additionally, the proof implies\footnote{\ML{as the proof is not explicitly given, this paragraph is a bit unclear.}}, that if a LDOI state $\rho_{(X,X,X)}$ is separable, then its separable decomposition can be expressed with real-valued atoms: $\rho_{(X,X,X)}\in\SEP_n$ if and only if $X\in\CP_n$. Therefore, there are $x_i\in\R^n_+$ such that $\sum x_ix_i^T=X$. Denote with $\sqrt{x_i}$ the entrywise square root of $x_i$. Now, $$\rho_{(X,X,X)}=\Pi_{\LDOI_n}(\sqrt{x_i}\sqrt{x_i}^*\ot\sqrt{x_i}\sqrt{x_i}^*)=\frac{1}{2^n}\sum_{O\in\MDO_n}O\sqrt{x_i}(O\sqrt{x_i})^*\ot\sqrt{x_i}(O\sqrt{x_i})^*,$$
%where the right hand side is real-valued. Note, that the same argument does not apply to CLDUI and LDUI states of the form $\rho_{(X,X)}$ and $\rho_{(X,\cdot,X)}$, since the averaging on the right hand side relies on complex diagonal unitaries. \footnote{\ML{I also think we should a proof of the relation (\ref{eq:CP-PCP-TCP}) somewhere .. this is why I added the lemma below}}

\begin{lemma}\label{lem:necCondTCP}\cite{SN-CLDUI}
    If $(X,Y,Z)\in \TCP_n$, then $(X,Y), (X,Z) \in\PCP_n$.% and $(X,Z)\in\PCP_n$. 
\end{lemma}

Here too, the  proof can  be done, either  on the `matrix level'  and using the decomposition in \eqref{eq:TCPDec} (following \cite{SN-CLDUI}), 
or on the `state level' and applying the projections $\Pi_{\CLDUI_n}$ and $\Pi_{\LDUI_n}$ to $\rho_{(X,Y,Z)}$. %We omit the (easy) details. 
Finally, we restate the links between the cones $\CP_n$, $\PCP_n$ (from \cite{JML-PCP}) and $\TCP_n$ (from \cite{SN-CLDUI}), mentioned earlier in (\ref{eq:CP-PCP-TCP}), together with a short proof. 

\begin{lemma}\label{lem:CP-PCP-TCP} \cite{SN-CLDUI}
For a matrix $X\in\MS^n$, we have 
\begin{align*}
X\in\CP_n \Longleftrightarrow (X,X)\in \PCP_n\Longleftrightarrow (X,X,X)\in \TCP_n.
\end{align*}
\end{lemma}

\begin{proof}
The implication $(X,X,X)\in\TCP_n\Longrightarrow (X,X)\in \PCP_n$ follows from Lemma \ref{lem:necCondTCP}.
%is clear, using the characterizations in Theorems \ref{theo:PCPDec} and \ref{theo:TCPDec}. 
The implication $X\in\CP_n\Longrightarrow (X,X,X)\in\TCP_n$ follows from the fact that, if $X=aa^T$ for some nonnegative  $a\in\R^n_+$, then, setting $x=y=\sqrt a $ (entrywise) provides a decomposition showing $(X,X,X)\in\TCP_n$ (using the characterization in Theorem \ref{theo:TCPDec}). 
%There remains to show the implication $(X,X)\in\PCP_n\Longrightarrow X\in \CP_n$. A
Finally, assume $(X,X)\in \PCP_n$; we show $X\in\CP_n$. As $(X,X)\in \PCP_n$, we have $\rho_{(X,X)}\in\SEP_n$, implying $\rho_{(X,\cdot,X)}=\rho_{(X,X)}^{T_B}\in \SEP_n$. Since $\rho_{(X,\cdot,X)}$ is in addition Bose symmetric it admits a symmetric separable decomposition, of the form $\sum_\ell x_\ell x_\ell^*\ot x_\ell x_\ell^*$ for some $x_\ell\in \C^n$ (Lemma \ref{lem:SEP-BS}). From this follows that $X=\sum_\ell a_\ell a_\ell^T$, where we set $a_\ell =(|(x_\ell)_i|^2)_{i=1}^n\in\R^n_+$, thus showing $X\in\CP_n$.
\qed\end{proof}

\begin{remark}
Observe that any separable state of the form $\rho_{(X,X,X)}$ admits a real separable decomposition, i.e., involving only real vectors. This follows as a byproduct of the above proof: If $\rho_{(X,X,X)}\in \SEP_n$, then $X\in \CP_n$ and thus $X=\sum_\ell a_\ell a_\ell^T$ for some $a_\ell\in \R^n_+$. Setting $x_\ell =(\sqrt{(a_\ell)_i})_{i=1}^n\in\R^n$, we obtain 
$\rho_{(X,X,X)}= \Pi_{\LDOI_n}(\sum_\ell x_\ell x_\ell ^*\ot x_\ell x_\ell ^*)$. Combined with the fact that the projection $\Pi_{\LDOI_n}$ can be formulated as an average by real-valued matrices  (see (\ref{eq:PiLDOI})), we obtain  a separable decomposition by real vectors, as desired. This argument does not extend to the CLDUI or LDUI cases, since then the averaging is by complex-valued matrices (see (\ref{eq:PiLDUI})). Indeed, in Section \ref{sec:exampleab} (see Remark \ref{rem:exampleab}), one can see a real-valued separable CLDUI state of the form $\rho_{(X,X)}$ that does not admit a real separable decomposition.
\end{remark}

\subsection{Dual cone of $\PCP_n$} 

%\ML{This is a first draft, in order to place the definition of the matrices $M_{S,T}$.}
%We discuss basic facts about the dual of the cone $\PCP_n$. 
As $\PCP_n$ is a cone in $  \R^{n\times n}\times\Herm^n$,  its dual cone is defined as
$$\PCP_n^*=\{(S,T)\in  \R^{n\times n}\times \Herm^n: \langle (S,T),(X,Y)\rangle \ge 0 \text{ for all } (X,Y)\in \PCP_n\},$$
where $\langle (S,T),(X,Y)\rangle=\langle S,X\rangle +\langle T,Y\rangle$ is the standard inner product in the product space. In order to define the dual of $\PCP_n$ at the `state level' we make  the following definition.
%The characterization of $\PCP_n$ from Theorem \ref{theo:PCPDec} motivates the following definition.

	\begin{definition}\label{def:MST}
		Given a pair  $(S,T)\in \R^{n\times n}\times \Herm^n $, define  the following matrix 
		% $M_{S,T}\in \Herm(\Cnn)$, 
		\begin{align}\label{eq:MST}
M_{S,T}=
%\sum_{i=1}^n(S_{ii}+T_{ii})e_ie_i^*\otimes e_ie_i^* 
 \sum_{i,j=1}^nS_{ij}e_ie_i^*\otimes e_je_j^*
+\sum_{i,j=1}^nT_{ij}e_je_i^*\otimes e_je_i^* \in \Herm(\Cnn).
		\end{align}
	\end{definition}

\noindent
So, $M_{S,T}\in \CLDUI_n$ by construction.
Given a pair $(X,Y)\in \R^{n\times n}\times \Herm^n$ with $\diag(X)=\diag(Y)$, recall from (\ref{eq:rhoXYZ-block})   the associated   bipartite state $\rho_{(X,Y)}\in\CLDUI_n$, defined as
\begin{align}\label{eq:rhoXYdef}
\rho_{(X,Y)}=
%\sum_{i=1}^n X_{ii} e_ie_i^*\ot e_ie_i^* 
 \sum_{i,j=1}^n X_{ij}e_ie_i^*\ot e_je_j^*
+ \sum_{i,j=1, i\ne j}^n Y_{ij}e_ie_j^*\ot e_ie_j^*\in \CLDUI_n.
\end{align}
%Note also that, for any $(S,T), (X,Y)\in\R^{n\times n}\times \Herm^n$ with $\diag(X)=\diag(Y)$, we have
Then, one can easily verify that 
\begin{align}\label{eq:MST-rhoXY}
\langle M_{S,T}, \rho_{(X,Y)}\rangle = \langle S,X\rangle +\langle T,Y\rangle = \langle (S,T),(X,Y)\rangle.
\end{align}
Moreover, 
one can verify the following polynomial identity:
\begin{align}\label{eq:MSTxy}
\langle M_{S,T}, x x^*\ot y y^*\rangle =
\langle S,(x \circ \ovx) (y \circ \ovy)^*\rangle 
+\langle T, (x \circ y) (x\circ y)^*\rangle.
\end{align}
%Combining with the characterization of $\PCP_n$ in Theorem  \ref{theo:PCPDec} gives the following result. % description for the dual cone $\PCP_n^*$.
The characterization of the dual cone $\PCP_n^*$ now follows easily.

\begin{lemma}\label{lem:PCP-dual}
For a pair $(S,T)\in \R^{n\times n}\times \Herm^n $, $(S,T)\in \PCP_n^*$ if and only if $M_{S,T}\in \SEP_n^*$, i.e., the polynomial $\langle M_{S,T},xx^*\ot yy^*\rangle$ is globally nonnegative.
%the following assertions are equivalent.
%\begin{itemize}
%\item[(i)] $(S,T)\in \PCP_n^*$.
%\item[(ii)] $M_{S,T}\in \SEP_n^*$, i.e., the polynomial $\langle M_{S,T},xx^*\ot yy^*\rangle$ is globally nonnegative.
%\end{itemize}
%we have 
%$$M_{S,T}\in \SEP_n^*\Longleftrightarrow (S,T)\in \PCP_n^*.$$
 \end{lemma}
\begin{proof}
%Observe that $\langle M_{S,T}, x x^*\ot y y^*\rangle =\langle S,(x \circ \ovx) (y \circ \ovy)^*\rangle +\langle T, (\ovx \circ \ovy) (\ovx\circ \ovy)^*\rangle$ for any $x,y\in\C^n$. Hence, 
By combining (\ref{eq:MSTxy}) and  the characterization (\ref{eq:PCPDec}) of $\PCP_n$,  it follows that the polynomial $\langle M_{S,T}, x x^*\ot y y^*\rangle$ is globally nonnnegative (i.e., $M_{S,T}\in\SEP_n^*$) if and only  $(S,T)\in\PCP_n^*$.
\qed\end{proof}

\begin{remark}\label{rem:compare-PCP-GNP}
Note that the  matrix $M_{S,T}$ in (\ref{eq:MST}) remains unchanged under  redistributing diagonal entries between $S$ and $T$, i.e.,   $M_{S,T}=M_{S+D,T-D}$   for any diagonal matrix $D$. 
Recall the notation $D_T=\Diag(\diag(T))$ for any square matrix $T$.
By  replacing $(S,T)$ by $(S+D,T-D)$ with $D=(D_T-D_S)/2$, one can   assume that $S$ and $T$ have the same diagonal. Moreover,
by   replacing $(S,T)$ by $(S+D_T, T-D_T)$, one can move all diagonal entries to the first argument, in which case  one gets 
$M_{S,T}= M_{S+D_T, T-D_T}=\rho_{(S+D_T, T^T +D_S)}$ and 
\begin{align}\label{eq:MST-AB-GNP}
\rho_{(A,B)}=M_{A, B^T-D_B} \text{ for any } (A,B)\in\R^{n\times n}\times\Herm^n \text{ with } \diag(A)=\diag(B).
\end{align}
We find it more convenient  %not to make this assumption, so that 
to keep the notation $M_{S,T}$ that can be freely used for {\em any} $S,T$ and permits   to distinguish between the primal side (where states $\rho_{(X,Y)}$ live) and the dual side (where matrices $M_{S,T}$ live). 
We now compare the cone $\PCP_n^*$ with the cone  $\PCOP_n$ (of {\em pairwise copositive matrices})  recently  introduced in   \cite[Def. 3.1]{GNP_2025}. 
%Both are  closely related, with  the following minor differences between the two settings.}
%There,  the authors introduce the cone  of {\em pairwise copositive matrices} $\PCOP_n$ and they show that it is the polar of the cone $\PCP_n$. As we now point out, there is a close connection to the above result in Lemma \ref{lem:PCP-dual}, but there are (minor) differences between the settings. 
%\footnote{The authors use also the notation $\stackrel{\circ}{B}=B-D_B$ for the matrix obtained by setting the diagonal entries of $B$ to 0.}, 

 As in \cite{GNP_2025}, set
$\R^{n\times n}\times_{\R^n} \Herm^n=\{(A,B)\in \R^{n\times n}\times \Herm^n: \diag(A)=\diag(B)\}$.
% as in \cite{GNP_2025}.
% and, for a square matrix $B$,  $\stackrel{\circ}{B}=B-D_B$ is obtained from $B$ by setting its diagonal entries to 0.
Then,   $\PCOP_n$ is defined  as the set of matrix pairs $(A,B)\in \R^{n\times n}\times_{\R^n}  \Herm^n$ for which  the polynomial 
$\langle A, (x\circ \ovx)(y\circ \ovy)^*\rangle +\langle B-D_B, (x\circ y) (x\circ y)^*\rangle  $ is globally nonnnegative. Hence, we have
%$(A,B)\in \PCOP_n$ if and only if $M_{A,B-D_B}\in \SEP_n^*$, i.e.,  $(A,B-D_B)\in\PCP_n^*$. 
$$(A,B)\in \PCOP_n \Longleftrightarrow M_{A,B-D_B}\in \SEP_n^*, \ \text{ i.e., } \ (A,B-D_B)\in\PCP_n^*.$$
 In \cite[Prop. 3.2]{GNP_2025} it is shown that $\PCOP_n$ is the dual of $\PCP_n$. The (minor) difference with our setting is that  the dual is taken with respect to  the variation of the standard inner product, where the diagonal entries are `counted only once'; that is, defining the scalar product of two matrix pairs $(A,B),(A',B')\in \R^{n\times n}\times_{\R^n}\Herm^n$  as
   $\langle A,A'\rangle +\langle B-D_B, B'-D_{B'}\rangle$.
   
   %Note, however, that the polar  is defined using the variation of the scalar product defined on the space $\{(A,B)\in \R^{n\times n}\times \Herm_n: \diag(A)=\diag(B)\}$, where the diagonal entries are `counted only once'; that is, defining the scalar product of two pairs $(A,B),(A',B')$ in this space as}    $\langle A,A'\rangle +\langle B-\overset{ \circ}{B}, B'-\overset{\circ}{B'}\rangle$.

%\ML{Finally, note that the definition of the matrix $M_{S,T}$ in (\ref{eq:MST}) remains unchanged under redistributing diagonal entries between $S$ and $T$, i.e., we have $M_{S,T}=M_{S+D,T-D}$ for any diagonal matrix $D$. So, we could as well assume that $S$ and $T$ have the same diagonal by taking $D=\Diag(\diag(T)-\diag(S))$. For convenience we do not make this} assumption.
\end{remark}

Here are some basic results about $\PCP_n^*$ (also mentioned in \cite{GNP_2025}) that we will need later.

\begin{lemma}\label{lem:PCP-dual1}
\begin{itemize}
\item[(i)]
If $(X,Y)\in \PCP_n$, then $(X^T,Y^T)\in \PCP_n$.% and $(S^T,T^T)\in \PCP_n^*$.
\item[(ii)] If $(S,T)\in \PCP_n^*$, then $(S^T,T^T)\in \PCP_n^*$.
\item[(iii)]
If $(S,T)\in \PCP_n^*$, then   $S+S^T+T+T^T\in \COP_n$.
\end{itemize}\end{lemma} 
\begin{proof}
(i),(ii) follow using the definitions. For (iii), note that  $(S,T) \in\PCP_n^*$ implies that the polynomial in (\ref{eq:MSTxy}) is globally nonnegative; by evaluating it at $y=\ovx$, one gets the desired result.%\footnote{\ML{shorter argument than what we had- can you please check it ..}}
%Assume $(S,T)\in \PCP_n^*$. Then, for any $X\in \CP_n$, 
%$\langle (S,T), (X,X)\rangle \ge 0$ holds since $(X,X)\in\PCP_n$ (by Lemma \ref{lem:CP-PCP-TCP}). As $(S^T,T^T)\in\PCP_n^*$ (by (ii)), $\langle (S^T,T^T),(X,X)\rangle \ge 0$ also holds. Summing these two inequalities gives the inequality
%$\langle S+S^T+T+T^T,X\rangle\ge 0$ for all $X\in\CP_n$, which shows $S+S^T+T+T^T\in\COP_n$, as desired.
\qed\end{proof}

%\begin{lemma}\label{lem:PCP-dual2}
%If $(S,T)\in \PCP_n^*$, then the matrix $S+S^T+T+T^T$ belongs to $\COP_n$.
%\end{lemma}

\begin{lemma}\label{lem:PCP-dual-diag}
%\footnote{\ML{Do we need this lemma?  any other fact to mention?}}
\begin{itemize}
\item[(i)] If $(S,T)\in \PCP_n^*$, then $S_{ij}+S_{ji}\ge 0$  (for $i\ne j\in [n]$) and $S_{ii}+T_{ii}\ge 0$ (for $i\in [n]$).
\item[(ii)] Let $S\in \C^{n\times n}$. Then,  
 $(S,S)\in \PCP_n^*\Longleftrightarrow S\in  \MS^n\cap \R^{n\times n}_+.$
 \end{itemize}
\end{lemma}
\begin{proof}
(i) follows by taking the inner product of $(S,T)$ with the matrix pairs $(e_ie_j^*+e_je_i^*,0)$ and $(e_ie_i^*,e_ie_i^*)$ that both belong to $\PCP_n$. 
We show (ii). If $(S,S)\in \PCP_n^*$, then $S$ is symmetric (since real, Hermitian) and   $S\ge 0$ by (i). Conversely, assume $S\in\MS^n$ and $S\ge 0$. For $(X,Y)\in\PCP_n$,  
$$\langle (S,S),(X,Y)\rangle =\sum_{i=1}^n S_{ii}(X_{ii}+Y_{ii}) + \sum_{1\le i<j\le n}S_{ij}(X_{ij}+X_{ji}+Y_{ij}+Y_{ji}).$$
The first sum is nonnegative since $S,X\ge 0$ and $\diag(Y)=\diag(X)$. Also the second sum is nonnnegative:  by (\ref{eq:propiii}), we have 
$\Rea(Y_{ij})^2\le |Y_{ij}|^2 \le X_{ij}X_{ji}$, implying $\Rea(Y_{ij}) \ge -\sqrt{X_{ij}X_{ji}}$ and thus 
$X_{ij}+X_{ji}+Y_{ij}+Y_{ji}= X_{ij}+X_{ji}+2 \Rea(Y_{ij}) \ge (\sqrt{X_{ij}}-\sqrt{X_{ji}})^2 \ge 0$. So, $(S,S)\in\PCP_n^*$.
\qed\end{proof}

More results about the cone $\PCOP_n$ can be found in \cite{GNP_2025}. In particular, it is shown in \cite[Theorem 3.9, Corollary 3.10]{GNP_2025} that, for a pair $(A,B)\in \R^{n\times n}\times_{\R^n}\Herm^n$, $A\ge 0$ and $\tilde A+ \Rea (U^*(B-D_B)U)\in\COP_n$ for all $U\in\MDU_n$ implies  $(A,B)\in \PCOP_n$, after setting $\tilde A=(\sqrt{A_{ij}A_{ji}})_{i,j=1}^n$. Moreover, this implication holds as an equivalence if $A$ is real symmetric. In addition, when $B$ is real symmetric, it suffices to consider real-valued $U\in \MDU_n$ in the previous claim. %\label{sec:background}

\section{The DPS hierarchy for diagonal unitary invariant states %CLDUI states
}\label{sec:DPS-CLDUI}

In this section we investigate how to adapt the DPS hierarchy for  bipartite states having some diagonal unitary invariance, with the aim of obtaining more economical relaxations. 
%Our strategy will be to show that separable bipartite states $\rho_{AB}\in \CLDUI_n$ admit $\DPSt$-certificates that enjoy invariance and sparsity properties inherited from the state $\rho_{AB}$. 
We   first focus on the class  $\CLDUI_n$ and then indicate how the results easily extend  to the classes $\LDUI_n$, $\LDOI_n$.

%introduce extended invariance classes for the $\DPS$ hierarchy based on the diagonal unitary invariance of the previous section. First the focus will be on extending $\CLDUI_n$, since the results will be easily generalized to $\LDUI_n$ and $\LDOI_n$.
%\smallskip

\subsection{An extended sparsity pattern for DPS certificates of CLDUI states}\label{sec:DPS-CLDUIt-cert}

Here, we  show that any separable bipartite state $\rho_{AB}\in \CLDUI_n$ admits a $\DPSt$-certificate that enjoys invariance and sparsity properties inherited from those of the state $\rho_{AB}$. 
For this, for an integer $t\ge 1$, define the set
\begin{align}\label{def:CLDUI^t}
   \begin{split}
    \CLDUI^{(t)}_n=\{& \rho_{AB[t]}\in\Herm(\C^n\ot(\C^n)^{\ot t}):\\
    & \rho_{AB[t]}= (U\ot \ovU^{\ot t}) \rho_{AB[t]} (U\ot \ovU^{\ot t})^* \ \forall U\in\MDU_n\}
\end{split}\end{align}
and let  $\Pi_{\CLDUI^{(t)}_n}$ denote the projection onto   $\CLDUI^{(t)}_n$.
So, for $t=1$, we recover $\CLDUI_n$ and, for $t\ge 1$, the definition follows the analogous recipe, motivated by the fact that the $t$ copies of the B-register should all be treated in the same way (in order to preserve invariance under any permutation of these $t$ registers).
%we Roughly speaking, $\CLDUI^{(t)}_n$ is set up the same way $\CLDUI_n$ is by extending the invariance in the $B$ register symmetrically to the $t$ copies of the $B$ register. 

%$In Section \ref{sec:DPS-CLDUIt-cert} 
Our aim is to prove the following result, % (Theorem \ref{theo:CLDUItDPSCert}), 
which claims   that any state in $\CLDUI_n\cap \DPSt_n$ admits a $\DPSt$-certificate in $\CLDUI^{(t)}_n$ and after that % in Section \ref{sec:DPS-CLDUIt-block}, 
we will see that such certificate enjoys a block-diagonal structure.

\begin{theorem}\label{theo:CLDUItDPSCert}
    Assume that $\rho_{AB}\in\DPSt_n$ with $\DPS^{(t)}$-certificate $\rho_{AB[t]}$. Then, $\Pi_{\CLDUI_n}(\rho_{AB})\in\DPS^{(t)}_n$ with $\DPS^{(t)}$-certificate $\Pi_{\CLDUI^{(t)}_n}(\rho_{AB[t]})$. Therefore, if $\rho_{AB}\in\CLDUI_n$, then $\rho_{AB}$ admits a $\DPS^{(t)}$-certificate that belongs to   $\CLDUI_n^{(t)}$.
\end{theorem}

 In order to prove Theorem \ref{theo:CLDUItDPSCert} 
  we  first collect structural results about the set $\CLDUI^{(t)}_n$.
Recall from relation \eqref{eq:suppCLDUI} that $\rho_{AB}\in\CLDUI_n$ if and only if $\supp(\rho_{AB})\subseteq\Omega_n$.  As in the case $t=1$, the set $\CLDUI^{(t)}_n$ allows for alternate descriptions through its sparsity pattern and the associated projection. For this, we now define the following   `extended' sparsity pattern:
\begin{align}\label{def:Omegat}
    \Omega_n^{(t)}=\{(i_0\ui,j_0\uj)\in[n]^{t+1}\times[n]^{t+1}:\{i_0\uj\}=\{j_0\ui\}, \ \text{i.e., } \alpha(i_0\uj)=\alpha(j_0\ui)\}.
\end{align}
%$Recall that  $\{\ui\}$ denotes the multiset $\{i_0,i_1,\ldots,i_t\}$ and %that the condition $\{i_0\uj\}=\{j_0\ui\}$ is equivalent to $\alpha(i_0\uj)=\alpha(j_0\ui)$
%  $e_{i_0}+\alpha(\uj)=e_{j_0}+\alpha(\ui)$ 
For $\ui\in [n]^t$, recall that $\alpha(\ui)\in\N^n$, with $k$-th entry the number of occurrences of $k$ in the multiset $\{\ui\}=\{i_1,\ldots,i_t\}$. So, $\alpha(i_0\uj)=e_{i_0}+\alpha(\uj)$. Setting $t=1$ recovers the original set $\Omega_n$, i.e.,  we have
$$\Omega^{(1)}_n=\{(ij,kl)\in[n]^2\times[n]^2:\{il\}=\{jk\}\}=\Omega_n.$$
Moreover, as in the case $t=1$ (Remark \ref{rem:finite-sum}), in definition (\ref{def:CLDUI^t}) it suffices to require invariance under a finite subset of $\MDU_n$.

\begin{lemma}\label{lem:suppCLDUIt}
    Let $\rho_{AB[t]}\in\Herm(\C^n\ot(\C^n)^{\ot t})$. Then, the following assertions are equivalent.
    \begin{itemize}
    \item[(i)] $\rho_{AB[t]}\in\CLDUI^{(t)}_n$, i.e., $\rho_{AB[t]}= (U\ot \ovU^{\ot t}) \rho_{AB[t]} (U\ot \ovU^{\ot t})^*$ for   $U\in\MDU_n$.
        \item[(ii)]  $\rho_{AB[t]}= (U_k\ot \overline{U_k}^{\ot t}) \rho_{AB[t]} (U_k\ot \overline{U_k}^{\ot t})^*$ for all $k\in [n-1]$, where  
        $U_k$ is the diagonal matrix with $(U_k)_{k,k}=\omega$ and all other diagonal entries equal to 1, and we set\footnote{Alternatively, one could choose  $\omega=e^\frac{2\pi \bfi}{\lambda}$ for some  $\lambda \in\R\setminus\Q$, which would  give a finite representation independent of the parameter $t$.}
         $\omega:=e^{2\pi\bfi\over t+2}$.
                    \item[(iii)] $\supp(\rho_{AB[t]})\subseteq\Omega_n^{(t)}$.
    \end{itemize}
\end{lemma}

\begin{proof}
The implication (i) $\Longrightarrow $ (ii) is obvious. %We show (ii) $\Longrightarrow$ (iii). 
Assume (ii) holds, we show (iii).
First, note that for any 
 $U\in\MDU_n$ with $u=\diag(U)$,  and any index  pair $(i_0\ui,j_0\uj)\in [n]^{t+1}\times [n]^{t+1}$, by (\ref{eq:Au}), one has
 %for indices $i_0\ui,j_0\uj\in [n]^{t+1}$, we have
  \begin{align}\label{eq:genU}
  \begin{split}
  ((U\ot \ovU^{\ot t}) \rho_{AB[t]} (U\ot \ovU^{\ot t})^*)_{i_0\ui,j_0\uj}
  & =
(\rho_{AB[t]})_{i_0\ui,j_0\uj} \cdot u_{i_0}\ou_{i_1}\dots\ou_{i_t}\ou_{j_0}u_{j_1}\dots u_{j_t}\\
& = (\rho_{AB[t]})_{i_0\ui,j_0\uj} \cdot u^{\alpha(i_0\uj)} \overline u^{\alpha(j_0\ui)}.
\end{split}
\end{align}
Consider an index pair $(i_0\ui, j_0\uj)\not\in \Omega^{(t)}_n$. 
   Then, % $(e_{i_0}+\alpha(\uj))_k\not=(e_{j_0}+\alpha(\ui))_k$ 
    $\alpha(i_0\uj)_k\ne \alpha(j_0\ui)_k$ for some index $k$ that can be chosen in $[n-1]$ since $|\alpha(i_0\uj)|=|\alpha(j_0\ui)|$.
    % since both sequences have the same weight $t+1$. %there exists $k\in[n-1]$ such that . 
 W.l.o.g. we may assume $\ell:= \alpha(j_0\ui)_k- \alpha(i_0\uj)_k>0$. Moreover, $\ell \le t+1$ (since $|\alpha(j_0\ui)|=t+1$). Note that this implies that $\omega^\ell\ne 1$. 
Now, we apply (\ref{eq:genU}) to the matrix $U=U_k$; then, the value in (\ref{eq:genU}) is equal to 
$(\rho_{AB[t]})_{i_0\ui,j_0\uj} \cdot \omega^{\alpha(i_0\uj)_k-\alpha(j_0\ui)_k}=  (\rho_{AB[t]})_{i_0\ui,j_0\uj} \cdot \omega^\ell$, which is equal to   $(\rho_{AB[t]})_{i_0\ui,j_0\uj}$  only if   $(\rho_{AB[t]})_{i_0\ui,j_0\uj}=0$, since $\omega^\ell\ne 1$. So, (iii) holds.

 Finally, assume (iii) holds, we show (i). In view of (\ref{eq:genU}) it suffices to show that the value in (\ref{eq:genU}) is equal to $(\rho_{AB[t]})_{i_0\ui,j_0\uj}$ when $(\rho_{AB[t]})_{i_0\ui,j_0\uj} \ne 0$. But, then,  $(i_0\ui,j_0\uj)\in \Omega_n^{(t)}$, which gives $\alpha(i_0\ui)=\alpha(j_0\ui)$ and thus 
 $u^{\alpha(i_0\uj)} \overline u^{\alpha(j_0\ui)}=1$. Hence, the value in (\ref{eq:genU})  equals
  $(\rho_{AB[t]})_{i_0\ui,j_0\uj}$, as desired.
  %
%      The right hand side of the above equation simplifies to $(\rho_{AB[t]})_{i_0\ui,j_0\uj}$ if $i_0=j_0$ and $\{\ui\}=\{\uj\}$,  or if $i_0\not=j_0$ and $\{\ui\}\setminus j_0=\{\uj\}\setminus i_0$, which is equivalent to asking that $(i_0\ui,j_0\uj)\in \Omega^{(t)}_n$.
%    Moreover,  as in the case $t=1$, at any other position, there exists $u\in\T^n$ such that $u_{i_0}\ou_{i_1}\dots\ou_{i_t}\ou_{j_0}u_{j_1}\dots u_{j_t} \ne 1$. This  implies that $\rho_{AB[t]}\in\CLDUI^{(t)}_n$ if and only if $(\rho_{AB[t]})_{i\ui,j\uj}=0$ when $(i_0\ui,j_0\uj)\not\in \Omega^{(t)}_n$. 
\qed\end{proof}

\begin{remark}\label{rem:finite-sum2}
    As for $t=1$, the projection onto the $\CLDUI^{(t)}_n$ subspace is given by 
\begin{align}\label{def:projCLDUIt}
    \Pi_{\CLDUI^{(t)}_n}(\rho_{AB[t]})=\int_{\MDU_n}(U\ot\overline{U}^{\ot t})\rho_{AB[t]}(U\ot\overline{U}^{\ot t})^*dU.
\end{align}
In view of Lemma \ref{lem:suppCLDUIt},  the projection can also be seen as the map that preserves the entries indexed by  $\Omega^{(t)}_n$ and places 0 everywhere else. Moreover, the integral can be replaced by a finite sum by summing over the subgroup generated by the $n-1$ matrices $U_k$ from Lemma~\ref{lem:suppCLDUIt}. This subgroup of $\MDU_n$ consists of $(t+2)^{n-1}$ elements.
\end{remark}

The following lemma is the key ingredient for the proof of Theorem \ref{theo:CLDUItDPSCert}, which is given thereafter. Its proof relies on the sparsity structure of $\CLDUI^{(t)}_n$ from Lemma \ref{lem:suppCLDUIt} and the  integral formulation (\ref{def:projCLDUIt}) of the corresponding projector.

\begin{lemma}\label{lem:WsPiCLDUIt}
Let $W\in \Herm(\C^n\ot(\C^n)^{\ot t})$ and consider an integer $0\le s\le t$.
  \begin{itemize}
  \item[(i)]  If $W\in\Herm(\C^n\ot S^t(\C^n))$, then $\Pi_{\CLDUI^{(t)}_n}(W)\in \Herm(\C^n\ot S^t(\C^n))$. 
  % If $W\in\MV^{(t)}_n$, then $\Pi_{\CLDUI^{(t)}_n}(W)\in \MV^{(t)}_n$.   
  \item[(ii)] If  $W\in \MW^{(t)}_{n,s}$, then $\Pi_{\CLDUI^{(t)}_n}(W)\in\MW^{(t)}_{n,s}$.	
  \end{itemize}
\end{lemma}
	
\begin{proof}
 (i) Assume $W\in\Herm(\C^n\ot S^t(\C^n))$; we show  $\widetilde W:=\Pi_{\CLDUI^{(t)}_n}(W)\in \Herm(\C^n\ot S^t(\C^n))$.
 For this, let $i_0\ui,j_0\uj\in [n]^{t+1}$ and $\sigma\in \Perm_t$. 
 By assumption, $W_{i_0\ui,j_0\uj}=W_{i_0\ui,j_0\sigma(\uj)}$; we need to show that the same holds for $\widetilde W$. By Lemma \ref{lem:suppCLDUIt}, the support of $\widetilde W$ is contained in the set $\Omega^{(t)}_n$. So, it suffices now to observe that $(i_0\ui,j_0\uj)\in\Omega^{(t)}_n$ $\Longleftrightarrow$ $(i_0\ui,j_0\uj^\sigma)\in \Omega^{(t)}_n$, since  $\{i_0\uj\}$ and $\{i_0\uj^\sigma\}$ are the same multisets.
   
 (ii)   Assume $W\in \MW^{(t)}_{n,s}$, i.e., $W^{T_{B[s]}}\succeq 0$ (by   \eqref{eq:Wst}). Let $U\in\MDU_n$ and $u=\diag(U)$. Then, 
    \begin{align*}
		(U\ot \overline U^{\ot t}) W (U\ot   \overline U^{\ot t})^* & =
		W\circ (uu^* \ot (\ou \  \ou^*)^{\ot t}),\\
 (W\circ (uu^* \ot (\ou \  \ou^*)^{\ot t}))^{T_{B[s]}} &  =W^{T_{B[s]}} \circ (uu^* \ot (\ou \  \ou^*)^{\ot t})^{T_{B[s]}}  \\
 & =W^{T_{B[s]}} \circ ( (u u^*)^{\ot (s+1)} \ot (\ou \ \ou^*)^{\ot(t-s)})\\
& = (U^{\ot (s+1)}\ot\overline{U}^{\ot (t-s)})W^{T_{B[s]}}(U^{\ot (s+1)}\ot\overline{U}^{\ot (t-s)})^*
 \end{align*}
(using (\ref{eq:Au}) at the first and last lines).    Using both equations gives $$(\Pi_{\CLDUI^{(t)}_n}(W))^{T_{B[s]}}
    =\int_{\MDU_n}(U^{\ot (s+1)}\ot\overline{U}^{\ot (t-s)})W^{T_{B[s]}}(U^{\ot (s+1)}\ot\overline{U}^{\ot (t-s)})^*dU.$$
As $W^{T_{B[s]}}\succeq 0$, we obtain that each integrand is positive semidefinite, and thus $(\Pi_{\CLDUI^{(t)}_n}(W))^{T_{B[s]}}\succeq~0$.
\qed
\end{proof}

\begin{proof}[Proof of Theorem \ref{theo:CLDUItDPSCert}]
Assume  $\rho_{AB}\in\DPSt$ with $\DPS^{(t)}$-certificate $\rho_{AB[t]}$. 
%So, $\rho_{AB[t]}$ belongs to $\MV^{(t)}_n$ and  $\MW^{(t)}_{n,s}$ for all integers $0\leq s\leq t$, and   $\Tr_{B[2:t]}(\rho_{AB[t]})=\rho_{AB}$. 
%For ease of notation we set $\widetilde{{\rho}_{AB}}=\Pi_{\CLDUI_n}(\rho_{AB})$ and $\widetilde{{\rho}_{AB[t]}}=\Pi_{\CLDUI_n}(\rho_{AB[t]})$; w
We show $\widetilde{{\rho}_{AB[t]}}= \Pi_{\CLDUI_n}(\rho_{AB[t]})$ is a $\DPSt_n$-certificate of $\widetilde{{\rho}_{AB}}=\Pi_{\CLDUI_n}(\rho_{AB}) $. So,  $\rho_{AB[t]}\in \Herm(\C^n\ot S^t(\C^n))$,
%\MV^{(t)}_n$,
   $\rho_{AB[t]}\in \MW^{(t)}_{n,s}$ for all integers $0\leq s\leq t$, and   $\Tr_{B[2:t]}(\rho_{AB[t]})=\rho_{AB}$. 
 By   Lemma \ref{lem:WsPiCLDUIt} we have $\widetilde{{\rho}_{AB[t]}}\in \Herm(\C^n\ot S^t(\C^n))$ 
 %\MV^{(t)}_n$ 
 and $\widetilde{{\rho}_{AB[t]}}\in\MW^{(t)}_{n,s}$ for any $0\leq s\leq t$. 
   % We now argue that $\widetilde{{\rho}_{AB[t]}}\in\MV^{(t)}_{n}$, i.e., 
 %   $(\widetilde{{\rho}_{AB[t]}})_{i_0\ui,j_0\uj}=(\widetilde{{\rho}_{AB[t]}})_{i_0\ui,j_0\sigma(\uj)}$ for all $i_0\ui,j_0\uj\in [n]^{t+1}$ and $\sigma\in \Perm_t$.
%Indeed, by 
%       Lemma \ref{lem:suppCLDUIt}, the support of $\widetilde{{\rho}_{AB[t]}}$ is contained in the set $\Omega^{(t)}_n$. It suffices now to observe that $(i_0\ui,j_0\uj)\in\Omega^{(t)}_n$ $\Longleftrightarrow$ $(i_0\ui,j_0\sigma(\uj))\in \Omega^{(t)}_n$ since  $\{i_0\uj\}=\{i_0\sigma(\uj)\}$ as multisets.
%              
    %     it holds that $$(\widetilde{\rho}_{AB[t]})_{i_0\ui,j_0\uj}=\begin{cases}
    %    (\rho_{AB[t]})_{i_0\ui,j_0\uj}&\text{ if }(i_0\ui,j_0\uj)\in\Omega^{(t)}_n,\\
     %   0&\text{ else.}
    %\end{cases}$$
%    Crucially, the condition $\{i_0\uj\}=\{j_0\ui\}$ is invariant under permuting indices of $\uj$. This means, that $(i_0\ui,j_0\uj)\in\Omega^{(t)}_n$ if and only if $(i_0\ui,j_0\sigma(\uj))\in\Omega^{(t)}_n$ for all $\sigma\in\Perm_t$. Together this implies $\widetilde{\rho}_{AB[t]}\in\MV^{(t)}_{n}$.
 There remains to show that $\Tr_{B[2:t]}(\widetilde{{\rho}_{AB[t]}})=\widetilde{{\rho}_{AB}}$. For this, let $i,j,k,l\in[n]$ and $\uh\in[n]^{t-1}$. Note that 
    \begin{align}\label{eq:DiagOmegat}
        (ij,kl)\in\Omega_n \text{ if and only if } (ij\uh,kl\uh)\in\Omega^{(t)}_n,  
    \end{align}
    because $\{ij\}=\{kl\}$ if and only if $\{ij\uh\}=\{kl\uh\}$. Using the definition of the partial trace, 
    $(\Tr_{B[2:t]}(\widetilde{{\rho}_{AB[t]}}))_{ij,kl}$ is equal to
    %    $$(\Tr_{B[2:t]}(\widetilde{{\rho}_{AB[t]}}))_{ij,kl}
   $$ \sum_{\uh\in[n]^{t-1}}(\widetilde{{\rho}_{AB[t]}})_{ij\uh,jk\uh}=\begin{cases}
        (\Tr_{B[2:t]}(\rho_{AB[t]}))_{ij,kl}&\text{ if }(ij,kl)\in\Omega_n,\\
        0&\text{ else.}
    \end{cases}$$
     Since $\Tr_{B[2:t]}(\rho_{AB[t]})=\rho_{AB}$ this is in fact precisely the definition of $\widetilde{{\rho}_{AB}}$.
\qed\end{proof}

The result of Theorem \ref{theo:CLDUItDPSCert} has an immediate counterpart on the dual side: 
any matrix in $(\DPS^{(t)}_n)^*\cap\CLDUI_n$ admits a certificate as in Theorem \ref{theo:FF}(iii) with, in addition,  $B,W_s\in \CLDUI^{(t)}_n$.

\begin{corollary}\label{cor:CLDUIt-dual}
If $M\in(\DPSt_n)^*$, then   $\Pi_{\CLDUI_n}(M)\ot I_n^{\ot (t-1)}$ belongs to 
\begin{align*}
%\Pi_{\CLDUI_n}(M)\ot I_n^{\ot (t-1)}\in 
\Herm(\C^n\ot S^t(\C^n))^\perp
\cap\CLDUI^{(t)}_n +\sum_{s=0}^t \MW^{(t)}_{n,s}\cap\CLDUI^{(t)}_n
\end{align*}
and thus $\Pi_{\CLDUI}(M)\in(\DPSt_n)^*$.
 \end{corollary}

\begin{proof}
%   For any $\rho\in\DPS^{(t)}_n$, we have
 %    $$\langle \Pi_{\CLDUI_n}(M),\rho\rangle=\sum_{(ij,kl)\in\Omega_n}\overline{M}_{ij,kl}\rho_{ij,kl}=\langle M,\Pi_{\CLDUI_n}(\rho)\rangle\geq0,$$
 %   since $\Pi_{\CLDUI}(\rho)\in\DPS^{(t)}_n$ by Theorem \ref{theo:CLDUItDPSCert}, which shows $\Pi_{\CLDUI_n}(M) \in (\DPS^{(t)}_n)^*$. 
Assume that   $M\in(\DPSt_n)^*$.   Then, by Theorem~\ref{theo:FF}(iii), there exist $B\in\Herm(\C^n\ot S^t(\C^n))^\perp$
%(\MV^{(t)}_n)^\perp$ 
and $W_s\in\MW^{(t)}_{n,s}$ such that 
    $M\ot I_n^{\ot (t-1)}=B+\sum_{s=0}^t W_s.$
Taking the projection onto  $\CLDUI^{(t)}_n$ at both sides, we get
  \begin{align*}
  \Pi_{\CLDUI_n}(M)\ot I_n^{\ot (t-1)} & = \Pi_{\CLDUI^{(t)}_n}(M\ot I_n^{\ot (t-1)}) \\
   & =\Pi_{\CLDUI^{(t)}_n}(B)+\sum_{s=0}^t \Pi_{\CLDUI^{(t)}_n}(W_s),
   \end{align*}
    where the first equality follows using    \eqref{eq:DiagOmegat}.
    By Lemma \ref{lem:WsPiCLDUIt}(ii), we have that $\Pi_{\CLDUI^{(t)}_n}(W_s)\in \MW^{(t)}_{n,s}$ for each $s$. 
    To conclude, we use Lemma   \ref{lem:WsPiCLDUIt}(i) to show that $\Pi_{\CLDUI^{(t)}_n}(B)\in\Herm(\C^n\ot S^t(\C^n))^\perp$. Indeed, if $W\in \Herm(\C^n\ot S^t(\C^n))$, then $\Pi_{\CLDUI^{(t)}_n}(W)\in\Herm(\C^n\ot S^t(\C^n))$ and thus we have $\langle  \Pi_{\CLDUI^{(t)}_n}(B),W\rangle =\langle B,\Pi_{\CLDUI^{(t)}_n}(W)\rangle =0$ holds, as desired. %giving the desired result.
 \qed   \end{proof}

%In fact, one can restate the result of Theorem \ref{theo:CLDUItDPSCert} about certificates of membership analogously on the dual side. Let $M\in(\DPSt_n)^*$. By Theorem \ref{theo:FF} the certificate of membership is of the form 
%\begin{align*}
 %   M\ot I_n^{\ot (t-1)}= B+\sum_{s=0}^t W_s,  \ \text{ where  }\  B\in (\MV_n^{(t)})^\perp,\ 
 %   W_s\in \MW^{(t)}_{n,s} \text{ for }s=0,1,\ldots,t.
%\end{align*}
%Then 
%\begin{align*}
%    (\Pi_{\CLDUI_n}(M))\ot I_n^{\ot (t-1)}&=\Pi_{\CLDUI^{(t)}_n}(M\ot I_n^{\ot (t-1)})\\
 %   &=\Pi_{\CLDUI^{(t)}_n}(B)+\sum_{s=0}^t \Pi_{\CLDUI^{(t)}_n}(W_s)
%\end{align*}
%where the first equality follows from relation \eqref{eq:DiagOmegat}. We now conclude by arguing, that $\Pi_{\CLDUI_n^{(t)}}$ preserves membership in $(\MV_n^{(t)})^\perp$, because with Lemma \ref{lem:WsPiCLDUIt} this yields a sparse certificate for $\Pi_{\CLDUI}(M)\in(\DPSt_n)^*$.

%Let $B\in(\MV_n^{(t)})^\perp$ and $A\in\MV_n^{(t)}$. Then $$\langle\Pi_{\CLDUI^{(t)}_n}(B),A\rangle=\langle B,\Pi_{\CLDUI^{(t)}_n}(A)\rangle=0,$$
%since we already observed in the proof of Theorem $\ref{theo:CLDUItDPSCert}$, that projecting onto $\CLDUI_n^{(t)}$ subspace preserves symmetries in the last $t$ registers.

\subsection{Exploiting sparsity of DPS certificates for CLDUI states}\label{sec:DPS-CLDUIt-block}

%Theorem \ref{theo:CLDUItDPSCert} raises the questions, how assuming that a $\DPSt_n$-certificate is in $\CLDUI^{(t)}_n$ can be exploited for testing membership more efficently. This section sets out to answer this question.
As we now see, matrices in the set $\CLDUI^{(t)}_n$ have a block-diagonal structure. 
For convenience, let $G^{(t)}_n$ denote the (undirected) graph with vertex set $[n]^{t+1}$ and edge set $\Omega^{(t)}_n$. 
%Note that $G^{(t)}_n$  is   an undirected graph since  $(i_0\ui,j_0\uj)\in\Omega^{(t)}_n$ implies  $(j_0\uj,i_0\ui)\in\Omega^{(t)}_n$. %We will therefore consider $G^{(t)}_n$ to be an undirected Graph.
So, any matrix  $W\in \CLDUI^{(t)}_n$ has a sparsity pattern that is contained in the graph $G^{(t)}_n$, i.e., $\supp(W)\subseteq \Omega^{(t)}_n$.
An easy but very useful fact is that the graph $G^{(t)}_n$ is a disjoint union of cliques. 
%So, this generalizes what we already saw earlier for the case $t=1$ in equation \eqref{eq:rhoXYZ-block}.

\begin{lemma}\label{lem:Cliques-s=0}
    The graph $G^{(t)}_n=([n]^{t+1}, \Omega^{(t)}_n)$ is a disjoint union of cliques. Hence, any matrix $W\in\CLDUI^{(t)}_n$ is block-diagonal, of the form $W=\oplus_C W[C]$, where the direct sum is over the maximal cliques of $G^{(t)}_n$. 
\end{lemma}
\begin{proof}
%    This follows from the fact that if 
    Assume $i_0\ui$ is adjacent to both  $j_0\uj$ and $k_0\underline k$ in $G^{(t)}_n$, we show that $(j_0\uj)$ is adjacent to $(k_0\uk)$. By assumption,  $\alpha(i_0\uj)=\alpha(j_0\ui)$ and $\alpha(i_0\uk)=\alpha(k_0\ui)$. This implies
     $$\alpha(j_0\uk)-\alpha(k_0\uj)= (\alpha(j_0\ui)-\alpha(i_0\uj))+ (\alpha(i_0\uk)-\alpha(k_0\ui))=0$$%\alpha(i_0\uj)-\alpha(i_0\uk)+\alpha(k_0\ui)=\alpha(j_0\ui)$$
(since $\alpha(\cdot)$ is linear under the concatenation of sequences), 
%showing   %that $\alpha(k_0\uj)=\alpha(j_0\uk)$, i.e., 
which shows that $(j_0\uj,k_0\uk)\in\Omega^{(t)}_n$. %$ is adjacent to $k_0\underline k$.
The second claim follows as an immediate consequence.
\qed\end{proof}

So, if $W\in\CLDUI^{(t)}_n$, then 
%\begin{corollary}
   $W\succeq 0$ if and only if $W[C]\succeq 0$ for all the maximal cliques $C$ of $G^{(t)}_n$. 
%\end{corollary}
This generalizes what we  saw earlier in   \eqref{eq:rhoXYZ-block} for the case $t=1$ .
%This already allows to test whether a given extended state $\rho_{AB[t]}\in\CLDUI^{(t)}_n$ is positive semidefinite more efficiently. 
The analogous result also holds  for testing positive semidefiniteness of partial transposes, i.e., membership in the set $\MW^{(t)}_{n,s}$ for $s\in [t]$.
%Additionally, we have to check, whether the partial transposes of the state are positive semidefinite.

\medskip
Let $1\leq s\leq t$ be an integer. Given a sequence $\ui=i_1\ldots i_t\in[n]^t$ we define the subsequences  $\ui[s]=i_1\ldots i_s\in[n]^s$  and $i[s+1:t]=i_{s+1}\ldots i_t\in[n]^{t-s}$, which are, respectively, the $s$-prefix and  $(t-s)$-suffix of $\ui$. Consider the function
\begin{align*}
    T_{B[s]}:[n]^{t+1}\times[n]^{t+1}&\to[n]^{t+1}\times[n]^{t+1}\\
    (i_0\ui, j_0\uj)&\mapsto(i_0\uj[s]\ui[s+1:t],j_0\ui[s]\uj[s+1:t])
\end{align*}
that swaps the $s$-prefices of $\ui$ and $\uj$. Then, the entry of $W^{T_{B[s]}}$ indexed by the pair $(i_0\ui,j_0\uj)$ is equal to the entry of $W$ indexed by the pair $T_{B[s]}((i_0\ui,j_0\uj))$. Now, we  define  the set
\begin{align*}
    \Omega^{(t)}_{n,s}:=\{(i_0\ui, j_0\uj)\in[n]^{t+1}\times[n]^{t+1}:T_{B[s]}((i_0\ui, j_0\uj))\in\Omega^{(t)}_{n}\}
\end{align*}
and, accordingly, $G_{n,s}^{(t)}$ is the graph with vertex set $[n]^{t+1}$ and edge set $\Omega^{(t)}_{n,s}$. So, the partial transpose $W^{T_{B[s]}}$  in the first $s$ $B$-registers of a matrix $W$ in $\CLDUI^{(t)}_n$ has a sparsity pattern contained in $G_{n,s}^{(t)}$.

\begin{lemma}\label{lem:Cliques-s=1..t}
    For any $s\in [t]$, the graph $G_{n,s}^{(t)}=([n]^{t+1}, \Omega^{(t)}_{n,s})$ is a disjoint union of cliques.
Hence, if $W\in\CLDUI^{(t)}_n$, then $W^{T_{B[s]}}$ is block-diagonal, of the form $W^{T_{B[s]}}=\oplus_C W^{T_{B[s]}}[C]$, where the direct sum is over the maximal cliques of $G^{(t)}_{n,s}$. 
\end{lemma}

\begin{proof}
    Let $i_0\ui$ be adjacent to    $j_0\uj$ and $k_0\underline k$ in $G_{n,s}^{(t)}$. Then,   $\alpha(i_0\ui[s]\uj[s+1:t])=\alpha(j_0\uj[s]\ui[s+1:t])$ and 
    $\alpha(i_0\ui[s]\uk[s+1:t])=\alpha(k_0\uk[s]\ui[s+1:t])$. Hence, 
  \begin{align*}
  &  \alpha(k_0\uk[s]\uj[s+1:t])-\alpha(j_0\uj[s]\uk[s+1:t]) =
   \\
   & \quad (\alpha(k_0\uk[s]\ui[s+1:t])-\alpha(i_0\ui[s]\uk[s+1:t])) \\
  &  \quad   + (\alpha(i_0\ui[s]\uj[s+1:t])-\alpha(j_0\uj[s]\ui[s+1:t]))
    \end{align*}
  is equal to 0, and thus   $j_0\uj$ is adjacent to $k_0\underline k$ in $G^{(t)}_{n,s}$.
\qed \end{proof}

\begin{corollary}\label{cor:BlockDiagCert}
    For a matrix $W\in \CLDUI^{(t)}_n$ and $0\le s\le t$, $W^{T_{B[s]}}\succeq 0$ if and only if $W^{T_{B[s]}}[C]\succeq 0$ for all the maximal cliques $C$ of $G^{(t)}_{n,s}$ (setting $G^{(t)}_{n,0}=G^{(t)}_t$).
\end{corollary}

\begin{example}\label{example:CLDUI-cliques} %new
As an illustration we describe the maximal cliques of the graphs $G^{(t)}_n$ and $G^{(t)}_{n,s}$ for the case $n=3$, $t=2$, $s=1,2$, in which case the graphs have  $n^{t+1}=27$ vertices. 
Denote the adjacency relation in the graph $G^{(2)}_3$ (resp., the graph $G^{(2)}_{3,s}$) with $\simeq$ (resp., with $\stackrel{s}{\simeq}$).
Then, 
$i_0i_1i_2\simeq j_0j_1j_2$   if  $\{i_0\uj\}=\{j_0\ui\}$ as multisets, 
$i_0i_1i_2 \edge j_0j_1j_2$ 
if and only if  $i_0j_1i_2\simeq j_0i_1j_2$, and $i_0i_1i_2 \edgetwo j_0j_1j_2$ if and only if  
$i_0j_1j_2\simeq j_0i_1i_2$.
For instance, $121\simeq 112$,  $111\edge 122$, $112\edgetwo 121$,  and $113\simeq 223$, $123\edge 213$, $123\edgetwo 213$. The maximal cliques of the graphs $G^{(2)}_3$, $G^{(2)}_{3,1}$ and $G^{(2)}_{3,2}$ are listed, respectively, in Tables~\ref{table0}, \ref{table1}, \ref{table2}. One can see that 
$G^{(2)}_3$ and $G^{(2)}_{3,1}$ each  have twelve maximal cliques (including three of size 5, three of size 2, six of size 1), and $G^{(2)}_{3,2}$ has ten maximal cliques (including one of size 6, six of size 3, and three of size 1).

We will return to the case $(n,t)=(3,2)$   in Example \ref{ex:t2n3} of Section \ref{sec:momDPSCLDUI}, where we will use the moment approach  
to reformulate the DPS hierarchy for CLDUI states
in a more compact manner.
\end{example}

\bigskip
\begin{table}[h]
    \centering
    \begin{tabular}{c|c|c}
         Block size & Number of blocks & Indexed by  maximal cliques of $G_3^{(2)}$\\\hline
          &  & \{111, 212, 221, 313, 331\}\\
         5$\times$5 & 3 &\{222, 121, 112, 323, 332\}\\
          & &\{333, 131, 113, 232, 223\}\\\hline
           &  & \{123, 132\}\\
          2$\times$2& 3 & \{213, 231\}\\
          & & \{312, 321\}\\\hline
          1$\times$1 & 6 & \{122\}, \{133\}, \{211\}, \{233\}, \{311\}, \{322\}
    \end{tabular}
     \caption{Block structure of $\rho_{AB[2]}\in\CLDUI^{(2)}_3$}\label{table0}
 \end{table}
 
 \bigskip  
\begin{table}[h]     \centering
    \begin{tabular}{c|c|c}
         Block size & Number of blocks & Indexed by maximal cliques of $G_{3,1}^{(2)}$ \\\hline
          &  & \{111, 122, 133, 212, 313\}\\
         5$\times$5 & 3 & \{222, 211, 233, 121, 323\}\\
          & &\{333, 311, 322, 131, 232\}\\\hline
           &  & \{123, 213\}\\
          2$\times$2& 3 & \{231, 321\}\\
          & & \{312, 132\}\\\hline
          1$\times$1 & 6 & \{112\}, \{113\}, \{221\}, \{223\}, \{331\}, \{332\}
    \end{tabular}
  \caption{Block structure of $\rho_{AB[2]}^{T_B[1]}$ for $\rho_{AB[2]}\in\CLDUI^{(2)}_3$}\label{table1}
  \end{table}
  
  \bigskip   
\begin{table}[h]     \centering
    \begin{tabular}{c|c|c}
         Block size & Number of blocks & Indexed by maximal cliques of $G_{3,2}^{(2)}$ \\\hline
         6$\times$6 & 1 & \{123, 132, 213, 231, 312, 321\}\\\hline
          &  & \{122, 212, 221\}\\
          &  & \{133, 313, 331\}\\
         3$\times$3 & 6 & \{211, 121, 112\}\\
          & &\{233, 323, 332\}\\
          &  & \{311, 131, 113\}\\
          &  & \{322, 232, 223\}\\\hline
          1$\times$1 & 3 & \{111\}, \{222\}, \{333\}
    \end{tabular}
  \caption{Block structure of $\rho_{AB[2]}^{T_B[1:2]}$ for $\rho_{AB[2]}\in\CLDUI^{(2)}_3$}\label{table2}
 \end{table}

\subsection{Generalizations to $\LDUI$ and $\LDOI$ states}

The previous results  for CLDUI states extend in a straightforward way to LDUI and LDOI states. The respective `extended' invariance classes are given by
\begin{align*}%\label{def:LDUI^t}
    \LDUI^{(t)}_n:=\{& \rho_{AB[t]}\in\Herm(\C^n\ot(\C^n)^{\ot t}):\\
    & \rho_{AB[t]}= (U\ot U^{\ot t}) \rho_{AB[t]} (U\ot U^{\ot t})^* \ \forall U\in\MDU_n\},\\
%\end{align*}
%\begin{align*}%\label{def:LDOI^t}
    \LDOI^{(t)}_n:=\{& \rho_{AB[t]}\in\Herm(\C^n\ot(\C^n)^{\ot t}): \\
    & \rho_{AB[t]}= (O\ot O^{\ot t}) \rho_{AB[t]} (O\ot O^{\ot t})^* \ \forall O\in\MDO_n\},
\end{align*}
where in the definition of $\LDUI^{(t)}_n$ one can restrict the invariance under $n-1$ selected matrices (as in Lemma \ref{lem:WsPiCLDUIt}(ii)).
The sparsity patterns of $\LDUI^{(t)}_n$ and $\LDOI^{(t)}_n$ are   captured, respectively, by 
\begin{align}\label{def:Psit}\begin{split}
    \Phi_n^{(t)}:=\{(i_0\ui,j_0\uj)\in[n]^{t+1}\times[n]^{t+1}:\alpha(i_0\ui)=\alpha(j_0\uj)\},\\
    \Psi_n^{(t)}:=\{(i_0\ui,j_0\uj)\in[n]^{t+1}\times[n]^{t+1}:\alpha(i_0\ui j_0\uj)\equiv0\text{ mod }2\},
\end{split}\end{align}
and the respective projections are given by % projection onto $\LDUI^{(t)}_n$ is given by 
\begin{align*}%\label{def:projLDUIt}
    \Pi_{\LDUI^{(t)}_n}(\rho_{AB[t]})=\int_{\MDU_n}(U\ot U^{\ot t})\rho_{AB[t]}(U\ot U^{\ot t})^*dU,\\
%\end{align*}
%and the projection onto $\LDOI^{(t)}_n$ is given by 
%
%\begin{align*}%\label{def:projLDOIt}
    \Pi_{\LDOI^{(t)}_n}(\rho_{AB[t]})=\frac{1}{2^n}\sum_{O\in\MDO_n}(O\ot O^{\ot t})\rho_{AB[t]}(O\ot O^{\ot t})^*.
\end{align*}

We now have all the prerequisites to adapt Theorem \ref{theo:CLDUItDPSCert} for LDUI and LDOI states.
The proof is similar, thus omitted.
\begin{theorem}\label{theo:LDUILDOItDPSCert}
   Assume that  $\rho_{AB}\in\DPSt$ with $\DPS^{(t)}$-certificate $\rho_{AB[t]}$. Then, $\Pi_{\LDUI_n}(\rho_{AB})\in\DPS^{(t)}$ (resp., $\Pi_{\LDOI_n}(\rho_{AB})\in\DPS^{(t)}$) with $\DPS^{(t)}$-certificate $\Pi_{\LDUI^{(t)}_n}(\rho_{AB[t]})$ (resp., $\Pi_{\LDOI^{(t)}_n}(\rho_{AB[t]})$). 
\end{theorem}

%Like Theorem \ref{theo:CLDUItDPSCert},
Again, this theorem has an immediate counterpart on  the dual side for $(\DPSt_n)^*$: the analog of Corollary \ref{cor:CLDUIt-dual} holds replacing $\CLDUI$, respectively,  by $\LDUI$ or $\LDOI$. 

Also the analogs of Lemmas \ref{lem:Cliques-s=0}, \ref{lem:Cliques-s=1..t}  hold:
each of the graphs $([n]^{t+1}, \Psi^{(t)}_n)$ and $([n]^{t+1},\Phi^{(t)}_n)$ is a disjoint union of cliques and the same holds for the graphs attached to taking partial transposes. (The  LDUI case follows from the CLDUI case up to a transposition and the proof in the LDOI case is in fact even easier since it involves a parity argument). In summary, matrices in $\LDUI^{(t)}_n$ or $\LDOI^{(t)}_n$ and their partial transposes all have a block-diagonal structure.
 %Also the block diagonalization under all partial transpositions transfers in a straightforward way to $\LDUI^{(t)}_n$ and $\LDOI^{(t)}_n$.

%\begin{lemma}
%If $W\in\LDUI^{(t)}_n$ or $W\in\LDOI^{(t)}_n$, then  $W^{T_{B[s]}}$ is block-diagonal for each $0\le s\le t$.
%\end{lemma}

%The proof of Lemma \ref{lem:Cliques-s=1..t} actually simplifies in the LDOI case, since $\alpha(i_0\ui j_0\uj)\equiv0\text{ mod }2$ and $\alpha(i_0\ui k_0\uk)\equiv0\text{ mod }2$ immediatly implies, that $\alpha(j_0\uj k_0\uk)\equiv0\text{ mod }2$.

%$\textcolor{red}{Argument for LDUI and CLDUI being the same.}

%We now  illustrate the  gain in matrix size that can be achieved when exploiting the block-diagonal structure of  matrices in $\LDOI^{(t)}_t$. 
\medskip
We now give in Lemma \ref{lem:sizecliqueLDOI} an upper bound on the size of the blocks occurring in a matrix in $\LDOI^{(t)}_n$, by upper bounding the maximum size of  a clique in the graph $([n]^{t+1}, \Psi^{(t)}_n)$. Note that this bound also applies to the maximum clique size in the graph corresponding to taking  partial transposes, because the definition of edges in $\Psi^{(t)}_n$ is `global' in the sense that it depends only on the concatenated sequence $\alpha(i_0\ui j_0\uj)$. In addition, the  bound from Lemma \ref{lem:sizecliqueLDOI} also applies to the  graphs $([n]^{t+1}, \Phi^{(t)}_n)$, $G^{(t)}_n=([n]^{t+1}, \Omega^{(t)}_n)$ (and their analogs corresponding to taking partial transposes) that   model the sparsity patterns of matrices in $\LDUI^{(t)}_n$ and $\CLDUI^{(t)}_n$, because  they  are subgraphs of the graph $([n]^{t+1}, \Psi^{(t)}_n)$.

\begin{lemma}\label{lem:sizecliqueLDOI}
The maximum size of a clique in the graph $([n]^{t+1}, \Psi^{(t)}_n)$ is at most $t! n^{\lceil t/2\rceil}.$
\end{lemma}

We omit the proof    since, later in Lemma \ref{lem:LDOIBlockRatio} of Section \ref{sec:implementation},
we will prove a  similar statement adapted to the more economical moment framework. 
%present the proof of Lemma \ref{lem:LDOIBlockRatio} in section \ref{sec:implementation} which makes a similar statement and relies on the same proof techniques. 
By the above discussion, Corollary \ref{cor:sizeclique} below follows as a direct application of Lemma~\ref{lem:sizecliqueLDOI}.

\begin{corollary}\label{cor:sizeclique}
If  $W\in \LDOI^{(t)}_n$ then, for any  $0\le s\le t$, $W^{T_{B[s]}}$ is a block-diagonal matrix with all its blocks of size at most $t! n^{\lceil t/2\rceil}$.
The same holds if $W\in \LDUI^{(t)}_n$ or $W\in \CLDUI^{(t)}_n$.
\end{corollary}

So, for $t=1$,  the maximum block size is $n$ for matrices in $\LDOI_n$, which matches  relation (\ref{eq:rhoXYZ-block}). 

\subsection{A class of CLDUI examples: exploiting additional structure in DPS certificates}

Here, we investigate a class of CLDUI bipartite quantum states (see Definition \ref{def:rhoaap}) and show {\em analytically} that   their separability is characterized by the DPS hierarchy at level $t=2$. We will use this characterization in Section \ref{sec:runtime} %\ref{sec:implementation}
 to generate test states for the PPT$^2$ conjecture. To be able to analytically analyze $\DPS^{(2)}$-certificates, one needs to exploit additional structural properties of the quantum state at hand (in order to  further reduce the block sizes and the number of variables). 
So, we begin with 
some results showing how additional properties of a quantum state $\rho_{AB}$ (like having a large kernel, or enjoying some symmetry property) are inherited by its $\DPS^{(t)}$-certificates $\rho_{AB[t]}$. These results are of independent interest and apply to general quantum states.% (without unitary diagonal invariance).

\subsubsection*{Properties of the kernel}
As we now see,  any vector in   the kernel  of a   state $\rho_{AB}$ can be used to construct vectors in the kernel of its $\DPS^{(t)}$-certificates. This relies on the partial trace property (\ref{eq:DPS3}) and positivity $\rho_{AB}\succeq 0$.

\begin{lemma}\label{lem:kernel-DPS}
Assume that $\rho_{AB}\in \DPS^{(t)}_n$ with $\rho_{AB[t]}$ as $\DPS^{(t)}$-certificate.
If $w\in \ker \rho_{AB}$, then $w\ot e_{k_2}\ot \ldots \ot e_{k_t}\in \ker \rho_{AB[t]}$ for all
$k_2,\ldots,k_t\in [n]$.
\end{lemma}

\begin{proof}
Since $w\in \ker \rho_{AB}$, we have $\rho_{AB}w=0$, and thus
\begin{align*}
 0=\langle \rho_{AB},ww^*\rangle & = \langle \Tr_{B[2:t]}(\rho_{AB[t]}), ww^*\rangle
=\langle \rho_{AB[t]}, ww^*\ot I_n^{\ot (t-1)}\rangle\\
& =\sum_{k_2,\ldots,k_t=1}^n \langle \rho_{AB[t]}, ww^*\ot e_{k_2}e_{k_2}^* \ot\ldots\ot e_{k_t}e_{k_t}^*\rangle\\
&= \sum_{k_2,\ldots,k_t=1}^n (w\ot e_{k_1}\ot\ldots\ot e_{k_t})^* \rho_{AB[t]} (w\ot e_{k_1}\ot\ldots\ot e_{k_t}).
\end{align*}
As $\rho_{AB[t]}\succeq 0$, each summand is nonnegative and thus equal to 0. From this one obtains that    $0=\rho_{AB[t]} (w \ot e_{k_2}\ot \ldots \ot e_{k_t})$, as desired.
\qed\end{proof}

\subsubsection*{Symmetry properties}

We now turn to exploiting symmetries of the initial quantum state for its $\DPS^{(t)}$-certificates.

First, let us introduce how the permutation group $\Perm_n$ acts on $[n]^t$ and   on $ \Herm((\C^n)^t$. 
Any permutation $\theta \in \Perm_n$ has an entrywise action on $[n]^t$, by setting
$\theta(\ui)=\theta(i_1)\ldots\theta(i_t)$ for $\ui=i_1\ldots i_t\in [n]^t$.
 In turn, $\theta$ acts on  $\rho\in \Herm((\C^n)^t$, by permuting its rows and columns according to the action of $\theta$ on $[n]^t$, thus producing  $\theta(\rho)\in \Herm((\C^n)^\ot t)$, with entries
 $$\theta(\rho)_{i_0\ui,j_0\uj}= \rho_{\theta(i_0) \theta(\ui),\theta(j_0)\theta(\uj)} \text{ for }
 i_0\ui,j_0\uj\in [n]^{t+1}.
 $$
 %  $$\theta(\rho)_{i_0i_1\dots i_t,j_0j_1\dots j_t}=\rho_{\theta(i_0)\theta(i_1)\dots \theta(i_t),\theta(j_0)\theta(j_1)\dots \theta(j_t)} \text{ for } i_0,j_0,\ldots,i_t,j_t\in [n].$$

\begin{remark}
%Let us compare with the action of  $\sigma\in \Perm_t$ that was earlier considered. 
So, a permutation $\theta\in\Perm_n$ acts on the values $i_k\in [n]$ and {\em not} on their indices  $k\in [t]$, as  was the case earlier, when a permutation $\sigma\in \Perm_{t}$ would produce 
$ \ui^\sigma=i_{\sigma(1)}\ldots i_{\sigma(t)}\in [n]^t$ and $\rho^\sigma\in\Herm((\C^n)^{\ot t})$ with entries
$\rho^\sigma_{\ui,\uj}=\rho_{\ui^\sigma,\uj^\sigma}$
%$ \rho^\sigma_{i_1\dots i_t,j_1\dots j_t} =
%\rho_{i_{\sigma(1)},\ldots,i_{\sigma(t)}, j_{\sigma(1)},\ldots, j_{\sigma(t)}}$ 
for $\ui,\uj\in[n]^t$. We will use the  following observation (that follows using the definitions): for a sequence $\ui=i_1\ldots i_t\in [n]^t$,
\begin{align}\label{eq:theta-sigma}
\theta(\ui^\sigma) =\theta(i_{\sigma(1)})\ldots\theta(i_{\sigma(t)})= (\theta(i_1)\ldots\theta(i_t))^\sigma=\theta(\ui)^\sigma \text{ for any } \theta\in\Perm_n, \sigma\in \Perm_t.
\end{align}
\end{remark}

%Note that, in contrast to earlier definitions, the permutation group acts on the value of $i_k$ and not on the index $k$. 
%In other words, $\theta\in \Perm_n$ induces a permutation of $[n]^{t+1}$, again denoted $\theta$, by 
%In other words, $\rho^\theta$ corresponds to permuting the rows and columns of $\rho$ at the same time according to the permutation of $[n]^{t+1}$:  $(i_0,i_1,\dots,i_t)\in [n]^{t+1} \mapsto(\theta(i_0),\theta(i_1),\dots,\theta(i_t))\in [n]^{t+1}$, again denoted as $\theta$ for simple notation.
%.\begin{align*}
 %   \theta^{(t)}:[n]^{t+1}&\to[n]^{t+1}\\
%    (i_0,i_1,\dots,i_t)&\mapsto(\theta(i_0),\theta(i_1),\dots,\theta(i_t)).
%\end{align*}
%We will also   write $\theta(i_0\ui)$ or $\theta(i_0)\theta(\ui)$ for $\theta((i_0,i_1,\dots,i_t))$.

\begin{lemma}\label{LThetaAppliedDPSLCertificate}
 Assume $\rho_{AB}\in \DPS^{(t)}_n$ with $\rho_{AB[t]}$ as $\DPS^{(t)}$-certificate.
 %\footnote{\ML{It does not seem that the condition $\rho_{AB}=\theta(\rho_{AB})$ is needed here.}}
For any $\theta\in \Perm_n$,  $\theta(\rho_{AB[t]}) $ is a $\DPS^{(t)}$-certificate for $\theta(\rho_{AB})$.
%, which implies $\theta(\rho_{AB})\in\DPS^{(t)}_n$.
 %    Let $\theta\in Perm(n)$, $\rho_{AB}\in\DPS^{(t)}$ with certificate $\rho_{AB[t]}$ and $\rho_{AB}^\theta=\rho_{AB}$. Then $\rho_{AB[t]}^\theta$ is also a certificate. 
\end{lemma}

\begin{proof}
  By assumption,   $\rho_{AB[t]}$ satisfies (\ref{eq:DPS1}),  (\ref{eq:DPS2}), (\ref{eq:DPS3}), we show that the same holds for $\theta(\rho_{AB[t]})$.
First, we show (\ref{eq:DPS1}). For this,    let $\sigma\in \Perm_t$ and $i_0\ui,j_0\uj\in [n]^{t+1}$. Then, we have
    \begin{align*}
        \theta(\rho_{AB[t]})_{i_0\ui^\sigma,j_0\uj^\sigma}&=(\rho_{AB[t]})_{\theta(i_0)\theta(\ui^\sigma),\theta(j_0)\theta(\uj^\sigma)}
        =(\rho_{AB[t]})_{\theta(i_0)(\theta(\ui))^\sigma,\theta(j_0)\theta(\uj)^\sigma}\\
        &=(\rho_{AB[t]})_{\theta(i_0)\theta(\ui),\theta(j_0)\theta(\uj)}
        =\theta(\rho_{AB[t]})_{i_0\ui,j_0\uj},
    \end{align*}
using (\ref{eq:theta-sigma}) at the first line and the fact that $\rho_{AB[t]}$ satisfies \eqref{eq:DPS1} at the second line. Hence, $\theta(\rho_{AB[t]})$ also satisfies \eqref{eq:DPS1}.

    Condition \eqref{eq:DPS2} follows from the fact that the operations of applying $\theta$ and taking the partial transpose commute:
  %  \begin{align*}
       $ \left(\theta(\rho_{AB[t]})\right)^{T_{B[s]}}=\theta\big(\rho_{AB[t]}^{T_{B[s]}}\big),$
%    \end{align*}
 which can be checked by computing
    \begin{align*}
    &    \big( \left(\theta(\rho_{AB[t]})\right)^{T_{B[s]}}\big)_{i_0\ui,j_0\uj} =
        \left(\theta(\rho_{AB[t]})\right)_{i_0j_1\dots j_si_{s+1}\dots s_t,j_0i_1\dots i_sj_{s+1}\dots j_t}\\
        &=(\rho_{AB[t]})_{\theta(i_0j_1\dots j_si_{s+1}\dots s_t),\theta(j_0i_1\dots i_sj_{s+1}\dots j_t)}
        =\big(\rho_{AB[t]}^{T_{B[s]}}\big)_{\theta(i_0\ui),\theta(j_0\uj)}\\
       & =\big(\theta\big(\rho_{AB[t]}^{T_{B[s]}}\big)\big)_{i_0\ui,j_0\uj}.
    \end{align*}
 Hence, $\rho_{AB[t]}^{T_{B[s]}}\succeq 0$ implies $\theta(\rho_{AB[t]}^{T_{B[s]}})\succeq 0$ and thus
  $(\theta(\rho_{AB[t]}))^{T_{B[s]}}\succeq 0$, as desired. 
      
       Finally, condition \eqref{eq:DPS3} follows from the fact that the operations of applying $\theta$ and taking the partial trace also commute, which 
       can be checked   by computing
    \begin{align*}
   &     \big(\Tr_{B[2:t]}\big(\theta(\rho_{AB[t]})\big)\big)_{i_0i_1,j_0j_1}
%        &=\sum_{k_2,\dots,k_t=1}^{n}\big(\theta(\rho_{AB[t]})\big)_{i_0i_1k_2\dots k_t,j_0j_1k_2\dots k_t}\\
  =\sum_{\uk\in [n]^{t-1}} \big(\theta(\rho_{AB[t]})\big)_{i_0i_1\uk, j_0j_1\uk}\\
 %       &=\sum_{k_2,\dots,k_t=1}^{n}\big(\rho_{AB[t]}\big)_{\theta(i_0i_1k_2\dots k_t),\theta(j_0j_1k_2\dots k_t)}\\
  &= \sum_{\uk\in [n]^{t-1}} \big(\rho_{AB[t]}\big)_{\theta(i_0i_1)\theta(\uk),\theta(j_0j_1)\theta(\uk)}
  %       &=\sum_{h_2,\dots,h_t=1}^{n}\big(\rho_{AB[t]}\big)_{\theta(i_0i_1)h_2\dots h_t,\theta(j_0j_1)h_2\dots h_t}\\
  =\sum_{\uh\in [n]^{t-1}} \big(\rho_{AB[t]}\big)_{\theta(i_0i_1)\uh, \theta(j_0j_1)\uh} \\
    &      =\Tr_{r-2}(\rho_{AB[t]})_{\theta(i_0i_1),\theta(j_0j_1)}=(\rho_{AB})_{\theta(i_0i_1),\theta(j_0j_1)} 
           =\theta(\rho_{AB})_{i_0i_1,j_0j_1}, %=(\rho_{AB})_{i_0i_1,j_0j_1}
    \end{align*}
using the fact that $\rho_{AB[t]}$ satisfies (\ref{eq:DPS3}) at the last but one equality.\qed \end{proof}

\begin{corollary}\label{CThetaInvDPSCert}
 Assume $\rho_{AB}\in\DPS^{(t)}$ and $\theta(\rho_{AB})=\rho_{AB}$ for some $\theta\in\Perm_t$. 
 Then,  $\rho_{AB}$ has  a $\DPS^{(t)}$-certificate $\rho_{AB[t]}$ that satisfies 
  $\theta(\rho_{AB[t]})=\rho_{AB[t]}$.
\end{corollary}

\begin{proof}
%    First note, that given $\DPS^{(t)}$ certificates $\rho_{AB[t]}^{(1)},\dots,\rho_{AB[t]}^{(m)}$ for $\rho_{AB}$, then 
%    $$\rho_{AB[t]}:=\frac{1}{m}\sum_{k=1}^m\rho_{AB[t]}^{(k)}$$
 %   is also a $\DPS^{(t)}$ certificate for $\rho_{AB}$. 
% We use Lemma  \ref{LThetaAppliedDPSLCertificate}, combined with the fact that the set of $\DPS^{(t)}$-certificates for   $\rho_{AB}$ is a convex set, to define a   $\DPS^{(t)}$-certificate that is invariant under $\theta$. For this, l
Let $\rho_{AB[t]}$ be a $\DPS^{(t)}$-certificate for  $\rho_{AB}\in\DPS^{(t)}$. 
 Then, for any integer $k\ge 1$, $\theta^k(\rho_{AB[t]})$ is a $\DPS^{(t)}$-certificate for $\theta^k(\rho_{AB})=\rho_{AB}$, which follows using Lemma \ref{LThetaAppliedDPSLCertificate} and the fact that $\theta(\rho_{AB})=\rho_{AB}$. Let $m\ge 1$ be an integer such that $\theta^m=id$. Since the set of $\DPS^{(t)}$-certificates for   $\rho_{AB}$ is a convex set,
 one finds that  
    $\widetilde\rho_{AB[t]}:=\frac{1}{m}\sum_{k=1}^m\theta^k(\rho_{AB[t]})$
is a $\DPS^{(t)}$-certificate for $\rho_{AB}$. Moreover, it satisfies 
$\theta(\widetilde\rho_{AB[t]})= \widetilde\rho_{AB[t]}$.
%,   which concludes the proof.
%
%Lemma \ref{LThetaAppliedDPSLCertificate} implies, that every summand of $\tilde\rho_{AB[t]}$ is a $\DPS^{(t)}$ certificate, which gives with the above observation, that $\tilde\rho_{AB[t]}$ is a $\DPS^{(t)}$ certificate for $\rho_{AB}$. Finally, $\tilde\rho_{AB[t]}$ is clearly $\theta$ invariant, since apllying $\theta$ amounts to reordering the summation.
\qed\end{proof}

\subsubsection*{A class of CLDUI bipartite quantum states}\label{sec:example}

We  consider the   CLDUI bipartite states $\rho_{a,a'}$ ($a,a'\ge 0$) in Definition \ref{def:rhoaap},   generalizing  the case  $a'=1/a$ ($a>0$) considered in  \cite[Example 3.3]{JML-PCP}.
%,  where it shown that  $\rho_{a,1/a}$ is separable if and only if $a=1$. 
We  show 
 that $\rho_{a,a'}$ is separable if and only if it belongs to $\DPS^{(2)}_3$,  or, equivalently,  $aa'\ge 1$   (Theorem \ref{theo:rhoaa}). Analytically characterizing membership in $\DPS^{(2)}_3$ is the most technical part, 
 which exploits knowing the maximal cliques of the graph $G^{(2)}_3$ (Table \ref{table0}) and  structural properties of the state $\rho_{a,a'}$ (on symmetry and  its kernel).
We will use this class of examples later in Section~\ref{sec:runtime}.

\begin{definition}\label{def:rhoaap}
For $a,a'\in\R_+$, define the quantum state $\rho_{a,a'}\in\CLDUI_3$ as
  %  \begin{align*}
$$        \rho_{a,a'}=
  \left(  \begin{array}{ccc|ccc|ccc}
              1 & \cdot & \cdot & \cdot & 1 & \cdot & \cdot & \cdot & 1\\
          \cdot & a & \cdot & \cdot & \cdot & \cdot & \cdot & \cdot & \cdot\\
          \cdot & \cdot & a' & \cdot & \cdot & \cdot & \cdot & \cdot & \cdot\\
          \hline 
          \cdot & \cdot & \cdot & a' & \cdot & \cdot & \cdot & \cdot & \cdot\\
          1 & \cdot & \cdot & \cdot & 1 & \cdot & \cdot & \cdot & 1\\
          \cdot & \cdot & \cdot & \cdot & \cdot & a & \cdot & \cdot & \cdot\\
          \hline
          \cdot & \cdot & \cdot & \cdot & \cdot & \cdot & a & \cdot & \cdot\\
          \cdot & \cdot & \cdot & \cdot & \cdot & \cdot & \cdot & a' & \cdot\\
          1 & \cdot & \cdot & \cdot & 1 & \cdot & \cdot & \cdot & 1
     \end{array}
    \right)\ = \ \rho_{(X,Y)}  \text { with } 
     X=\begin{pmatrix}
            1&a&a'\\
            a'&1&a\\
            a&a'&1
        \end{pmatrix},
        Y=\begin{pmatrix}
            1&1&1\\
            1&1&1\\
            1&1&1
        \end{pmatrix},
            $$
%    \end{align*}
where $\rho_{a,a'}$ is indexed by $11,12,13,21,22,23,31,32,33$, $X=((\rho_{a,a'})_{ij,ij})_{i,j=1}^3$,  $Y=((\rho_{a,a'})_{ii,jj})_{i,j=1}^3$.
\end{definition}
%Example 1 in \cite{JML-PCP} is recovered by setting $a'=\frac{1}{a}$.

\begin{lemma}\label{lem:rhoaaDPS1}
    The matrix $\rho_{a,a'}$ belongs to $\DPS^{(1)}_3$ if and only if $a,a'\geq0$ and $a a'\geq1$.
\end{lemma}

\begin{proof}
  %  Let $\rho_{a,a'}\in\DPS^{(1)}_3$. Then, $a,a'\geq0$ follows from $\rho_{a,a'}$ being positive semidefinite. The only other condition for $\rho_{a,a'}\in\DPS^{(1)}$ is, that the partial transpose of $\rho_{a,a'}$ is positive semidefinite. This is equivalent to $aa'-1\geq0$. The other direction follows immediately from the same arguments.
 By definition,  $\rho_{a,a'}\in \DPS^{(1)}_3$ if and only if $\rho_{a,a'}\succeq 0$ and $\rho_{a,a'}^{T_B}\succeq 0$. This is equivalent to $a,a'\ge 0$ and the block {\tiny $\begin{pmatrix} a&1\cr 1 &a'\end{pmatrix}$} being positive semidefinite, i.e., $aa'\ge 1$.
\qed\end{proof}

\begin{lemma}\label{LExBlocksAreScalar}
Assume that $\rho_{a,a'}\in\DPS^{(2)}_3$. Then, $\rho_{a,a'}$ has  a $\DPS^{(2)}$-certificate $\rho$ with the following properties: $\rho$ is block-diagonal and each of its blocks is a scalar multiple of the all-ones matrix.
%, for every block, there exists one scalar such that every entry of that block is equal to that scalar. 
\end{lemma}

\begin{proof}
    The existence of a block-diagonal certificate $\rho$ follows from Theorem \ref{theo:CLDUItDPSCert} and Corollary \ref{cor:BlockDiagCert}. So,  $\rho=\oplus_C \rho[C]$, where the direct sum is over the maximal cliques $C$ of the graph $G^{(2)}_3$ (described in Table \ref{table0}). We now argue that, for  every maximal clique $C$, $\rho[C]=x_C J_{|C|}$ for some   $x_C\in\R$. That is, the columns $\rho_{\cdot, i_0i_1i_2}$, $\rho_{\cdot,j_0j_1j_2}$ of $\rho$ indexed by any two elements $i_0i_1i_2,j_0j_1j_2\in C$ are equal. It suffices to show this for the 
    six cliques of size 2 or 5.
    %three cliques of size 5, including the clique $C=\{111,212,221,313,331\}$.
    
 %   As  explained in Remark \ref{remark:CLDUI-cliques}, if two indices form an edge of type 1, the corresponding rows and columns of $\rho$ are identical (since $\rho$ satisfies (\ref{eq:DPS1})). Hence, each clique $C$ of $G^{(2)}_3$ can be replaced by a reduced clique $C'\subseteq C$ that contains exactly one element out of every clique of type 1 edges. That is,  $\rho[C]\succeq 0$ if and only if $\rho[C']\succeq 0$. E.g.,  $C'=\{111,221,331\}$ for the clique  $C=\{111,212,221,313,331\}$.

%This applies analogously to the other cliques $C_2=\{\}$ and $C_3=\{\}$ of $G^{(2)}_3$, with $C'_2=\{\}$, $C_3'=\{\}$, so that $\rho\succeq 0$ is equivalent to requiring positive semidefiniteness of three blocks of size 3 and 12 singleton blocks.

% Then, As explained already in Example \ref{example:CLDUI-cliques}, the columns and rows of this block indexed by 221 and 212 are identical since $\rho$ satisfies \eqref{eq:DPS1} and therefore $\rho_{\cdot,221}=\rho_{\cdot,212}$ and $\rho_{221,\cdot}=\rho_{212,\cdot}$. Further note, that this argument works for all Type I neighbors, which implies that also the rows and columns indexed by 331 and 313 are identical.

First, we use the property (\ref{eq:DPS1}) of $\rho$, which claims that  $\rho_{\cdot, j_0j_1j_2}=\rho_{\cdot, j_0j_2j_1}$ for any $j_0j_1j_2\in[3]^3$. This directly settles the case of the cliques of size 2,  $\{123,132\}$, $\{213,231\}$ and $\{312,321\}$.
%So, $\rho[C]$ is indeed a multiple of $J_2$ for each of the three cliques of size 2.
Consider now a clique of size 5, say $C=\{111,212,221,313,331\}$. Then, 
$\rho_{\cdot, 212}=\rho_{\cdot, 221}$, $\rho_{\cdot, 313}=\rho_{\cdot,331}$ holds.  There remains to show that the columns indexed by the elements $111,221,331$ are equal.
%   To show that each of the remaining $3\times 3$ blocks $\rho[C']$ is a multiple of the all-ones matrix $J_3$, 
For this,    we use the fact that %$\rho_{a,a'}$ has a rich kernel: %and  Lemma \ref{lem:kernel-DPS}. 
the vectors $w_1=e_2\ot e_2-e_3\ot e_3$ and $w_2=e_1\ot e_1-e_3\ot e_3$  belong to the kernel of $\rho_{a,a'}$. By Lemma \ref{lem:kernel-DPS},
   it follows that $w\ot e_i\in \ker \rho$ %$ belongs to the kernel of $\rho$ 
   for each $i\in [3]$ and $w\in\{w_1,w_2\}$.
From this one obtains that the columns of $\rho$ indexed by the indices $111, 221, 331$ are all identical, as desired.
%which shows that the block indexed by 
%   $\{111,221,331\}$ is a multiple of $J_3$, as desired.
The same reasoning applies to the other two cliques of size 5.
%   In the same way, the block indexed by $\{222,332,112\}$ is a multiple of $J_3$ as well as the block indexed by  $\{333,113,223\}$. 
   %   Now consider the rows and columns indexed by 111, 221 and 331. Define $u_1\in\C^3\otimes\C^3$ such that $(u_1)_{11}=1,(u_1)_{22}=-1$ and 0 everywhere else and $u_2\in\C^3\otimes\C^3$ such that $(u_2)_{11}=1,(u_2)_{33}=-1$ and 0 everywhere else. Clearly $p_{a,a'}u_1=0=p_{a,a'}u_2$. By Lemma \ref{LKernelDPSCert} it follows, that  $u_1\otimes e_1$ and $u_2\otimes e_1$ are in the kernel of $\rho$. Therefore, the three columns indexed by 111, 221 and 331 are identical. Since $\rho
   % $ is hermitian it also follows, that the rows are identical.   
%    The argument runs analogous for the other two blocks of size 5, the blocks of size 2 consists of neighbors of Type I and for the blocks of size 1 nothing needs to be shown.
\qed\end{proof}

Next, we further simplify the structure of a $\DPS^{(2)}$-certificate for $\rho_{a,a'}$, by exploiting the fact that $\rho_{a,a'}$ is invariant under the cyclic permutation $\theta=(123)\in\Perm_3$, i.e., $\theta(\rho_{a,a'})=\rho_{a,a'}$.

\begin{lemma}\label{temp4}
 Assume that $\rho_{a,a'}\in\DPS^{(2)}_3$. Then,  $\rho_{a,a'}$ has a $\DPS^{(2)}$-certificate $\rho$ with the following properties: $\rho$ is block-diagonal: $\rho=\oplus_C \rho[C]$, where the blocks are 
 indexed by the maximal cliques $C$ of $G^{(2)}_3$, and 
     \begin{itemize}
        \item [(i)] the three  blocks of size 5 are identical: if $|C|=5$, then $\rho[C]=xJ_5$  for some   $x\in\R$,
        \item [(ii)] the three blocks of size 2 are identical: if $|C|=2$, then $\rho[C]=y J_2$ for some  $y\in\R$,
        \item [(iii)] $\rho_{122,122}=\rho_{233,233}=\rho_{311,311}=:z\in\R$,
        \item [(iv)] $\rho_{133,133}=\rho_{211,211}=\rho_{322,322}=:z'\in\R$.
    \end{itemize}
\end{lemma}

\begin{proof}
 The additional properties in (i)-(iv) follow using the fact that $\rho_{a,a'}$ is invariant under the cyclic permutation $\theta=(123)$. Indeed, by  Corollary \ref{CThetaInvDPSCert}, $\rho_{a,a'}$ has  a  $\theta$-invariant $\DPS^{(2)}$-certificate $\rho\in \CLDUI^{(2)}_3$. We show that  this certificate $\rho$ has the desired properties. Observe that  $C_1:=\{111,212,221,313,331\}$, $C_2:=\theta(C)$, $C_3:=\theta^2(C)$ are the three cliques of size 5 in $G^{(2)}_{3}$ (where $\theta$ is applied element-wise).  
  First, we have $$\rho_{111,111}=\rho_{\theta(111),\theta(111)}=\rho_{222,222}=\rho_{\theta(222),\theta(222)}=\rho_{333,333}.$$
For $i=1,2,3$, the index $iii$ belongs to the clique $C_i$.  By  Lemma \ref{LExBlocksAreScalar}, all entries of $\rho[C_i]$ are equal, say to $x_i$, implying   $x_i=\rho_{iii,iii}$. Combined with the above relation, one obtains $x_1=x_2=x_3=:x$, showing (i).

% is a multiple of the all-ones matrix. Hence, one can con   Since the blocks containing the diagonal entries of 111,222 and 333 are all given by a single scalar (Lemma \ref{LExBlocksAreScalar}) the first item of the claim follows.
    
  (ii) follows   analogously,    using the fact that  $\theta(123)=231$ and $\theta(231)=312$, which implies that the scalars determining the blocks $\rho[C_{23}], \rho[C_{31}],\rho[C_{12}]$ are all equal to the same scalar $y$. 
%  the item regarding the blocks of size 3 follows analogously. 
 
For (iii)-(iv),  note that the six blocks of size 1 lie in two different orbits: $\theta(122)=233,\theta(233)=122$ and  $\theta(133)=211,\theta(211)=322$., which implies $\rho[122]=\rho[233]=\rho[122]=:z$ and $\rho[133]=\rho[211]=\rho[322]=:z'$, as desired.
\qed\end{proof}

\begin{theorem}\label{theo:rhoaa}
 For the   state $\rho_{a,a'}$ from Definition \ref{def:rhoaap}, the following are equivalent:
 $$\rho_{a,a'}\in \SEP_3\Longleftrightarrow \rho_{a,a'}\in\DPS^{(2)}_3\Longleftrightarrow a\ge 1 \text{ and } a'\ge 1.$$
 %   \begin{enumerate}
 %       \item[(i)] $\rho_{a,a'}\in\SEP_3$.
%        \item[(ii)] $\rho_{a,a'}\in\DPS^{(2)}_3$. 
 %       \item[(iii)] $a,a'\geq1$.    \end{enumerate}
\end{theorem}
\begin{proof}
 %   (ii) $\Longrightarrow$ (iii): 
 Assume $\rho_{a,a'}\in\DPS^{(2)}_3$, we show $a,a'\ge 1$. Let $\rho$ be a $\DPS^{(2)}$-certificate for $\rho_{a,a'}$ that satisfies the properties claimed in Lemma \ref{temp4}.
  %   there is a $\DPS^{(t)}$ certificate $\rho$, which satisfies the properties of the lemma. 
     Since $\rho$ satisfies \eqref{eq:DPS3}, we get
    $$1=(\rho_{a,a'})_{11,11}=\Tr_B(\rho)_{11,11}=\rho_{111,111}+\rho_{112,112}+\rho_{113,113}=3x,$$
   which implies    $x=\frac{1}{3}$. Similarly, we get
 \begin{align*}
 a=\Tr_B(\rho)_{12,12}=x+z+y,\quad
a'=\Tr_B(\rho)_{21,21}=y+z'+x,
\end{align*}
 which implies      $z=a-y-\frac{1}{3}$ and $z'=a'-y-\frac{1}{3}$. 
    
    We now use the fact that $\rho^{T_{B[1]}}$ is positive semidefinite (condition \eqref{eq:DPS2}). 
    Using the submatrix of $\rho^{T_{B[1]}}$ indexed by $\{123, 213\}$ (in that order), we obtain
         %. The submatrix is given by
    $$
  \begin{pmatrix}
  \rho_{123,123} & \rho_{113,223}\cr \rho_{223,113}& \rho_{213,213}
  \end{pmatrix}  
  =
    \begin{pmatrix}
        y&x\\
        x&y
    \end{pmatrix}\succeq 0,$$
 implying      $y\geq x=\frac{1}{3}$. Using the submatrix of $\rho^{T_{B[1]}}$  indexed by $\{111,122, 133\}$ (in that order) gives
     $$
    \begin{pmatrix}
    \rho_{111,111} & \rho_{121,112} & \rho_{131,113}\cr
    \rho_{112,121} & \rho_{122,122} & \rho_{132,123}\cr
    \rho_{113,131} & \rho_{123,132} & \rho_{133,133}
    \end{pmatrix}    
    =
    \begin{pmatrix}
        x&x&x\\
        x&z&y\\
        x&y&z'
    \end{pmatrix}. $$ %\succeq 0.$$
    By taking the Schur complement w.r.t. the upper left corner and using the values of $z,z'$, we obtain
    $$\begin{pmatrix}
    z-x&y-x\cr y-x&z'-x\end{pmatrix}= 
    \begin{pmatrix}
        a-y-\frac{2}{3}&y-\frac{1}{3}\\
        y-\frac{1}{3}&a'-y-\frac{2}{3}
    \end{pmatrix}\succeq 0.$$
  In particular, $a-y-{2\over 3}, a'-y-{2\over 3}\ge 0$.   Combining with  $y\geq\frac{1}{3}$, we get that $a,a'\geq1$, as desired.
    
%    (iii) $\Longrightarrow$ (i): 
Assume now $a,a'\ge 1$, we show that $\rho_{a,a'}\in\SEP_3$. For this, observe that  the matrices $X,Y$ in Definition \ref{def:rhoaap} satisfy 
    $X=J_3+X'$, $Y=J_3$, where $X'=X-J\ge 0$ since $a,a'\ge 1$. 
    So, we have 
    $\rho_{a,a'}=\rho_{(J_3,J_3)}+\rho_{(X',0)}$.
       % $$X=\begin{bmatrix}0&a-1&a'-1\\
    %        a'-1&0&a-1\\
   %         a-1&a'-1&0 \end{bmatrix}, \ Y=\begin{bmatrix}
     %       0&0&0\\
     %       0&0&0\\
     %       0&0&0
    %    \end{bmatrix}.$$
The   state $\rho_{(J_3,J_3)}$ is   separable (e.g., since it is equal to $\Pi_{\CLDUI_3}(ee^*\ot ee^*)$, with $e$ the all-ones vector).
%because $J_3\in\CP_3$, so that $(J_3,J_3)\in\PCP_3$, using Lemma \ref{lem:CP-PCP-TCP}).
Also, the state $\rho_{(X',0)}$ is separable (using, e.g., (\ref{eq:rhoXYdef})).
%because $(X',0)\in\PCP_3$, by   \cite[Lemma 1]{JML-PCP}).
 %   Now, $\rho_{a,a'}$ is separable as the sum of two separable states. Indeed, $\rho_{J,J}$ is separable, because $(J,J)\in\PCP$ if and only if  $J\in\CP$ (by Lemma \ref{lem:CP-PCP-TCP}) and the all-ones matrix $J$ is clearly completely positive. Moreover, $\rho_{X,Y}$ is separable, because $(X,Y)\in\PCP$ (by Lemma 1 in \cite{JML-PCP}).
Hence, the state $\rho_{a,a'}$ is the sum of two separable states and thus $\rho_{a,a'}$ is separable.
    This concludes the proof. % (i)$\Longrightarrow$ (ii) is clear.
 %   (3) $\Longrightarrow$  (1) follows immediately, since $\SEP\subseteq\DPS^{(t)}$ for every integer $t$.
\qed\end{proof}

\begin{remark}
In their recent work \cite{GNS-cyclic_2025} the authors consider LDOI states $\rho_{(X,Y,Z)}$, where   $X,Y,Z$ are circulant matrices. In particular, when $Y=J_n$ is the all-ones matrix, they characterize membership in $\SEP_n$ and $\DPS^{(1)}_n$ (see Theorem VI.5). Note that the state $\rho_{a,a'}$ from Definition~\ref{def:rhoaap}
has $X$ circulant and $Y=J_n$, so it falls within this class (for $n=3$). In Theorem \ref{theo:rhoaa} we additionally characterize its membership in $\DPS^{(2)}_3$ and show equivalence with separability. Whether this property extends to more general states  with circulant structure is left for future research.
\end{remark}

 %\label{sec:DPS-CLDUI}

%\section{Strengthened DPS hierarchy for symmetric states}\label{sec:DPS-symmetric}
\section{Dual of the symmetric DPS hierarchy and matrix copositivity}\label{sec:DPS-symmetric}

%begin ignore
\ignore{
We begin with   some easy  facts about Bose symmetric bipartite states and their connections to CLDUI and LDUI states.
Recall from Example \ref{ex:Dicke} that any mixed Dicke state (i.e., any linear combination  of the Dicke states   $D_{ij}D_{ij}^*$ for $1\le i\le j\le n$)  is Bose symmetric and has LDUI sparsity structure. This is in fact an equivalence.    (\ML{as observed in \cite{GNP_2025}, check if true}).
  
 \begin{lemma}\label{lem:mixedDicke}
Consider $\rho_{AB}\in\Herm(\C^n\ot \C^n)$. The following assertions are equivalent.
\begin{itemize}
\item[(i)] $\rho_{AB}\in \LDUI_n \cap \BS(\C^n\ot\C^n)$.
\item [(ii)] $\rho_{AB}=\rho_{(X,\cdot,X)}$ for some $X\in\MS^n$.
\item[(iii)] $\rho_{AB}$ is a mixed Dicke state, i.e., $\rho_{AB}\in \R\{D_{ij}D_{ij}^*: 1\le i\le j\le n\},$ where $D_{ii}=e_i\ot e_i$ and
$D_{ij}=(e_i\ot e_j +e_j\ot e_i)/\sqrt 2$ ($i\ne j$).
\end{itemize}
\end{lemma}
\begin{proof}
It suffices to show (i) $\Longrightarrow$ (ii) $\Longrightarrow$ (iii). Assume (i) holds. As $\rho_{AB}$ is LDUI,  $\rho_{AB}= \rho_{(X,\cdot,Z)} $ for some $X\in\R^{n\times n}, Z\in \Herm^n$. Moreover, $X=Z$ holds since $\rho_{AB}$ is Bose symmetric. So, (ii) holds.
To see (iii), note that
$\rho_{(X,\cdot,X)}= \sum_{i=1}^nX_{ii} D_{ii}D_{ii}^* +2\sum_{1\le i<j\le n}  X_{ij}D_{ij}D_{ij}^*$.
\end{proof}
}
% end ignore

Given a matrix $A\in\MS^n$, consider the matrix $M_{A,0}\in \CLDUI_n$
%$$M_{A,0}=\sum_{i,j=1}^n A_{ij} e_ie_j^*\ot e_je_i^* \in \CLDUI_n,$$
as defined in relation (\ref{eq:MST}) and the associated polynomial 
$$\langle M_{A,0}, xx^*\ot xx^*\rangle 
%= \sum_{i,j=1}^n A_{ij}\langle  e_ie_j^*\ot e_je_i^*,xx^*\ot xx^*\rangle
=\langle A, (x\circ \ovx)(x\circ \ovx)^T\rangle= \sum_{i,j=1}^n A_{ij} |x_i|^2|x_j|^2 = f_A(|x_1|^2,\ldots,|x_n|^2),
$$
where $f_A$ is the polynomial from (\ref{eq:fA}). Hence, 
%if $A\in \COP_n$, then the matrix $M_{A,0}$ lies in the dual cone of $\SEPBS_n$. 
$$M_{A,0}\in (\SEPBS_n)^* \Longleftrightarrow A\in \COP_n.$$
Moreover,  if $A_{ij}<0$ for some $i<j$, then $\langle M_{A,0}, D_{ij}D_{ij}^*\rangle = A_{ij}<0$; hence  $M_{A,0}$ provides an entanglement witness that  separates the entangled state $D_{ij}D_{ij}^*$ from $\SEPBS_n$. 
%These connections were made already  in   \cite[Th. 2.1]{Marconi_2021}. 
% and they suggest investigating whether they extend to the hierarchies for the dual of $\SEPBS_n$ and for $\COP_n$.
So, when restricting to matrices with diagonal unitary invariance, there is a tight connection between the dual cone $(\SEPBS_n)^*$ and the copositive cone $\COP_n$, as   shown  in   \cite[Theorem 2.1]{Marconi_2021}. 
Recently, it is shown in \cite{GNP_2025}  that this connection extends to  the hierarchies  {$(\tilDPSt_n)^*$} for $(\SEPBS_n)^*$, and $\MK^{(t)}_n$ for $\COP_n$. In this section we revisit this connection. First, we present a characterization of the dual cone  {$(\tilDPSt_n)^*$} that fully mirrors the result  in Theorem \ref{theo:FF} for $(\DPSt_n)^*$. Then, we use it to offer a  simple argument for the connection between  {$(\tilDPS_n^{(t)})^*$} and  {$\MK^{(t)}_n$}, using only the basic correspondence   (\ref{eq:rsos-sos}) between real-sos Hermitian polynomials in complex variables  and sos polynomials in real variables.
%it is interesting to further investigate separability of Bose symmetric states,  and their connections to the copositive cone when considering additional diagonal unitary invariance, which we do in the next sections.

\subsection{The dual of the symmetric  DPS hierarchy }\label{sec:DPStil}

For $M\in\Herm(\Cnn)$, recall the polynomial $F_M(x,\ovx,y,\ovy)\in\C[x,\ovx,y,\ovy]_{1,1,1,1}$ as defined in (\ref{eq:FM}). By restricting to one set of variables, i.e., setting $y=x$, we obtain  the polynomial
\begin{align}\label{eq:FM-BS}
\tilF_M(x,\ovx):=
F_M(x,\ovx,x,\ovx)=\langle M, xx^*\ot xx^*\rangle \in \C[x,\ovx]_{2,2},
\end{align} 
 with degree 2 in $x$ and degree 2 in $\ovx$. So, $M\in (\SEPBS_n)^*$ means %$\tilF_M(x,\ovx)
 $F_M(x,\ovx,x,\ovx)\ge 0$ for all $x\in \C^n$ (equivalently, for all $x\in \sS^{n-1})$.
In the next result we characterize when $M$ belongs to the dual of the relaxation $\tilDPSt_n$, thus giving an   analog of Theorem \ref{theo:FF} for the symmetric setting.

\begin{theorem}\label{theo:FF-BS}
Consider a matrix $M\in \Herm(\Cnn)$ and an integer $t\ge 1$. Set $\tau=\lfloor{(t+1)/2}\rfloor$.
The following assertions are equivalent.
\begin{itemize}
\item[(i)] $M\in (\tilDPSt_n)^*$.
\item[(ii)] The polynomial $\|x\|^{2(t-1)} \tilF_M(x,\ovx)=\|x\|^{2t} \langle M, xx^*\ot xx^*\rangle \in \C[x,\ovx]_{2(t+1),2(t+1)}$ is r-sos.
\item[(iii)] There exist matrices $B,W_0,W_1,\ldots,W_\tau \in\Herm(\Cnnt)$ such that 
$$M\ot I_n^{\ot (t-1)} = B+\sum_{s=0}^\tau W_,$$
where $B\in \Herm(S^{t+1}(\C^n))^\perp,\
%( \BS((\C^n)^{\ot (t+1)}))^\perp,\
W_s\in \MW^{(t)}_{n,s} \text{ for } s=0,1,\ldots,\tau.$
\end{itemize}
\end{theorem}
In comparison with Theorem \ref{theo:FF}, we are now asking in (iii) that the matrix $B$ belongs to the larger space
%$ ( \BS((\C^n)^{\ot (t+1)}))^\perp$, 
$\Herm(S^{t+1}(\C^n))^\perp$, which contains the space $\Herm(\C^n\ot S^t(\C^n))^\perp$
%$(\MV^{(t)}_n)^\perp$ 
involved in Theorem~\ref{theo:FF}(iii). In addition, we only need  the subspaces $\MW^{(t)}_{n,s}$ for $s\le \tau=\lfloor{(t+1)/2}\rfloor$.
%We defer the proof, which is a bit technical, to Appendix \ref{sec:proof-theo-FF-BS}.

\begin{remark}\label{rem:compare-GNP}
%\ML{In their recent work\cite[Section 7]{GNP_2025}, the authors introduce  hierarchies of cones $\text{\rm PPTBExt}^{(t)}_n$ for separable states and $\text{\rm PPTBExt}^{(K_t)}_n$ ($r\ge 0$) for separable Bose symmetric states, to which they refer, respectively, as `star extendibility' and  `complete extendability'.
% Observe that   $\text{\rm PPTBExt}^{(t)}_n=\DPS^{(t)}_n$ (although this connection is not explicitly mentioned in \cite{GNP_2025}) and that 
%$\text{\rm PPTBExt}^{(K_{t+1})}_n=\tilDPSt_n$. }

%\ML{They also give characterizations for the dual hierarchies. 
The set \smash{$\tilDPSt_n$} corresponds to the set  \smash{$\text{\rm PPTBExt}^{(K_{t+1})}_n$} considered in \cite{GNP_2025}. 
The characterization of the dual of \smash{$\text{\rm PPTBExt}^{(t)}_n$} given in 
% \cite[Proposition 7.3]{GNP_2025}  corresponds to the equivalence of items (i) and (iii) in Theorem \ref{theo:FF} (due to Fang and Fawzi \cite{FF-sphere}), and their characterization of the dual of  $\text{\rm PPTBExt}^{(K_t)}_n$ 
 \cite[Prop.~8.9]{GNP_2025}  corresponds to the equivalence of items (i) and (iii) in Theorem~\ref{theo:FF-BS}. 
The additional characterization (ii) in terms of r-sos polynomials, %for \smash{$M\in (\tilDPSt_n)^*$}, 
which is present in  Theorem  \ref{theo:FF-BS} (as well as in Theorem \ref{theo:FF}), is not given in \cite{GNP_2025}. Showing   this characterization (ii) constitutes in fact  the most technical part of the proof.
On the other hand, it
%in Theorem \ref{theo:FF-BS}(ii) for matrices $M\in (\tilDPS^{(t)}_n)^*$ 
 is especially relevant since it is directly comparable to the characterization (\ref{eq:KtCOP}) in terms of sums of squares    for matrices $A\in \MK^{(t)}_n$. 
%This motivates investigating the relationships between the relaxations $\tilDPS^{(t)}_n$ for the Bose symmetric separable cone $\SEPBS_n$ and the relaxations $\MK^{(t)}_n$ for the copositive cone $\COP_n$, as we do in the next section.
 \end{remark}

Before giving the proof, note that, as 
 in the general (non-symmetric) case, the LDUI projection preserves membership in \smash{$\tilDPSt_n$}and its dual cone, a property we will use in the next section.

\begin{lemma}\label{lem:LDUI-BS} 
Assume that $\rho_{AB}, M\in \Herm(\Cnn)$. Then,
 $\rho_{AB}\in \tilDPSt_n$ implies  $\PiLDUI(\rho_{AB})\in \tilDPSt_n$, and 
 $M\in (\tilDPSt_n)^*$ implies $\PiLDUI(M)\in (\tilDPSt_n)^*$.
 \end{lemma}
 
 \begin{proof}
It suffices to show the first claim, since it implies the second one using the fact that
$\langle \PiLDUI(M),\rho_{AB}\rangle=\langle M, \PiLDUI(\rho_{AB}\rangle.$
Assume $\rho_{AB}\in {\tilDPSt_n}$, and let $\rho_{AB[t]}$ be a  {$\tilDPSt$}-certificate for it.
Then, the projection of this certificate
$$\PiLDUIt(\rho_{AB[t]})=\int_{\MDU_n} U^{\ot (t+1)} \rho_{AB[t]} (U^{\ot (t+1)})^* dU
$$
is a \smash{$\tilDPSt$}-certificate for $\PiLDUI(\rho_{AB})$, showing
$\PiLDUI(\rho_{AB})\in\smash{\tilDPSt_n}$.
 \qed\end{proof}

%In the next section we use the characterization from Theorem \ref{theo:FF-BS} to establish a correspondence between the relaxations $\tilDPSt_n$ for $\SEPBS_n$ and the (dual of the)  relaxations $\MKt_n$ for   $\COP_n$.

\begin{proof}[Proof of Theorem \ref{theo:FF-BS}] %We now give the proof of Theorem \ref{theo:FF-BS}.
The proof is along the same lines as for Theorem \ref{theo:FF} in \cite{FF-sphere}. There are, however,   some differences since one needs  to restrict to $s\le \tau=\lfloor(t+1)/2\rfloor$ in (iii), which adds some technicalities. Our argument is also more compact.
Let $M\in \Herm(\Cnn)$ and let $$p(x,\ovx)=\|x\|^{2(t-1)} \langle M, xx^*\ot xx^*\rangle\in\C[x,\ovx]_{t+1,t+1}$$ denote the polynomial appearing in Theorem \ref{theo:FF-BS}(ii).
Recall the conditions involved in  Theorem~\ref{theo:FF-BS}: (i) \smash{$M\in(\tilDPSt_n)^*$}; (ii) $p$ is r-sos; 
(iii)  $M\ot I_n^{(t-1)}$ belongs to $ (\BS((\C^n)^{\ot (t+1)})^\perp+\sum_{s=0}^{\tau} \MW^{(t)}_{n,s}$; and the additional condition (iv) $M\ot I_n^{(t-1)}\in \text{\rm cl}((\BS((\C^n)^{\ot (t+1)})^\perp+\sum_{s=0}^\tau \MW^{(t)}_{n,s})$ (where `cl' denotes taking the closure).
Recall the definition of \smash{$\tilDPSt_n$} in (\ref{eq:DPSt-BS}) and of   $\MW^{(t)}_{n,s}$ in  (\ref{eq:Wst}), and   $\tau=\tftwo$.

\medskip
The implication  (iii) $\Longrightarrow$ (iv) is clear. The implication (iii) $\Longrightarrow$ (ii) follows by taking the inner product of $M\ot I_n^{\ot (t-1)}$  with $(xx^*)^{\ot (t+1)}$ and using the fact that 
$$\langle W_s, (xx^*)^{\ot (t+1)}\rangle=\langle W_s^{T_B[s]}, xx^* \ot (\ovx \ovx ^*)^{\ot s}\ot (xx^*)^{\ot (t-s)}\rangle$$ is r-sos for any $W_s\in \MW^{(t)}_{n,s}$ (since $W_s^{T_{B[s]}}\succeq 0$).% and $0\le s\le t$.

\medskip
We now show     (i) $\Longleftrightarrow$ (iv). 
By definition of the dual cone $(\tilDPSt_n)^*$,  
\begin{align*}
&M\in (\tilDPSt_n)^* \Longleftrightarrow \langle M, \rho_{AB}\rangle\geq0 \text{ for all } \rho_{AB}\in \tilDPSt_n\\
& \Longleftrightarrow 
\langle M\ot I_n^{\ot (t-1)}, \rho_{AB[t]}\rangle\geq0 \text{ for all }   
\rho_{AB[t]} \in \BS((\C^n)^{\ot (t+1)})\cap \bigcap_{s=0}^\tau \MW^{(t)}_{n,s}\\
& \Longleftrightarrow M\ot I_n^{\ot (t-1)} \in \Big(\BS(\Cnnt) \cap \bigcap_{s=0}^\tau \MW^{(t)}_{n,s}\Big)^*
%= \text{cl}\Big((\BS((\C^n)^{\ot (t+1)}))^{\perp} +\sum_{s=0}^\tau \MW^{(t)}_{n,s}\Big).
\end{align*}
where the latter set is equal to
$ \text{cl}\big((\BS((\C^n)^{\ot (t+1)}))^{\perp} +\sum_{s=0}^\tau \MW^{(t)}_{n,s}\big)$.
 Above, at the second line, we use the definition of the partial trace and of  {$\tilDPSt_n$} in (\ref{eq:DPSt-BS}).
 At the last line, we use the well-known convex geometry fact that the dual of an intersection of cones is the closure of the sum of the duals of the cones, combined with the fact that the cone $\MW^{(t)}_{n,s}$ is self-dual.

\medskip
We show  (iv) $\Longrightarrow$ (ii).
For  this, assume $M\ot I_n^{\ot (t-1)}=\lim_{\ell \to\infty} A_\ell$, where 
$A_\ell = B_\ell +\sum_{s=0}^\tau W_{\ell,s}$ with $B_\ell\in (\BS((\C^n)^{\ot (t+1)}))^{\perp}$ and $W_{\ell,s}\in \MW^{(t)}_{n,s}$. Define the polynomial 
\begin{align*}
\sigma_\ell(x,\ovx)=\langle A_\ell, (xx^*)^{\ot (t+1)}\rangle 
& = 
\big\langle \sum_{s=0}^\tau  W_{\ell,s}, (xx^*)^{\ot (t+1)}\big\rangle \\
& =
 \sum_{s=0}^\tau \langle W_{\ell,s}^{T_B[s]}, x x^*\ot (\ovx \ovx^*)^{\ot s}\ot (xx^*)^{\ot (t-s)} \rangle.
\end{align*}
Hence, $\sigma_\ell $ is r-sos since   $W_{\ell,s}^{T_B[s]}\succeq 0$ for each $s$.
As $p=\lim_{\ell\to\infty} \sigma_\ell$,  it follows that  $p$ is r-sos   (using the fact that 
the cone of r-sos polynomials in $\C[x,\ovx]_{t+1,t+1}$  is a closed set).  So, (ii) holds.

\medskip
Finally, we show  (ii) $\Longrightarrow$ (iii),  which is the most technical part. Assume $p$ is r-sos. 
Then, $p=\sum_\ell q_\ell^2  $ for some Hermitian polynomials $q_\ell\in \C[x,\ovx]$ with total degree $t+1$ in $x,\ovx$. We can write
$$q_\ell(x,\ovx)= \sum_{s=0}^t \langle a_{\ell,s}, \ovx^{\ot s}\ot x^{\ot (t+1-s)}\rangle \quad \text{ for some } a_{\ell,s}\in (\C^n)^{\ot (t+1)}.$$
As $q_\ell$ is Hermitian, we have $q_\ell=\overline{q_{\ell}}$, which implies
\begin{align}\label{eq:qHermitian}
\langle a_{\ell,s}, \ovx^{\ot s}\ot x^{\ot (t+1-s)}\rangle 
= \langle \overline{a_{\ell, t+1-s}}, x^{\ot (t+1-s)}\ot \ovx^{\ot s}\rangle
 \text{ for } s=0,1,\ldots,t.
\end{align}
Then, $q_\ell^2= \sum_{s,s'=0}^t 
\langle a_{\ell,s}, \ovx^{\ot s}\ot x^{\ot (t+1-s)}\rangle \cdot 
 \langle  {a_{\ell,s'}}, \ovx^{\ot s'}\ot x^{\ot (t+1-s')}\rangle.$
 As $p=\sum_\ell q_\ell^2$ has degree $t+1$ in $\ovx$, it follows that  the following sum  is identically zero:
$$\sum_\ell \sum_{s,s'=0, s+s'\ne t+1}^t \langle a_{\ell,s}, \ovx^{\ot s}\ot x^{\ot (t+1-s)}\rangle \cdot 
 \langle {a_{\ell,s'}}, \ovx^{\ot s'}\ot x^{\ot (t+1-s')}\rangle =0,$$
because any monomial occurring in it has degree in $\ovx$ distinct from $t+1$.
Therefore, we obtain
\begin{align*}
p=\sum_\ell q_\ell ^2 & =
\sum_\ell  \sum_{s=0}^t  
\underbrace{\langle a_{\ell,s}, \ovx^{\ot s}\ot x^{\ot (t+1-s)}\rangle \cdot 
 \langle {a_{\ell,t+1-s}}, \ovx^{\ot (t+1-s)}\ot x^{\ot s}\rangle}_{= T_{\ell,s}} \\
 &= \sum_\ell \sum_{s=0}^t T_{\ell,s}.
\end{align*}
Observe that $T_{\ell,s}=T_{\ell,t+1-s}$. This implies that 
$
\sum_{s=0}^t T_{\ell,s} =\sum_{s=0}^\tau c_s T_{\ell,s}$
with $c_s=2$  for $0\le s\le \tau -1$,
$c_\tau =2$  for odd $t$, and $c_\tau =1$   for even $t$.
Consider  the summand $\sum_{\ell} T_{\ell,s}$ for some  given $s\le \tau$. Using (\ref{eq:qHermitian}), we obtain
\begin{align*}
\sum_\ell T_{\ell,s}  & =\sum_\ell \langle a_{\ell,s}, \ovx^{\ot s}\ot x^{\ot (t+1-s)}\rangle \cdot 
 \langle {a_{\ell,t+1-s}}, \ovx^{\ot (t+1-s)}\ot x^{\ot s}\rangle
 \\
 &= \sum_\ell \langle a_{\ell,s}, \ovx^{\ot s}\ot x^{\ot (t+1-s)}\rangle \cdot 
 \overline{ \langle a_{\ell,s}, \ovx^{\ot s}\ot x^{\ot (t+1-s)}\rangle}  \\
 & =\sum_\ell (a_{\ell,s})^* (\ovx^{\ot s}\ot x^{\ot (t+1-s)}) ( \ovx^{\ot s}\ot x^{\ot (t+1-s)})^*a_{\ell,s} \\
%&=  \sum_\ell    (a_{\ell,s})^* (\ovx^{\ot s}\ot x^{\ot (t+1-s)}) \ (\ovx^{\ot s}\ot x^{\ot (t+1-s)})^* a_{\ell,s}\\
&=\Big\langle \sum_\ell   a_{\ell,s} (a_{\ell,s})^*, (\ovx^{\ot s}\ot x^{\ot (t+1-s)}) \ (\ovx^{\ot s}\ot x^{\ot (t+1-s)})^* \Big\rangle
\\
&=\langle A_s, (\ovx^{\ot s}\ot x^{\ot (t+1-s)}) \ (\ovx^{\ot s}\ot x^{\ot (t+1-s)})^* \rangle\\
&
= \langle A_s^{T_B[s]}, x^{\ot (t+1)} (x^{\ot (t+1)})^*\rangle.
\end{align*}
Above, at the last but one line, we set $A_s=\sum_\ell a_{\ell,s}(a_{\ell,s})^*$ and, at the last line, we take the partial transpose $T_{B[s]}$ at both arguments in the inner product. 
%Note that $A_s^{T_B[s]}\in \MW^{(t)}_{n,s}$ (since $A_s\succeq 0$ by construction).
Putting things together we obtain
\begin{align*}
p= \sum_\ell q_\ell^2 & =\sum_\ell  \sum_{s=0}^t    T_{\ell,s} 
= \sum_\ell \sum_{s=0}^\tau c_s T_{\ell,s}\\
&=\sum_{s=0}^\tau c_s\sum_\ell T_{\ell,s}
=\Big \langle \sum_{s=0}^\tau c_s A_s^{B[s]}, x^{\ot (t+1)} (x^{\ot (t+1)})^*\Big\rangle.
\end{align*}
So, we have shown  the polynomial identity
$$p=\langle M\ot I_n^{\ot (t-1)}, (xx^*)^{\ot (t+1)}\rangle =  \big\langle  \sum_{s=0}^\tau c_sA_s^{T_{B[s]}}, (xx^*)^{\ot (t+1)}\big\rangle.$$
%where $A_s^{T_{B[s]}}\in \MW^{(t)}_{n,s}$ (since $A_s\succeq 0$).
This implies   $M\ot I_n^{\ot (t-1)} - \sum_{s=0}^\tau c_s A_s^{T_{B[s]}} \in (\BS((\C^n)^{\ot (t+1)})^\perp$.
This shows that (iii) holds, since $c_sA_s^{T_{B[s]}}\in \MW^{(t)}_{n,s}$ (because $A_s\succeq 0$). This concludes the proof of Theorem \ref{theo:FF-BS}.
\qed \end{proof}

%\subsection{Connections between the hierarchy $(\tilDPSt_n)^*$ and the hierarchy $\MKt_n$}
\subsection[\smash{Connections between the hierarchy $(\tilDPSt_n)^*$ and the hierarchy $\MKt_n$}]{Connections between the hierarchy $(\tilDPSt_n)^*$ and the hierarchy $\MKt_n$}

Given a matrix $X\in \MS^n$, consider the associated bipartite states
$\rho_{(X,X)}$ and $ \rho_{(X,X)}^{T_B}=\rho_{(X,\cdot,X)}$, which  are, respectively, CLDUI and LDUI. We focus here on the state $ \rho_{(X,X)}^{T_B}$ that has the additional property of being Bose symmetric.  
Recall from relation (\ref{eq:CP-PCP-TCP}) that
$$\rho_{(X,X)}\in\SEP_n\Longleftrightarrow \rho_{(X,X)}^{T_B} \in \SEP_n\Longleftrightarrow X\in \CP_n.$$
Hence, to check separability of $\rho_{(X,X)}$, a first option is to  use the DPS hierarchy. However, another, likely better, option is to use the stronger $\tilDPS$ hierarchy, applied to the Bose symmetric state $\rho_{(X,X)}^{T_B}$. A third option is to use the  hierarchy of cones $(\MKt_n)^*$, applied to the matrix $X$. We will now  
show that  these last two options  are in fact
 equivalent (see Theorem \ref{theo:DPS-CP}).

\medskip
For this, we first establish a link on the `dual polynomial side'. In Theorem \ref{theo:DPS-KCOP} we essentially\footnote{For a pair $(A,B)\in \R^{n\times n}\times  \Herm^n$ with $\diag(A)=\diag(B)$,  \cite[Theorem 8.11]{GNP_2025} shows that    $\rho_{(A,\cdot,B)}$ belongs to the dual of {$\tilDPSt_n$} if and only if the matrix $A+A^T+ 2\Rea(B-D_B)\in\MK^{(t-1)}_n$. In view of relation (\ref{eq:MST-AB-GNP}), we have $\rho_{(A,\cdot,B)}=\rho_{(A,B)}^{T_B}= M_{S,T}^{T_B}$ with $S=A$ and $T=B^T-D_B$, which makes the link to Theorem \ref{theo:DPS-KCOP}.}
 recover the result from \cite[Theorem 8.11]{GNP_2025}. Here, we give a direct proof for it, 
based on using   Theorem \ref{theo:FF-BS}(ii) combined with the  correspondence (\ref{eq:rsos-sos}) between Hermitian r-sos polynomials and their real couterparts.

Given  a matrix pair $(S,T)\in \R^{n\times n}\times \Herm^n$, consider the matrix $M_{S,T}$ from    (\ref{eq:MST}), which is CLDUI by construction, so that its partial transpose is LDUI and  reads
\begin{align}\label{eq:MST-TB}
	M_{S,T}^{T_B} 
	%=\sum_{i=1}^n(S_{ii}+T_{ii})e_ie_i^*\otimes e_ie_i^* +
	=\sum_{i,j=1}^nS_{ij}e_ie_i^*\otimes e_je_j^*+\sum_{i,j=1}^nT_{ij}e_je_i^*\otimes e_ie_j^*.
	\end{align}
%We will   use the following fact from (\ref{eq:MST-rhoXY})): for a matrix $X\in \MS^n$,
%\begin{align}\label{eq:MST-X}
%\langle M_{S,T},\rho_{(X,X)}\rangle = \langle S+T,X\rangle=\langle (S+S^T+T+T^T)/2, X\rangle.
%\end{align}
	
\begin{theorem}\label{theo:DPS-KCOP}
For $(S,T)\in\R^{n\times n}\times \Herm^n$, consider the matrix $M_{S,T}$ and its partial transpose $M_{S,T}^{T_B}$ as in (\ref{eq:MST-TB}), and set $A=(S+S^T+T+T^T)/2\in\MS^n$.
 For any integer $t\ge 1$, we have 
 $$ M_{S,T}^{T_B}\in \smash{ (\tilDPSt_n)^* }\ \Longleftrightarrow \ A\in\MK^{(t-1)}_n.$$ 
In particular,  for any $A\in\MS^n$,    $M_{A,0}^{T_B} \in \smash{ (\tilDPSt_n)^*}$ $\Longleftrightarrow $ 
$A\in \MK^{(t-1)}_n$.
\end{theorem}

\begin{proof}
We only need to show the first equivalence since the second one follows as a direct application. Using (\ref{eq:MST-TB}), one can easily verify the following polynomial identity:
\begin{align}\label{eq:pMA}
\begin{split}
p(x,\ovx):= \|x\|^{2(t-1)} \langle M_{S,T}^{T_B}, xx^*\ot xx^*\rangle \\
=\|x\circ \ovx\|^{2(t-1)} \langle A, (x\circ \ovx) (x\circ \ovx)^T\rangle=: f(x\circ \ovx),
\end{split} \end{align}
where $A=(S+S^T+T+T^T)/2$.  Then, the proof amounts to comparing the r-sos property for the polynomial $p$  
(in complex variables $x,\ovx$) at the left-hand side, with the sos property for the polynomial $f$
 (in real variables $x\circ \ovx$) at the right-hand side.
For this, write $x=x_{\Rea}+\bfi x_{\Ima}$, where $x_\Rea=(x_{k,\Rea})_{k=1}^n, x_{\Ima}=(x_{k, \Ima})_{k=1}^n \in \R^n$ are the real and imaginary parts of $x$.
Then, we have $x\circ \ovx= (|x_k|^2)_{k=1}^n=(x_{\Rea,k}^2+x_{\Ima,k}^2)_{k=1}^n$ and thus 
the real part of the polynomial $p$ in (\ref{eq:pMA}) reads
\begin{align*} 
& p_\Rea(x_\Rea,x_\Ima) \\
& = \Big(\sum_{h=1}^n (x_{h,\Rea})^2+(x_{h,\Ima})^2\Big)^{t-1}
\sum_{h,k=1}^n A_{hk} ((x_{h,\Rea})^2+(x_{h, \Ima})^2)
 ((x_{k,\Rea})^2+(x_{k,\Ima})^2)\\
 & =  \|(x_{\Rea},x_{\Ima})\|^{2(t-1)} \begin{pmatrix} x_{\Rea}\cr x_{\Ima}\end{pmatrix}^T  \begin{pmatrix} A & A \cr A& A\end{pmatrix}   \begin{pmatrix} x_{\Rea}\cr x_{\Ima}\end{pmatrix}\\
& =  \|(x_{\Rea},x_{\Ima})\|^{2(t-1)} \begin{pmatrix} x_{\Rea}\cr x_{\Ima}\end{pmatrix}^T (J_2\ot A)  \begin{pmatrix} x_{\Rea}\cr x_{\Ima}\end{pmatrix}.
  %& =  \big(\sum_{h=1}^n (x_{h,\Rea})^2+(x_{h,\Ima})^2\big)^{t-1}
% \sum_{h,k=1}^n A_{hk} ((x_{h,\Rea})^2+(x_{h, \Ima})^2)
% ((x_{k,\Rea})^2+(x_{k,\Ima})^2).
  \end{align*}
 Hence,  the  polynomial $p_\Rea(x_\Rea,x_\Ima)$ is a sum of squares  if and only if the matrix $J_2\ot A$ belongs to $\MK^{(t-1)}_{2n}$, which, 
using \cite[Lemma 15]{GL_2007}, 
is equivalent to asking that   $A$ belongs to $\MK^{(t-1)}_n$.  
  
 We can now complete the proof: $M_{S,T}^{T_B}$ belongs to \smash{$\tilDPSt_n$} precisely when the polynomial $p(x,\ovx)$ is r-sos. By relation (\ref{eq:rsos-sos}), $p(x,\ovx)$ is r-sos if and only if  its real part $p_\Rea(x_\Rea,x_\Ima)$ is sos, which, as shown above, is equivalent to $A$ belonging to $\MK^{(t-1)}_n$. %This concludes the proof. % of Theorem  \ref{theo:DPS-KCOP}.  
\qed\end{proof}

Using duality arguments, one   derives the following result, %from Theorem \ref{theo:DPS-KCOP} 
recovering \cite[Theorem 8.16]{GNP_2025}.

\begin{theorem}\label{theo:DPS-CP}
For a matrix $X\in\MS^n$ and an integer $t\ge 1$, we have
$$  \rho_{(X,X)}^{T_B} \in \smash{\tilDPSt_n} \Longleftrightarrow    
X\in (\MK^{(t-1)}_n)^*.$$
\end{theorem}

\begin{proof}
Assume $\rho_{(X,X)}^{T_B}\in \smash{\tilDPSt_n}$, we show  
$X\in (\tilMK^{(t-1)}_n)^*$. For this,  we show that $\langle X,A\rangle \ge 0$ for any $A \in \MK^{(t-1)}_n$.
By Theorem \ref{theo:DPS-KCOP}, $A\in \MK^{(t-1)}_n$ implies  $M_{A,0}^{T_B}\in (\tilDPSt_n)^*$. So, combining with (\ref{eq:MST-rhoXY}) gives 
$$0\le \langle M_{A,0}^{T_B}, \rho_{(X,X)}^{T_B}\rangle=\langle M_{A,0}, \rho_{X,X}\rangle 
= \langle A,X\rangle,
$$
as desired. Conversely, assume $X\in (\MK^{(t-1)}_n)^*$, we show $\rho_{(X,X)}^{T_B}\in \tilDPSt_n$.
For this,
we show that $\langle \rho_{(X,X)}^{T_B},M\rangle\ge 0$ for  $M\in  (\tilDPSt_n)^*$.
As $\rho_{(X,X)}^{T_B}$ is LDUI, equality  $\langle \rho_{(X,X)}^{T_B},M\rangle =
\langle \rho_{(X,X)}^{T_B},\Pi_{\LDUI}(M)\rangle$ holds. 
By Lemma \ref{lem:LDUI-BS},   projecting onto $\LDUI_n$ preserves membership in 
the cone  {$(\tilDPSt_n)^*$}. Hence, we may assume that $M$ is LDUI, i.e.,  of the form $M=M_{S,T}^{T_B}$ for some $S,T$. By Theorem \ref{theo:DPS-KCOP},  $M_{S,T}^{T_B}\in \smash{(\tilDPSt_n)^*}$ implies 
$S+S^T+T+T^T\in \MK^{(t-1)}_n$.
Then, we have
$$\langle M, \rho_{(X,X)}^{T_B}\rangle=\langle M_{S,T}^{T_B},  \rho_{(X,X)}^{T_B}\rangle=  \langle M_{S,T}, \rho_{(X,X)}\rangle
=\langle X, (S+S^T+T+T^T)/2\rangle \ge 0,
$$
as desired.
\qed\end{proof}

\subsection{A class of Bose symmetric LDOI states}\label{sec:exampleab}

As an illustration we consider here a class of Bose symmetric states with LDOI structure. Given 
 scalars $a,b\in \R$, consider the  matrix $\rho(a,b):=\rho_{(X,Y,X)}\in \Herm(\C^3\ot \C^3)$,   with 
\begin{align}\label{eq:rhoLDOIab}
X=\begin{pmatrix} 1 & a & a \cr a & 1 & a\cr a & a &1\end{pmatrix}, 
Y=\begin{pmatrix} 1 & b & b \cr b & 1 & b\cr b& b &1\end{pmatrix}, \
\rho(a,b)= \begin{pmatrix} 1 & b & b \cr b & 1 & b\cr b& b &1\end{pmatrix}
\oplus \begin{pmatrix} a & a \cr a & a\end{pmatrix}
\oplus \begin{pmatrix} a & a \cr a & a\end{pmatrix}
\oplus \begin{pmatrix} a & a \cr a & a\end{pmatrix},
\end{align}
where the first block is indexed by $\{11,22,33\}$ and the last three blocks are indexed, respectively, by $\{12,21\}$, $\{23,32\}$ and $\{31,13\}$. Then, $\rho(a,b)\in \LDOI_3\cap \BS(\C^3)$  (recall (\ref{eq:LDUI-BS})) and thus 
any separable decomposition of it must be symmetric (Lemma \ref{lem:SEP-BS}).

In view of   (\ref{eq:rhoLDOIab}), $\rho(a,b)$ is a quantum state (i.e., $\rho(a,b)\succeq 0$) precisely when $a\ge 0$ and  $-{1\over 2}\le b\le 1$. As we  see in Lemma \ref{lem:exampleab}, separability of $\rho(a,b)$ is characterized already at the first level of the DPS hierarchy.
This extends to the case when  $X,Y$ are  defined analogously with size  $n\ge 3$, now requiring $-{1\over n-1}\le b$ instead of $-{1\over 2}\le b$. 
More generally, equivalence of membership in $\DPS^{(1)}_n$ and $\SEP_n$ holds  for LDOI states of the form $\rho_{(X,Y,Z)}$, where each of $X,Y,Z$ lies in the linear span of $I_n$ and the all-ones matrix $J_n$, see \cite[Theorem 4.1]{PJPY_2024} and \cite[Theorem III.3]{GNS-cyclic_2025}. We yet give the (easy)  details for the states $\rho(a,b)$ since this also provides examples of real-valued separable states that do not have real separable decompositions (Remark \ref{rem:exampleab}).

\begin{lemma} \label{lem:exampleab}
Given scalars $a,b\in\R$, we have 
$$\rho(a,b)\in \SEP_3 \Longleftrightarrow \rho(a,b)\in \DPS_3^{(1)}\Longleftrightarrow -{1\over 2}\le b \text{ and } |b|\le a\le 1.$$
\end{lemma}

\begin{proof}
Only two implications need a proof. First, assume $\rho(a,b)\in \DPS^{(1)}_3$, we show that $-{1\over 2} \le b$ and $|b|\le a\le 1$. Indeed,  the first condition follows since $\rho(a,b)\succeq 0$ and the second one from the fact that $\rho(a,b)^{T_B}\succeq 0$, since 
$$\rho(a,b)^{T_B} = \rho_{(X,X,Y)}=
 \begin{pmatrix} 1 & a & a \cr a & 1 & a\cr a& a &1\end{pmatrix}
\oplus \begin{pmatrix} a & b \cr b & a\end{pmatrix}
\oplus \begin{pmatrix} a & b \cr b & a\end{pmatrix}
\oplus \begin{pmatrix} a & b \cr b & a\end{pmatrix}.
$$
Assume now that $-{1\over 2}\le b$ and $|b|\le a\le 1$, we show that $\rho(a,b)\in\SEP_3$. For this note that 
the set 
$$\Big\{(a,b)\in\R^2: |b|\le a\le 1, \ -{1\over 2}\le b \Big\}$$ is convex bounded, and equals  the convex hull of the four points 
$(0,0)$, $(1,1)$, $\big({1\over 2}, -{1\over 2}\big)$, and $\big(1, -{1\over 2}\big)$. Hence, it suffices to show that $\rho(a,b)$ is separable for each of them.  For $(a,b)=(0,0)$ we have
$$\rho(0,0)= (e_1e_1^*)^{\ot 2} +(e_2e_2^*)^{\ot 2}+ (e_3e_3^*)^{\ot 2}\in \SEP_3.$$
The all-ones bipartite state $ (ee^*)^{\ot 2}$  (with $e=\sum_{i=1}^3 e_i$) is separable, hence also its projection $\rho(1,1)=\Pi_{\LDOI}((ee^*)^{\ot 2})$ is separable. Next, one can verify that
\begin{align} \label{eq:rhohalfhalf}
\begin{split}
 \rho\Big({1\over 2},-{1\over 2}\Big) & ={1\over 2} \Pi_{\LDOI} ((xx^*)^{\ot 2}+(yy^*)^{\ot 2}+(zz^*)^{\ot 2}),\\ & \text{ with } 
x= \begin{pmatrix} 0 \cr 1 \cr \bfi\end{pmatrix},
 y=  \begin{pmatrix} 1 \cr 0 \cr \bfi\end{pmatrix},
 z=  \begin{pmatrix} 1 \cr \bfi  \cr 0\end{pmatrix},  
 \end{split}\end{align}
 \begin{align}
 \begin{split}
 \rho\Big(1,-{1\over 2}\Big)& = \Pi_{\LDOI} ((uu^*)^{\ot 2} + (vv^*)^{\ot 2}), \\
 & \text{ with }
u=\begin{pmatrix} e^{\bfi \pi\over 3}\cr e^{-\bfi \pi\over 3}\cr 1\end{pmatrix},\ \ 
v=\begin{pmatrix} e^{2\bfi \pi\over 3}\cr e^{-2\bfi \pi\over 3}\cr 1\end{pmatrix}. \label{eq:rho1half}
\end{split}
\end{align}
Hence, $\rho(a,b)\in\SEP_3$ in both cases, again because  the LDOI projection preserves separability.
\qed \end{proof}

\begin{remark}\label{rem:exampleab}
Observe that, if $a=b \in [0,1]$, then $\rho(a,b)$ has a   separable decomposition that involves only real-valued rank one states (since this is the case for $\rho(0,0)$ and $\rho(1,1)$). However, if $b<a$, then no  separable decomposition exists that uses only real-valued rank one states. To see this, observe that if $\rho(a,b)=\sum_\ell (x_\ell x_\ell^*)^{\ot 2}$ for some $x_\ell\in \R^3$, then we have
$b=\rho(a,b)_{11,22}= \sum_\ell (x_\ell)_1^2 (x_\ell)_2^2 =\rho(a,b)_{12,12}=a$.
In particular, if $b=0< a\le 1$, then $\rho(a,0)=\rho_{(X,\cdot,X)}$  (resp., $\rho(a,0)^{T_B}=\rho_{(X,X)}$) is LDUI (resp.,   CLDUI),  separable, and has   no real-valued separable decomposition.

In addition, observe that, for   $b=-1/2$, any separable decomposition of the form $\rho(a,-1/2)=\sum_\ell (x_\ell x_\ell^*)^{\ot 2}$ must satisfy the condition
$\sum_{\ell} (x_\ell)_i^2=0$ for $i\in [3].$ Indeed, the vector $\sum_{i=1}^3 e_i\ot e_i$ belongs to the kernel of $\rho(a,-1/2)$, and thus also to the kernel of each $(x _\ell x_\ell^*)^{\ot 2}$, which implies the claimed fact. 
Hence, we again see that the $x_\ell$'s cannot be real-valued, and this also indicate  how to come up with the separable decompositions from (\ref{eq:rhohalfhalf}) and (\ref{eq:rho1half}).
\end{remark} %\label{sec:DPS-symmetric}

% temporarily cutting chap 5 out because of latex errors
\section{Implementation }\label{sec:implementation}

In this section we consider the question of how to implement the $\DPS$ hierarchy more efficiently in practice. %practically. 
For this, we reformulate it using the moment approach within the polynomial optimization framework. We do this first in the generic case and then we consider the $\CLDUI$ and $\LDOI$ cases. We conclude with a  comparison of runtimes in the different regimes by showing how the developed sparsity reduction technique can be used to tackle a class of examples motivated by the   PPT$^2$ conjecture from quantum information theory.

\subsection{The $\DPS$ hierarchy in the moment formalism}

Suppose one is given a bipartite quantum state $\rho_{AB}\in\Herm(\C^n\ot\C^n)_+$ and one wants to test whether $\rho_{AB}$ lies in a given level $t$ of the $\DPS$ hierarchy, i.e., whether $\rho_{AB}\in\DPS^{(t)}_n$. 
The obvious way to do so is to implement the semidefinite program modeling definition (\ref{eq:DPS}), asking for the existence of an extended state $\rho_{AB[t]}\in\Herm(\C^n\ot S^t(\C^n))$ that is positive semidefinite under   partial transposition of the $B$-registers and recovers the original state $\rho_{AB}$ by taking a partial trace. We refer to this original viewpoint as the `tensor formalism'.

In the tensor formalism the resulting SDP inherits redundancies arising from the symmetries in the $t$ $B$-registers as required in (\ref{eq:DPS1}): for any $\sigma \in \Perm_t$ and $i_0\ui\in[n]^{t+1}$,  the   columns (and rows) indexed by $i_0\ui$ and $i_0  \ui^\sigma$ are identical, i.e., $(\rho_{AB[t]})_{\cdot, i_0 \ui^\sigma} = (\rho_{AB[t]})_{\cdot, i_0\ui}$. Hence, one can set up an equivalent, reduced SDP that is obtained by keeping only one column out of the orbit consisting of the  columns indexed by $i_0\ui^\sigma$ for     $\sigma\in\Perm_t$.
%ignoring the columns (and rows) corresponding to the indices $i_0\ui^\sigma$ whenever $\sigma$ is not the identity permutation. 
An example of this fact was seen in Lemma~\ref{LExBlocksAreScalar} (and its proof), where it was used to show that the block of a $\DPS^{(2)}$-certificate  indexed, e.g.,  by $\{123,132\}$,    is   a scalar multiple of the all-ones matrix.

\medskip
A way to carry out this reduction more systematically is to consider the $\DPS$ hierarchy in the `moment formalism', as has been developed in \cite{GLS_2021}. For this, note that two elements $i_0\ui,j_0\uj\in [n]^{t+1}$ indexing the extended state $\rho_{AB[t]}$  lie in the same orbit of $\Perm_t$ if and only if $i_0=j_0$ and $\alpha(\ui) = \alpha(\uj)$. Hence, roughly speaking, it suffices to index the extended state by elements $\alpha(i_0)\alpha(\ui)\in \N^n\times \N^n$ (indexing monomials)   instead of sequences $i_0\ui\in [n]^{t+1}$. %Here, $\N^n_{=t}=\{\alpha\in \N^n: |\alpha|=\sum_{i=1}^n\alpha_i=t\}$.

More precisely, the moment approach works as follows. We consider Hermitian linear functionals $L\in (\C[x,\ox,y,\oy]_{2,2t})^*$, i.e., $L:\C[x,\ox,y,\oy]_{2,2t}\to \C$ acting  on bihomogenous polynomials of degree 2 in $x,\ox$ and of degree $2t$ in $y,\oy$. Being Hermitian means that $L(\overline p)=\overline { L(p)}$ for any polynomial $p\in \C[x,\ox,y,\oy]_{2,2t}$.
Then, the corresponding moment matrix of $L$, denoted here as $M_{1,t}(L)$, is indexed by 
elements  $(\alpha,\alpha',\beta,\beta')\in (\N^n)^4$ such that $|\alpha+\alpha'|=1$ and $|\beta+\beta'|=t$;
 its $(\alpha\alpha'\beta\beta',\tilde \alpha\tilde\alpha'\tilde\beta\tilde\beta')$-entry is  $M_{1,t}(L)_{\alpha\alpha'\beta\beta',\tilde\alpha\tilde\alpha'\tilde\beta\tilde\beta'}$ , defined as 
$$
%M_{1,t}(L)_{\alpha\alpha'\beta\beta',\tilde\alpha\tilde\alpha'\tilde\beta\tilde\beta'} = 
L( x^\alpha \ox^{\alpha'} y^\beta \oy ^{\beta'} \cdot 
\overline { x^{\tilde \alpha} \ox^{\tilde \alpha'} y^{\tilde \beta} \oy^{\tilde \beta'}  })
= L(x^{\alpha+\tilde\alpha'}\ox^{\alpha'+\tilde\alpha}y^{\beta+\tilde\beta'}\oy^{\beta'+\tilde\beta}).$$ 
%where $\alpha,\alpha',\tilde\alpha,\tilde\alpha'\in\N^n$ are such that $|\alpha+\alpha'|=1$, $|\tilde\alpha+\tilde\alpha'|=1$, and $\beta,\beta',\tilde\beta,\tilde\beta'\in\N^n$ are such that $|\beta+\beta'|=t$, $|\tilde\beta+\tilde\beta'|=t$. 
Hence, $M_{1,t}(L)$ is a Hermitian matrix since $L$ is Hermitian. Moreover, we have  $M_{1,t}(L)\succeq 0$ if and only if $L(p\overline p)\ge 0$ for any $p\in \C[x,\ox,y,\oy]_{=1,=t}$, which reflects the fact that the moment approach is dual to the sum of squares approach used earlier.
The key connection between the `tensor formalism' and the `moment formalism' is then established by setting 
\begin{align}\label{eq:Lrho}
 L(xx^*\ot(yy^*)^{\ot t})=\rho_{AB[t]},
 \end{align}
 where $L$ is applied entrywise to the matrix $  xx^*\ot(yy^*)^{\ot t}$.
  Indeed, any linear functional $L\in (\C[x,\ox,y,\oy]_{2,2t})^*$ defines via (\ref{eq:Lrho}) an extended state lying in $\Herm(\C^n\ot S^t(\C^n))$ that satisfies (\ref{eq:DPS1}). Conversely,  given such an extended state $\rho_{AB[t]}$,  relation (\ref{eq:Lrho})  permits to define a linear functional $L\in (\C[x,\ox,y,\oy]_{2,2t})^*$ as follows: For $(\gamma,\gamma',\delta,\delta')\in (\N^n)^4$ such that $|\gamma+\gamma'|=2$ and $|\delta+\delta'|=2t$, set
%  if $|\gamma|=|\gamma'|=1$ and $|\delta|=|\delta'|=t$ then set   $L(x^\gamma \ox^{\gamma'}y^\delta \oy^{\delta'})= (\rho_{AB[t]})_{i_0\ui,j_0\uj}$ where $\gamma=\alpha(i_0),\gamma'=\alpha(j_0), \delta=\alpha(\ui), \delta'=\alpha(\uj)$, 
  \begin{align}\label{eq:Lnot0}
  L(x^\gamma \ox^{\gamma'}y^\delta \oy^{\delta'})= (\rho_{AB[t]})_{i_0\ui,j_0\uj} \text{ if } 
  |\gamma|=|\gamma'|=1 \text{ and }   |\delta|=|\delta'|=t,
    \end{align}
 where   $\gamma=\alpha(i_0),\gamma'=\alpha(j_0), \delta=\alpha(\ui), \delta'=\alpha(\uj)$
    and, otherwise, set  $L$ to 0, i.e.,
  \begin{align}\label{eq:L0}
  L(x^\gamma \ox^{\gamma'}y^\delta \oy^{\delta'})=0 \text{ if } |\gamma|\ne |\gamma'| \text{ or } |\delta|\ne |\delta'|. % \quad \text{ for }  |\gamma+\gamma'|=2, |\delta+\delta'|=2t.
    \end{align}
%  one obtains   a Hermitian linear functional $L\in (\C[x,\ox,y,\oy]_{=2,=2t})^*$.
  For the moment matrix of $L$,  property (\ref{eq:L0}) implies  that $M_{1,t}(L)_{\alpha\alpha'\beta\beta',\tilde\alpha\tilde\alpha'\tilde\beta\tilde\beta'}\not=0$ only if $|\alpha+\tilde\alpha'|=|\alpha'+\tilde\alpha|=1$ and $|\beta+\tilde\beta'|=|\beta'+\tilde\beta|=t$. 
  %Putting all conditions on the $\alpha$  and $\beta$ together yields 
  As observed in \cite[Section 5.2]{GLS_2021}, this permits to show that the matrix $M_{1,t}(L)$ is block-diagonal and that its  blocks are indexed%\footnote{We denote the index as $s'$, since $s'$ corresponds to the amount of partial transpositions in the DPS certificate. Since the original $s$ selecting the amount of partial transpositions runs from 0 to $t$ we denote the new variable with a prime. In fact, $s'$ relates to $s$ by a linear shift.} 
  \footnote{As we see  in the proof of Lemma \ref{lem:DPSmom}, the principal submatrix of    $M_{1,t}(L)$ indexed by   $I_{1,s'}$ corresponds to applying the partial transpose $T_{B[s]}$ to $\rho_{AB[t]}$, when $s'=t-2s$. This explains why we use the letter $s'$ here.}
  by the following sets 
  $$I_{r,s'}=\{(\alpha,\alpha',\beta,\beta')\in(\N^n)^4:|\alpha+\alpha'|=1, |\beta+\beta'|=t, |\alpha|-|\alpha'|=r, |\beta|-|\beta'|=s'\}$$ for $r\in\{-1,1\}$ and $s'\in\{-t, -t+2,-t+4,\dots,t-4,t-2,t\}$. As we see in the next lemma, it suffices to consider $r=1$, in which case any $(\alpha,\alpha',\beta,\beta')\in I_{1,s}$ has $\alpha'=0$, so that the entry $\alpha'$ can be ignored. 
      
\begin{lemma} \cite{GLS_2021}\label{lem:DPSmom}
    Let $\rho_{AB}\in\Herm(\C^n\ot\C^n)$. Then, $\rho_{AB}\in\DPS^{(t)}_n$ if and only if there exists a Hermitian linear functional $L\in (\C[x,\ox,y,\oy]_{2,2t})^*$ satisfying conditions (\ref{eq:L0})  and    (i),(ii) below:     \begin{itemize}
        \item[(i)] $L(xx^*\ot yy^*||y||^{2(t-1)}) = \rho_{AB}$,
        \item[(ii)] $M_{1,t}(L)\succeq 0$ or, equivalently, 
        $M_{1,t}(L)[I_{r,s'}]\succeq 0$ for  $r=1$ and $s'\in [-t,t]$ with $s'\equiv t$ (mod 2).
        %$s'\in\{-t, -t+2, \dots, t-2,t\}$.
       \end{itemize}
\end{lemma}

\begin{proof}
  The lemma essentially follows\footnote{We refer in particular to \cite[Prop.23]{GLS_2021} (which considers the general case where also the $A$-register is extended). There, the statement is in terms of existence of a linear functional $\tilde L\in (\C[x,\ovx,y,\ovy]_{\le 2,\le 2t})^*$
  %, acting on polynomials of degree at most $2$ in $x,\ovx$ and at most degree $2t$ in $y,\ovy$, 
  satisfying the `ideal' conditions 
  $\tilde L(p(1-\|x\|^2)=\tilde L(p(1-\|y\|^2)=0$ for all  $p\in \C[x,\ovx,y,\ovy]$ such that
  $p(1-\|x\|^2), p(1-\|y\|^2)\in \C[x,\ovx,y,\ovy]_{\le 2,\le 2t}.$ This non-homogeneous setting was used in \cite{GLS_2021} to give an alternative, real-algebraic proof for the completeness of the DPS hierarchy. 
  Note that one can easily navigate between the non-homogeneous and homogeneous settings:  by restricting   $\tilde L$ to $\C[x,\ox,y,\oy]_{2,2t}$ and, conversely, one can extend $L\in (\C[x,\ox,y,\oy]_{2,2t})^*$ to $\tilde L\in (\C[x,\ovx,y,\ovy]_{\le 2,\le 2t})^*$ using the above `ideal' constraints combined with  (\ref{eq:L0}). }  
  from results in \cite[Section 5]{GLS_2021},  
  %The full details of the proof can be found in  
   we   give a   sketch of proof. %for the convenience of the reader.
As mentioned above, the guiding fact for the proof is relation (\ref{eq:Lrho}) stating $\rho_{AB[t]}=L(xx^*\ot (yy^*)^{\ot t})$, which indicates how to construct a $\DPS^{(t)}$-certificate from a linear functional and vice versa.

First, note that  $L(xx^*\ot (yy^*)^{\ot t})$ clearly satisfies condition (\ref{eq:DPS1}).
Second, note that condition (i) precisely corresponds to condition \eqref{eq:DPS3}, since $$\Tr_{B[2:t]}(L(xx^*\ot(yy^*)^{\ot t}))=L(\Tr_{B[2:t]}(xx^*\ot(yy^*)^{\ot t}))=L(xx^*\ot yy^*||y||^{2(t-1)}).$$
%The final step of the proof consists in relating 
There remains to relate condition (ii) to condition (\ref{eq:DPS2}).
% to positive semidefiniteness of the blocks   $M_{1,t}(L)[I_{1,s'}]$ for $-t\le s'\le t$ with $s'\equiv t$ (mod 2).
  As observed above,  under condition (\ref{eq:L0}), $M_{1,t}(L)$ is block-diagonal, with its blocks indexed by the sets $I_{r,s'}$ for $r=\pm 1$ and $-t\le s'\le t$ with $s'\equiv t$ (mod 2). Hence,  $M_{1,t}(L)\succeq 0$  if and only if $M_t(L)[I_{r,s'}]\succeq 0$  for all $r\in\{-1,1\}$ and $s'\in\{-t, -t+2,\dots,t-2,t\}$. Moreover, one can  restrict   to $r=1$, since $L$ being Hermitian implies $$M_{1,t}(L)[I_{r,s'}]=\overline{M_{1,t}(L)[I_{-r,-s'}]}$$ (and positive semidefiniteness is preserved under taking the complex conjugate).
% The final step of the proof consists in relating condition (\ref{eq:DPS2}) to positive semidefiniteness of the blocks   $M_t(L)[I_{1,s'}]$. 
For $\rho_{AB[t]}=L(xx^*\ot (yy^*)^{\ot t})$ and $s\in \{0,1,\ldots,t\}$, we have $(\rho_{AB[t]})^{T_{B[s]}}
= L(xx^*\ot (\oy \oy^*)^{\ot s}\ot (yy^*)^{\ot (t-s)})$ and thus $(\rho_{AB[t]})^{T_{B[s]}}
\succeq 0$ if and only if
$\widetilde \rho:= L(xx^*\ot  (yy^*)^{\ot (t-s)} \ot (\oy \oy^*)^{\ot s})\succeq 0.$
The final step now  relies on the fact (see relation (67) in \cite{GLS_2021}) that $M_{1,t}(L)[I_{1,t-2s}]$ is identical to the matrix obtained from $\widetilde\rho$,  
    %with regards to the symmetry in the $B$ register as described above, i.e., 
    by removing its identical rows and columns (arising from the Bose symmetry of the $t$ B-registers, as explained above). Hence, condition (ii) is equivalent to (\ref{eq:DPS2}), as desired.
\qed\end{proof}

\subsection{The DPS hierarchy for CLDUI and LDOI states  in the moment formalism}\label{sec:momDPSCLDUI}

We now specialize the result of Lemma \ref{lem:DPSmom} for 
%DPS hierarchy in the moment formalism 
 bipartite states with CLDUI or LDOI sparsity structure. For $\alpha\in \N^n$, $\alpha\equiv 0$ (mod 2) means that $\alpha_i$ is even for all $i\in [n]$.

\begin{lemma}\label{lem:DPS-mom-LDOI}
 Let $\rho_{AB}\in\CLDUI_n$ (resp.,  $\rho_{AB}\in\LDOI_n$).
Then,     $\rho_{AB}\in\DPS^{(t)}_n$ if and only if there exists a Hermitian linear functional 
$L\in (\C[x,\ox,y,\oy]_{2,2t})^*$ satisfying   (\ref{eq:L0}),   conditions (i),(ii) from Lemma \ref{lem:DPSmom}, and, for any $\gamma,\gamma',\delta,\delta'\in \N^n$ such that $ |\gamma|=|\gamma'|=1$ and $ |\delta|=|\delta'|=t$,
%   Then, $\rho_{AB}$ has a $\DPS^{(t)}$-certificate,  provided by $L\in\C^*[x,\ox,y,\oy]_{=2,=2t}$ as in (\ref{eq:Lrho}),  which can be chosen to satisfy the   condition
      $$L(x^\gamma\ox^{\gamma'}y^\delta\oy^{\delta'}) = 0\quad 
      \text{ if } \gamma+\delta'\not=\gamma'+\delta \quad (\text{resp., if } \gamma+\gamma'+\delta+\delta'\not\equiv 0 \ \text{(mod } 2)).           $$
%       if $\gamma+\delta'\not=\gamma'+\delta$ (resp.,
 %     if  $\gamma+\gamma'+\delta+\delta'\not\equiv 0$ (mod 2)).   
% for $\gamma,\gamma',\delta,\delta'\in \N^n$ such that $ |\gamma+\gamma'|=2$ and $ |\delta+\delta'|=2t.$
\end{lemma}

\begin{proof}
%The proof builds up on  Lemma \ref{lem:DPSmom}.  
Consider a $\DPS^{(t)}$-certificate $\rho_{AB[t]}$ for $\rho_{AB}$ and the associated linear functional $L$ defined via (\ref{eq:Lrho}) as in Lemma \ref{lem:DPSmom}. In particular, relations (\ref{eq:Lnot0}) and (\ref{eq:L0}) hold.
%Then, for  $i_0\ui, j_0\uj\in [n]^{t+1}$, we have
%$$(\rho_{AB[t]})_{i_0\ui,j_0\uj}= \left(L(xx^*\ot (yy^*)^{\ot t})\right)_{i_0\ui,j_0\uj}
%=L(x^\gamma\ox^{\gamma'}y^\delta\oy^{\delta'}),$$
%    setting $\gamma = \alpha(i_0)=e_{i_0}$, $\gamma'=\alpha(j_0)=e_{j_0}$, $\delta=\alpha(\ui)$, and $\delta'=\alpha(\uj)$. 
    Assume first  $\rho_{AB[t]}\in\CLDUI^{(t)}_n$. Then, by   Lemma \ref{lem:suppCLDUIt} (iii) (and the definition of the sparsity pattern $\Omega^{(t)}_n$ in (\ref{def:Omegat})), $(\rho_{AB[t]})_{i_0\ui,j_0\uj}=0$  if $e_{i_0}+\alpha(\uj)\not=e_{j_0}+\alpha(\uj)$, which implies 
    $L(x^\gamma \ox^{ \gamma'}    y^\delta \oy^{\delta'})=0$ if $\gamma+\delta'\not=\gamma'+\delta$.   %concludes the proof in the CLDUI case.
     %Theorem \ref{theo:CLDUItDPSCert} 
 %. This concludes the case $\rho_{AB}\in\CLDUI$. 
    
    Assume now $\rho_{AB[t]}\in\LDOI^{(t)}$. Then, its sparsity pattern is provided by the set $\Psi^{(t)}_n$ from  (\ref{def:Psit}). Hence,   $(\rho_{AB[t]})_{i_0\ui,j_0\uj}=0$ if $\alpha(i_0\ui j_0\uj)\ne 0$ (mod 2), which implies $L(x^\gamma \ox^{ \gamma'}    y^\delta \oy^{\delta'})    =0$ if 
    $\gamma+\gamma'+\delta+\delta'\not\equiv 0    $ (mod 2).
 %           The case where $\rho_{AB[t]}\in\LDOI^{(t)}$ follows very similarly, since the right hand side vanishes by Theorem \ref{theo:LDUILDOItDPSCert} for $e_{i_0}+\alpha(\uj)+e_{j_0}+\alpha(\uj)$ not entrywise even.
\qed\end{proof}

%Below we display the block sizes corresponding to implementing the DPS hierarchy as an SDP in the three different regimes. Consider for example the first table at $n=3,t=2$. In the tensor formulation, one has to implement an SDP with three different blocks of size 27 corresponding to $\rho_{AB[2]},\rho_{AB[2]}^{T_{B[1]}}$ and $\rho_{AB[2]}^{T_{B[2]}}$. In the moment formulation both $\rho_{AB[2]}$ and $\rho_{AB[2]}^{T_{B[2]}}$ reduce to blocks of size 18 and $\rho_{AB[2]}^{T_{B[1]}}$ remains of size 27, which coincides with the values displayed in the table. We will now argue why these are the precise size reductions that occur. 

In Tables \ref{tab:GenericBlockSize}, \ref{tab:LDOIBlockSize} and \ref{tab:CLDUIBlockSize} below, we display the block sizes that arise when reformulating the DPS hierarchy in the moment formalism,   in the following three  regimes:  $\rho_{AB}\in \Herm(\C^n\ot \C^n)$ (generic case),  $\rho_{AB}\in \LDOI_n$ (LDOI case),  $\rho_{AB}\in \CLDUI_n$ (CLDUI case), respectively, We display these block sizes  for  $t=2,3,4,5,6,7$ and $n=3,4,5$, and the notation   $m_k$ is used to indicate that there are $k$ blocks of size $m$. For instance,  for $(t,n)=(2,3)$,   the entry  `$27_1, 18_2$' in Table \ref{tab:GenericBlockSize}  indicates that the SDP implementing the $\DPS^{(2)}$-relaxation for generic states in $\Herm(\C^3\ot \C^3)$ involves one block of size 27 and two blocks of size 18.
We next give some more details for the case $(t,n)=(2,3)$.

\begin{example}\label{ex:t2n3} 
We consider the case $(t,n)=(2,3)$ and discuss the associated entries in Tables \ref{tab:GenericBlockSize}  and \ref{tab:CLDUIBlockSize}, which  give the  block sizes  in the reduced SDP's modeling existence of a $\DPS^{(2)}$-certificate in the moment formulation, for a generic or CLDUI bipartite state $\rho_{AB}\in \Herm(\C^3\ot \C^3)$. 
%$\rho_{AB[2]}\in \Herm(\C^3\ot (\C^3)^{\ot 2})$ for   $\rho_{AB}\in \Herm(\C^3\ot \C^3)$.  

\smallskip
Consider first the generic case. In  the tensor formulation, the SDP involves  three   blocks  of size $3^3=27$, corresponding to $\rho_{AB[2]},\rho_{AB[2]}^{T_{B[1]}}$ and $\rho_{AB[2]}^{T_{B[2]}}$. In the moment formulation, both $\rho_{AB[2]}$ and $\rho_{AB[2]}^{T_{B[2]}}$ reduce to blocks of size 18 and $\rho_{AB[2]}^{T_{B[1]}}$ remains of size 27, thus matching the entry `$27_1, 18_2$'   in Table~\ref{tab:GenericBlockSize}. We will now argue why these  size reductions do occur. 

%We start by considering the rows of $\rho_{AB[2]}$ indexed by $(i_0i_1i_2)$ and $(j_0j_1j_2)$. If $i_0=j_0$ and $\alpha(i_1,i_2)=\alpha(j_1,j_2)$ it follows by condition \eqref{eq:DPS1} that $(\rho_{AB[2]})_{i_0 i_1 i_2,k_0 k_1 k_2}=(\rho_{AB[2]})_{j_0 j_1 j_2,k_0 k_1 k_2}$ for every $(k_0k_1k_2)$. The single corresponding row of $M_t(L)$ is indexed by $\alpha=e_{i_0}=e_{j_0}$, $\beta=\alpha(i_1,i_2)=\alpha(j_1,j_2)$ and $\beta'=0.$ All possible reductions of this form are given by $(i_0,1,2)$ and $(i_0,2,1)$, $(i_0,2,3)$ and $(i_0,2,3)$ and finally $(i_0,3,1)$ and $(i_0,1,3)$. Therefore, the overall size reduction of $\rho_{AB[2]}$ is given by deleting 9 rows and columns which results in a block of size 18.
First, let us see why $\rho_{AB[2]}$ reduces to a block of size 18. Consider two rows of $\rho_{AB[2]}$ indexed by  distinct   $i_0i_1i_2, j_0j_1j_2\in [3]^2$. If $i_0=j_0$ and $\{i_1,i_2\}=\{j_1,j_2\}$ %$\alpha(i_1i_2)=\alpha(j_1j_2)$ 
then, by (\ref{eq:DPS1}),  these two rows are equal and they correspond to a single row of $M_{1,2}(L)$ indexed by $(\alpha=e_{i_0}, \alpha'=0,\beta=\alpha(i_1i_2), \beta'=0)$. All possible reductions of this form    arise for  $(i_012,i_021)$,  $(i_013,i_031)$, and $(i_023,i_032)$ for  $i_0\in [3]$. Hence, one can remove 9 rows/columns from $\rho_{AB[2]}$ so that it reduces to a block of size $27-9=18$. 

%Running the same analysis for $(\rho_{AB[2]}^{T_{B[2]}})$ yields that the same rows coincide as for $(\rho_{AB[2]}$. A row indexed by $(i_0i_1i_2)$ corresponds to the row of $M_t(L)$ indexed by $\alpha=e_{i_0}$, $\beta=0$ and $\beta'=\alpha(e_{i_1}e_{i_2})$.
The same analysis applies to   $\rho_{AB[2]}^{T_{B[2]}}$ since the partial conjugate is applied to both $B$-registers, and thus $\rho_{AB[2]}^{T_{B[2]}}$ also reduces to a block of size 18. Note that, as expected,  the row of $\rho_{AB[2]}^{T_{B[2]}}$
indexed by $i_0i_1i_2$  corresponds to the row of $M_{1,2}(L)$ indexed by $(\alpha=e_{i_0},\alpha'=0,\beta=0,\beta'=\alpha(i_1i_2))$.

%We now want to derive when when $(\rho_{AB[2]}^{T_{B[1]}})_{i_0i_1i_2,k_0k_1k_2}=(\rho_{AB[2]}^{T_{B[1]}})_{j_0j_1j_2,k_0k_1k_2}$ holds for all $(k_0k_1k_2)$. Using the defintion of the partial trace this is equivalent to asking when $(\rho_{AB[2]})_{i_0k_1i_2,k_0i_1k_2}=(\rho_{AB[2]})_{j_0k_1j_2,k_0j_1k_2}$ holds. Condition \eqref{eq:DPS1} is only applicable when $i_0=j_0$, $\alpha(k_1i_2)=\alpha(k_1j_2)$ and $i_1=j_1$, implying that $(i_0i_1i_2)=(j_0j_1j_2)$. Therefore, no size reduction occurs. The row of $(\rho_{AB[2]}^{T_{B[1]}})$ indexed by $(i_0i_1i_2)$ corresponds to the row of $M_t(L)$ indexed by $\alpha=e_{i_0}$, $\beta=e_{i_2}$ and $\beta'=e_{i_1}$.
Consider now $\rho_{AB[2]}^{T_{B[1]}}$. By the definition of the partial transpose, for   $i_0i_1i_2,$ $j_0j_1j_2,k_0k_1k_2\in [3]^3$,  $(\rho_{AB[2]}^{T_{B[1]}})_{i_0i_1i_2,k_0k_1k_2}=(\rho_{AB[2]}^{T_{B[1]}})_{j_0j_1j_2,k_0k_1k_2}$    if and only if 
$(\rho_{AB[2]})_{i_0k_1i_2,k_0i_1k_2}= (\rho_{AB[2]})_{j_0k_1j_2, k_0j_1k_2}$. The latter equality holds using condition (\ref{eq:DPS1})  only if $i_0=j_0$, $\{k_1,i_2\}=\{k_1,j_2\}$, $\{i_1,k_2\}=\{j_1,k_2\}$,
%$\alpha(k_1i_2)=\alpha(k_1j_2)$ and $\alpha(i_1k_2)=\alpha(j_1k_2)$, 
which implies $i_0i_1i_2=j_0j_1j_2$. Hence, no size reduction can occur.  Note that the row of $\rho_{AB[2]}^{T_{B[1]}}$ indexed by $i_0i_1i_2$ corresponds to the row of $M_{1,2}(L)$ indexed by $(\alpha=e_{i_0},\alpha'=0,\beta=e_{i_2},\beta'=e_{i_1})$.

\smallskip
Assume now $\rho_{AB}\in \CLDUI_3$.
 %Adding up the block sizes of table \ref{table0}, \ref{table1} and \ref{table2} one gets that $6_1,5_6,3_6,2_6,1_{15}$ are the necessary block sizes and amounts for setting up the SDP in the tensor formulation in the CLDUI regime. By applying the size reduction discussed in the generic case (i.e. $(i_0i_1i_2)$ and $(j_0j_1j_2)$ index the same rows in $\rho_{AB[2]}$ and $\rho_{AB[2]}^{T_{B[2]}}$ if $i_0=j_0$ and $\alpha(i_1i_2)=\alpha(j_1j_2)$) one gets for example that the block indexed by $\{(111),(212),(221),(313),(331)\}$ reduces to a block indexed by $(111),(212),(313)$. Doing this for every maximal Clique one gets precisely the block sizes and amounts displayed in table \ref{tab:CLDUIBlockSize}.
In the tensor formalism, the SDP modeling a $\DPS^{(2)}$-certificate $\rho_{AB[2]}\in \CLDUI_3^{(2)}$ involves $5_3, 2_3,1_6$ blocks for $\rho_{AB[2]}$ (see Table \ref{table0}), $5_3,2_3,1_6$ blocks for $\rho_{AB[2]}^{T_{B[1]}}$ (see Table \ref{table1}), and $6_1,3_6,1_3$ blocks for $\rho_{AB[2]}^{T_{B[2]}}$ (see Table \ref{table2}), leading to a total of $6_1,5_6,3_6,2_6,1_{15}$ blocks. In comparison, the SDP in the moment formulation involves a total of $5_3,3_4,2_9,1_{18}$ blocks, as can be seen in Table \ref{tab:CLDUIBlockSize}. To see why this size reduction takes place, it suffices to check all the maximal cliques that are listed in Tables \ref{table0}, \ref{table1}, \ref{table2}. For each such clique, one just checks whether it contains two distinct  indices $i_0i_1i_2$  and $j_0j_1j_2$ such that $i_0=j_0$ and $\{i_1,i_2\}=\{j_1,j_2\}$, in which case one of them is removed from  the clique. So, in Table \ref{table0}, each clique of size 5 (resp., size 2) reduces to size 3 (resp., size 1), yielding block sizes $3_3,1_9$ after reduction. In Table \ref{table1}, no size reduction applies, so one keeps $5_3,2_3,1_6$ as block sizes.  Finally, in Table \ref{table2}, the clique of size 6 reduces to size 3 and each clique of size 3 reduces to size 2, yielding block sizes $3_1,2_6,1_3$. Putting things together one indeed arrives at a total of $5_3,3_4,2_9,1_{18}$ blocks as shown in Table~\ref{tab:CLDUIBlockSize}.
\end{example}

\begin{table}[h!]
    \centering
    \begin{tabular}{c|c|c|c}
         $t\backslash n$ & 3 & 4 & 5 \\ \hline
         2 & $27_1$, $18_2$ & $64_1$, $40_2$ & $125_1$, $75_2$ \\ \hline
         3 & $54_2$, $30_2$ & $160_2$, $80_2$ & $375_2$, $175_2$ \\ \hline
         4 & $108_1$, $90_2$, $45_2$ & $400_1$, $320_2$, $140_2$ & $1125_1$, $875_2$, $350_2$\\\hline
         5 & $180_2$, $135_2$, $63_2$ & $800_2$, $560_2$, $224_2$ & $2625_2$, $1750_2$, $630_2$\\\hline
         6 & $300_1$, $270_2$, $189_2$, $84_2$ & $1600_1$, $1400_2$, $896_2$, $336_2$ & $6125_1$, $5250_2$, $3150_2$, $1050_2$\\\hline
         7 & $450_2$, $378_2$, $252_2$, $108_2$ & $2800_2$, $2240_2$, $1344_2$, $480_2$ & $12250_2$, $9450_2$, $5250_2$, $1650_2$
    \end{tabular}
    \caption{Block sizes in the moment formulation of $\DPS^{(t)}$-certificates for generic states}
    \label{tab:GenericBlockSize}
\end{table}

\begin{table}[h!]
    \centering
    \begin{tabular}{c|c|c|c}
         $t\backslash n$ & 3 & 4 & 5 \\ \hline
         2 & $7_{3}$, $6_{1}$, $5_{6}$, $3_{2}$ & $10_{4}$, $7_{8}$, $6_{4}$, $3_{8}$ & $13_{5}$, $9_{10}$, $6_{10}$, $3_{20}$ \\ \hline
         3 & $15_{2}$, $13_{6}$, $9_{2}$, $7_{6}$ & $28_{2}$, $20_{12}$, $16_{2}$, $12_{2}$, $10_{12}$, $4_{2}$ & $45_{2}$, $27_{20}$, $25_{2}$, $13_{20}$, $12_{10}$, $4_{10}$ \\ \hline
         4 & $28_{3}$, $24_{1}$, $23_{6}$,  & $58_{4}$, $46_{8}$, $42_{4}$, & $99_{5}$, $77_{10}$, $60_{10}$, $47_{20}$, $35_{10}$, \\
         & $21_{2}$, $12_{6}$, $9_{2}$ & $34_{8}$, $22_{8}$, $13_{8}$ &  $30_{1}$, $20_{2}$, $17_{20}$, $5_{2}$\\\hline
         5& $48_2$, $44_6$, $36_2$, & $124_2$, $100_{12}$, $88_2$, $76_2$, $70_{12}$, & $255_2$, $180_{20}$, $175_2$, $121_{20}$, $114_{10}$, \\
          & $33_6$, $18_2$, $15_6$ & $52_2$, $40_2$, $28_{12}$, $16_2$ & $75_2$, $73_{10}$, $45_{20}$, $21_{10}$\\\hline
        6& $76_{3}$, $72_{1}$, $69_{6}$, $63_{2}$, & $218_{4}$, $193_{8}$, $182_{4}$, $157_{8}$, & $479_{5}$, $417_{10}$, $350_{10}$, $297_{20}$, $255_{10}$, $230_{1}$, \\
        & $48_{6}$, $45_{2}$, $22_{6}$, $18_{2}$ & $124_{8}$, $100_{8}$, $50_{8}$, $34_{8}$ & $195_{2}$, $177_{20}$, $105_{2}$, $95_{10}$, $55_{20}$, $25_{2}$\\\hline
         & $117_{2}$, $111_{6}$, $99_{2}$, & $404_{2}$, $350_{12}$, $328_{2}$, $296_{2}$, & $1045_{2}$, $825_{2}$, $817_{20}$, $633_{20}$, \\
        7 & $93_{6}$, $66_{2}$, $62_{6}$,& $280_{12}$, $232_{2}$, $200_{2}$, $168_{12}$, & $607_{10}$, $475_{2}$, $459_{10}$, $355_{20}$, \\
         &  $30_{2}$, $26_{6}$ & $136_{2}$, $80_{2}$, $60_{12}$, $40_{2}$ & $245_{10}$, $175_{2}$, $115_{20}$, $65_{10}$
    \end{tabular}
    \caption{Block sizes in the moment formulation of $\DPS^{(t)}$-certificates for states in $\LDOI_n$}
    \label{tab:LDOIBlockSize}
\end{table}

\begin{table}[h!]
    \centering
    \begin{tabular}{c|c|c|c}
         $t\backslash n$ & 3 & 4 & 5 \\ \hline
         2 & $5_3$, $3_4$, $2_9$, $1_{18}$ & $7_4$, $4_4$, $3_4$, $2_{24}$, $1_{40}$ & $9_5$, $5_5$, $3_{10}$, $2_{50}$, $1_{75}$ \\ \hline
         3 & $9_1$, $7_3$, $5_9$, & $16_1$, $10_6$, $7_{16}$, $4_{11}$, & $25_1$, $13_{10}$, $9_{25}$, $5_{15}$, \\
         & $3_9$, $2_{18}$, $1_{30}$ & $3_{16}$, $2_{60}$, $1_{80}$ & $4_5$, $3_{50}$, $2_{150}$, $1_{175}$ \\\hline
         4 & $12_3$, $9_4$, $7_9$, $5_{18}$, & $22_4$, $16_4$, $13_4$, $10_{24}$, $7_{40}$, & $35_5$, $25_5$, $17_{10}$, $13_{50}$, $9_{75}$, \\
          &  $3_{16}$, $2_{30}$, $1_{45}$ &  $4_{24}$, $3_{40}$, $2_{120}$, $1_{140}$ &  $5_{36}$, $4_{25}$, $3_{150}$, $2_{350}$, $1_{350}$\\ \hline
          & $18_1$, $15_3$, $12_9$, & $40_1$, $28_6$, $22_{16}$, $16_{11}$, & $75_1$, $45_{10}$, $35_{25}$, $25_{15}$, $21_{5}$, \\
          5 & $9_9$, $7_{18}$, $5_{30}$, &  $13_{16}$, $10_{60}$, $7_{80}$, $4_{45}$, & $17_{50}$, $13_{150}$, $9_{175}$, $5_{75}$, \\
           & $3_{25}$, $2_{45}$, $1_{63}$ & $3_{80}$, $2_{210}$, $1_{224}$ & $4_{75}$, $3_{350}$, $2_{700}$, $1_{630}$\\\hline
         & $22_3$, $18_4$, $15_9$, $12_{18}$, & $50_4$, $40_4$, $34_4$, $28_{24}$, $22_{40}$, & $95_5$, $75_5$, $55_{10}$, $45_{50}$, $35_{75}$, \\
         6  & $9_{16}$, $7_{30}$, $5_{45}$, & $16_{24}$, $13_{40}$, $10_{120}$, $7_{140}$, & $25_{36}$, $21_{25}$, $17_{150}$, $13_{350}$, $9_{350}$, \\
           & $3_{36}$, $2_{63}$, $1_{84}$ & $4_{76}$, $3_{140}$, $2_{336}$, $1_{336}$ & $5_{141}$, $4_{175}$, $3_{700}$, $2_{1260}$, $1_{1050}$\\\hline
          & $30_1$, $26_3$, $22_9$, $18_9$, & $80_1$, $60_6$, $50_{16}$, $40_{11}$, $34_{16}$, & $175_1$, $115_{10}$, $95_{25}$, $75_{15}$, $65_5$, $55_{50}$, \\
         7& $15_{18}$, $12_{30}$, $9_{25}$, $7_{45}$, & $28_{60}$, $22_{80}$, $16_{45}$, $13_{80}$, $10_{210}$, & $45_{150}$, $35_{175}$, $25_{75}$, $21_{75}$, $17_{350}$, $13_{700}$, \\
          & $5_{63}$, $3_{49}$, $2_{84}$, $1_{108}$ & $7_{224}$, $4_{119}$, $3_{224}$, $2_{504}$, $1_{480}$ & $9_{630}$, $5_{245}$, $4_{350}$, $3_{1260}$, $2_{2100}$, $1_{1650}$
    \end{tabular}
    \caption{Block sizes in the moment formulation of $\DPS^{(t)}$-certificates for states in $\CLDUI_n$}
    \label{tab:CLDUIBlockSize}
\end{table}

We now want to give some insight on how the size of the biggest block involved in the SDP in the generic setting compares to the biggest block in the LDOI setting. We restrict ourselves to the LDOI case, since going from generic to LDOI yields a far greater reduction in size than going from LDOI to CLDUI. Furthermore, the analysis in the LDOI case is significantly easier because the LDOI condition in 
Lemma \ref{lem:DPS-mom-LDOI} %Lemma \ref{lem:DPSmom}
 is `global', in the same sense as explained before Lemma \ref{lem:sizecliqueLDOI}.
 %that it involves only a parity argument on the whole sequence.} %it is invariant under taking partial transposes. 
 We fix $n\in\N$, $t\ge 1$, $s'\in [-t,t]$ such that $s'\equiv t$ (mod 2).  Define 
 $$N_{s',t}:=\max_{(\tilde\alpha,\tilde\beta,\tilde\beta')\in I_{1,s'}}|\{(\alpha,\beta,\beta')\in I_{1,s'}:\alpha+\beta+\beta'\equiv\tilde\alpha+\tilde\beta+\tilde\beta'\text{ (mod 2)}\}|.$$

\begin{lemma}\label{lem:LDOIBlockRatio}
    Let $n,t\in\N$ and $s'\in\{-t,-t+2,\dots,t-2,t\}$. Then, we have
    $$\frac{N_{s',t}}{|I_{1,s'}|}\leq\frac{(\frac{t+|s'|}{2})!}{n^{\lceil\frac{t+|s'|}{4}\rceil}}.$$
   Hence,     the ratio of the largest block size   in the LDOI regime to the largest block size in the generic regime is given by 
   $$\max _{s'} \frac{N_{s',t}}{|I_{1,s'}|}\leq \max_{s'} \frac{(\frac{t+|s'|}{2})!}{n^{\lceil\frac{t+|s'|}{ {4}}\rceil}}\le {t!\over {n^{\lceil {t\over 4}\rceil}}}.$$
\end{lemma}

\begin{proof}
Fix an arbitrary $(\tilde\alpha,\tilde\beta,\tilde\beta')\in I_{1,s'}$. Assume first $0\leq s'$. We want to count how many $(\alpha,\beta,\beta')\in I_{1,s'}$ exist such that $\alpha+\beta+\beta'\equiv\tilde\alpha+\tilde\beta+\tilde\beta'\text{ (mod 2)}$. Define the odd support  of $\tilde\alpha+\tilde\beta+\tilde\beta'$ as the set $O\subseteq[n]$   consisting of the indices $i\in [n]$ for which the $i$th entry of the vector $\tilde\alpha+\tilde\beta+\tilde\beta'$ is odd. Note that $|O|\equiv |\tilde\alpha+\tilde\beta+\tilde\beta'|=t+1$ (mod 2). Therefore, we need to count the sequences $(\alpha,\beta,\beta')\in I_{1,s'}$ whose odd support is $O$.

        First, we pick $\alpha$ and $\beta'$ arbitrarily, for which  there are   $n\cdot\binom{n-1+\frac{t-s'}{2}}{n-1}$ options. Denote the odd support of $\alpha+\beta'$ as $O'\subseteq[n]$. Now we will select $\beta$. In order to fulfil the desired congruence condition,  the odd support of $\beta$ must be equal to the symmetric difference $O\Delta O'$ of $O$ and $O'$. Note that $|O'|\equiv |\alpha+\beta'|=1+(t-s')/2$ (mod 2).
      %  has to be of the form $\beta=2\gamma+\chi^{O\triangle O'}$ where $O\triangle O'=O\setminus O'\cup O'\setminus O$ is the symmetric difference of $O,O'$ and $\chi^A\in\N^n$ denotes the characteristic vector of a set $A\subseteq[n]$ (i.e., $(\chi^A)_i=1$ if $i\in A$ and $(\chi^A)_i=0$ if $i\not\in A$). 
 So, one must have $\beta=  2\gamma+\chi^{O\triangle O'}$, where $\gamma\in\N^n$ and $\chi^{O\Delta O'}$ is the characteristic vector of $O\Delta O'$ (with $i$th entry 1 if $i\in O\Delta O'$ and 0 otherwise).  Since $|\beta|=\frac{t+s'}{2}$, it follows that %$\gamma\in\N^n$ has to be of weight 
    $|\gamma|=(\frac{t+s'}{2}-|O\triangle O'|)/2$. This is a valid weight because 
    $|O\Delta O'|\equiv |O|+|O'|\equiv t +(t-s')/2 \equiv (t+s)/2$ (mod 2).
  %  $\frac{t+s}{2}$ and $|O\triangle O'|$ are of the same parity which can be seen by the following congruence relation $$|O\triangle O'|\equiv |O|+|O'|\equiv t+\frac{t-s}{2}\equiv\frac{t+s}{2}\text{ (mod 2)}.$$
    In total we get $\binom{n-1+(\frac{t+s'}{2}-|O\triangle O'|)/2}{n-1}$ options for $\gamma$, which also gives the number  of options for $\beta$. This term is clearly maximized for $|O\triangle O'|$ minimal, which depends on the parity of $\frac{t+s'}{2}$. Set $a=0$ if $\frac{t+s'}{2}$ is even and set $a=1$ if $\frac{t+s'}{2}$is odd. We can conclude  that
    \begin{align}\label{eq:NstUpperBound}
    \begin{split}
        N_{s',t}& \leq \  n\binom{n-1+\frac{t-s'}{2}}{n-1}\binom{n-1+(\frac{t+s'}{2}-a)/2}{n-1}\\
        & =n\binom{n-1+\frac{t-s'}{2}}{n-1}\binom{n-1+\lfloor\frac{t+s'}{4}\rfloor}{n-1}.  
    \end{split}\end{align}
Since $|I_{1,s'}|=n{n-1+(t-s')/2 \choose n-1}{n-1+(t+s')/2\choose 2}$, we obtain that   
 $$\frac{N_{s',t}}{|I_{1,s'}|}\leq\frac{\binom{n-1+\lfloor\frac{t+s'}{4}\rfloor}{n-1}}{\binom{n-1+\frac{t+s'}{2}}{n-1}}\leq\big(\frac{t+s'}{2}\big)! \cdot \frac{n^{\lfloor\frac{t+s'}{4}\rfloor}}{n^{\frac{t+s'}{2}}},$$ where we get the second inequality by applying the following (easy to check) inequality
    \begin{align*}
        {n^t\over t!}\le {n-1+t\choose n-1}\le n^t \ \text{ for any integers } n,t\ge 1.
    \end{align*}
    One can quickly see that this formula coincides with the formula given in the lemma. Consider now the case $s'\leq0.$ Then we can   redo the same analysis, only altering the approach by switching the roles of $\beta,\beta'$, i.e., first selecting $\beta$ arbitrarily and after that selecting $\beta'$.  This leads to replacing $\frac{t+s'}{2}$ with $\frac{t-s'}{2}$ in formula \eqref{eq:NstUpperBound}. This completes the proof.
\qed\end{proof}

\begin{remark}
    In the case $s=t=1$  the formula of Lemma \ref{lem:LDOIBlockRatio} gives   an upper bound of $\frac{1}{n}$. This exactly coincides with the known sparsity ratio for the  LDOI pattern (and the CLDUI/LDUI patterns) (recall relation (\ref{eq:rho-ABC})).
    \end{remark}

\subsection{Runtime comparison for a class of examples based on the PPT$^2$ conjecture}\label{sec:runtime}

The goal of this section is to present a comparison of the runtimes for an implementation of the DPS hierarchy in the three   regimes: for generic, CLDUI,  LDOI bipartite states. We first give an overview of the setup for our benchmark examples, which is motivated by the PPT$^2$ conjecture in quantum information. The PPT$^2$ conjecture is usually formulated on the `channel side'. We will state it here on the `state side' and only explain the necessary steps for setting up our benchmark instances. Further details can be found in Appendix \ref{A:PPT^2Conj}. The PPT$^2$ conjecture states that the composition of any two {PPT} linear maps from $\C^{n\times n}$ to $\C^{n\times n}$ is entanglement breaking (see Conjecture~\ref{conj:PPT2}). It has been shown to hold  for $n\leq3$ \cite{CMHW-2018,CYT-2019} and for the linear maps arising from one state in $\CLDUI_n\cup\LDUI_n$ and the other one in $\LDOI_n$ \cite{Singh-Nechita-PPT2,Nechita-Park_2025}. It is still open whether it also holds when both states lie in $\LDOI_n$.
For this reason we next formulate the PPT$^2$ conjecture for LDOI states, posed in \cite[Conjecture 5.1]{Singh-Nechita-PPT2}.

\begin{conjecture}[The PPT$^2$ conjecture for LDOI states]\label{conj:LDOIstates}
Let $(X_k,Y_k,Z_k)\in \R^{n\times n}\times\Herm^n\times\Herm^n$ such that $\diag(X_k)=\diag(Y_k)=\diag(Z_k)$ for $k=1,2$.
%   Let  $(X_1,Y_1,Z_1),(X_2,Y_2,Z_2)\in\R^{n\times n}\times\Herm^n\times\Herm^n$ such that $\diag(X_1)=\diag(Y_1)=\diag(Z_1)$ and $\diag(X_2)=\diag(Y_2)=\diag(Z_2)$. 
Assume that the corresponding LDOI states are positive semidefinite and have positive partial transpose, i.e.,  $\rho_{(X_k,Y_k,Z_k)}\in \DPS^{(1)}_n$ for $k=1,2$.
%$\rho_{(X_1,Y_1,Z_1)},\rho_{(X_2,Y_2,Z_2)}\in\DPS^{(1)}_n$.
Then,   $\rho_{(X_3,Y_3,Z_3)}\in\SEP_n$, where we define 
\begin{align*}
(X_3,Y_3,Z_3)& =(X_1,Y_1,Z_1)\bullet(X_2,Y_2,Z_2)\\
& :=(X_1X_2,Y_1\circ Y_2+Z_1\circ Z_2^T+DY_3,Y_1\circ Z_2+Z_1\circ Y_2^T+DZ_3),
\end{align*}
$DY_3:=\Diag(X_1X_2-Y_1\circ Y_2-Z_1\circ Z_2^T), DZ_3:=\Diag(X_1X_2-Y_1\circ Z_2-Z_1\circ Y_2^T)$, so that 
 $\diag(X_3)=\diag(Y_3)=\diag(Z_3)$ holds.
% \footnote{Both $DY_3$ and $DZ_3$ only serve to ensure, that the matrix triple $(X_3,Y_3,Z_3)$ shares the same diagonal.}.
\end{conjecture}

Suppose we want to construct a counterexample to the above conjecture. 
As mentioned above, the PPT$^2$ conjecture holds for $n\leq3$ and for states in $\CLDUI_n\cup\LDUI_n$; in addition, it holds 
%we have  $\rho_{(X_3,Y_3,Z_3)}\in\SEP_n$
 if either $\rho_{(X_1,Y_1,Z_1)}\in\SEP_n$, or $\rho_{(X_2,Y_2,Z_2)}\in\SEP_n$ (see  Lemma \ref{lem:conj-holds}).
Hence, the smallest case to investigate is $n=4$. So, we want to construct matrix triplets $(X,Y,Z)\in\R^{4\times 4}\times\Herm^4\times\Herm^4$, such that $\diag(X)=\diag(Y)=\diag(Z)$, $\rho_{(X,Y,Z)}\in\DPS^{(1)}_4\setminus\SEP_4$, and neither $Y$ nor $Z$ is a diagonal matrix. 
Then, we will use such triplets as arguments $(X_k,Y_k,Z_k)$ ($k=1,2$) to build their compositions $(X_3,Y_3,Z_3)$ that we will then test for existence of a $\DPS^{(t)}$-certificate. In case of a negative answer, one would obtain a counterexample to the PPT$^2$ conjecture.

For this, we start by considering the states $\rho_{a,a'}$ as defined in Definition \ref{def:rhoaap}. We will restrict\footnote{We choose $a'= 1/a$, which implies the largest violation of the conditions $a'\ge 1/a$ (to ensure $\rho_{a,a'}\in\DPS^{(1)}_3$) and $a'<1$ (to ensure $\rho_{a,a'}\not\in\DPS^{(2)}_3$).
}
to $(a,a')\in\{(2,1/2),(3,1/3),(4,1/4)\}$, in which case $\rho_{a,a'}\in\DPS^{(1)}_3\setminus\DPS^{(2)}_3$ in view of  Lemma~\ref{lem:rhoaaDPS1} and Theorem \ref{theo:rhoaa}.
 Then, $\rho_{a,a'}= \rho_{(\tilde X_{a},\tilde Y)}\not\in\SEP_3$, after setting
  $$\tilde X_{a}=\begin{pmatrix}
            1&a&1/a\\
            1/a&1&a\\
            a&1/a&1
        \end{pmatrix},\
        \tilde Y=\begin{pmatrix}
            1&1&1\\
            1&1&1\\
            1&1&1
        \end{pmatrix}.$$
%encodes the state $\rho_{a,a'}$. 
We   extend the matrices $\tilde X_{a},\tilde Y$ to the following $4\times 4$ matrices $X_{a},Y$:
%enlarge the ambient dimension to $n=4$ in the following way: we define 
\begin{align}\label{eq:XYaap}
X_{a}=\begin{pmatrix}
            1&a&1/a&1\\
            1/a&1&a&1\\
            a&1/a&1&1\\
            1&1&1&1
        \end{pmatrix},\
        Y=\begin{pmatrix}
            1&1&1&1\\
            1&1&1&1\\
            1&1&1&1\\
            1&1&1&1
        \end{pmatrix}.
        \end{align}
By construction,     $\rho_{(X_{a},Y)}\in\DPS^{(1)}_4$ (i.e., $\rho_{(X_{a},Y)}, \rho_{(X_{a},Y)}^{T_B}=\rho_{(X_{a},\cdot,Y)}\succeq 0$, which is easy to check using (\ref{eq:rhoXYZ-block}), (\ref{eq:rhoXYZ-block-PT})) and $\rho_{(X_{a},Y)}\not\in\SEP_4$ (since, otherwise, this would imply  $\rho_{(\tilde X_{a},\tilde Y)}\in\SEP_3$). Next, we generate matrices $Z$ by choosing random unit vectors $v_k\in\R^4$ for $k\in[4]$ and  defining $Z=( v_h^Tv_k )_{h,k=1}^4$ as their Gram matrix. Then, $\diag(Z)=\diag(X_{a})$  (since the $v_k$'s are unit vectors) and  %if $|Z_{i,j}|^2\leq (X_{a})_{i,j}(X_{a})_{j,i}$   for all $i,j\in [4]$, then 
$\rho_{(X_{a},Y,Z)}\in\DPS^{(1)}_4\setminus\SEP_4$. Indeed, 
$\rho_{(X_{a},Y,Z)}\in \DPS^{(1)}_4$, because $\rho_{(X_{a},Y,Z)}\succeq 0$ and $\rho_{(X_{a},Y,Z)}^{T_B}=\rho_{(X_{a},Z,Y)} \succeq 0$ (as $|Z_{i,j}|^2\leq (X_{a})_{i,j}(X_{a})_{j,i}=1$ for any $1\le i<j\le 4$ and using  (\ref{eq:rhoXYZ-block}), (\ref{eq:rhoXYZ-block-PT})). Moreover, if $\rho_{(X_{a},Y,Z)}$ would be separable then,  by Lemma \ref{lem:projection-psd-sep}, also its CLDUI projection $\rho_{(X_{a},Y)}$ would be separable, yielding a contradiction.% by Theorem \ref{theo:LDUILDOItDPSCert} and Lemma \ref{lem:projection-psd-sep}. 

%This way we generate a list of $Z[i]$ for $i\in[5]$ (which is valid for all three choices of $(a,a')$). We can now test whether $\rho_{(X,Y,Z[i])\bullet(X,Y,Z[j])}\not\in\SEP$ using the implementations of the $\DPS$ Hierarchy. Note, that $(X,Y,Z[i])\bullet(X,Y,Z[j])=(X,Y,Z[j])\bullet(X,Y,Z[i])$ because $X_1=X_2$, $Y=Y^T$ and $Z[i]^T=Z[i]$ (the $Z$'s are real and therefore symmetric).
We generate   five such random matrices $Z$, denoted as $Z[i]$ for $i\in [5]$. So, we get 15  triplets $(X_a,Y,Z[i])$ such that $\rho_{(X_a,Y,Z[i])}\in\DPS^{(1)}_4\setminus \SEP_4$, where  $X_a,Y$ are as in (\ref{eq:XYaap}) for  $a=2,3,4$. Now, the goal is to test whether the composed states\footnote{We also tested states of the form $\rho_{(X_a,Y,Z[i])\bullet(X_b,Y,Z[j])} $ with $a\ne b$. We do not include the runtimes for these states in Figure \ref{fig:runtimes} since the results are %the same as 
similar to those in the setup ($a=b$) described in the text.}
$\rho_{(X_a,Y,Z[i])\bullet(X_a,Y,Z[j])} $ are separable or entangled.
Note that both arguments commute,
%$(X_a,Y,Z[i])\bullet(X_a,Y,Z[j])=(X_a,Y,Z[j])\bullet(X_a,Y,Z[i])$ 
because %$X_1=X_2$, 
$Y $ and all $Z[i]$'s are real symmetric. % $Z[i]^T=Z[i]$ (the $Z$'s are real and therefore symmetric)
So, we consider only $1\le i\le j \le 5$, which leads to a total of $3\times 15=45$ 
different composed states $\rho_{AB}$ for which membership in $\DPS^{(t)}_4$ can be tested. 

By construction, these composed states $\rho_{AB}$ have a LDOI sparsity structure. We use them to test the behaviour of the SDP implementing existence of a $\DPS^{(t)}$-certificate in the three regimes: first, in the generic case (searching for a $\DPS^{(t)}$-certificate in $\Herm((\C^n)^{\ot (t+1)}$, thus ignoring the sparsity structure of $\rho_{AB}$), second in the LDOI regime (searching for a $\DPS^{(t)}$-certificate in $\LDOI_n^{(t)}$), and third in the CLDUI regime (searching for a $\DPS^{(t)}$-certificate for $\Pi_{\CLDUI}(\rho_{AB})$ in $\CLDUI_n^{(t)}$). 
We note that on all tested examples, a $\DPS^{(t)}$-certificate could be found. So, no instance of an entangled state was found, so that the status of the PPT$^2$ conjecture remains open for LDOI states.

\medskip
The graphic in Figure \ref{fig:runtimes} below   displays the different runtimes (in seconds) for the three regimes (generic, LDOI, CLDUI) and relaxation levels from $t=1$ to $t=7$. The runtime consists only of the time needed for solving the SDP. It excludes the preprocessing phase, since the preprocessing time remains in the seconds for every regime and every level $t$ displayed in Figure \ref{fig:runtimes}.
%\footnote{\ML{Question: for which computation was the SDP solving in the hours??}}. 
The experiments were carried out on a laptop running Windows 11 Home 64-bit with an AMD Ryzen 7 8845HS @ 3.8GHz, 8 cores, 16 threads and 16GB of Ram.
%Our code is available at (add link where to find the code) and was coded in Julia (insert ref) utilizing JuMP (insert ref) for problem formulation, and Hypatia (insert ref) as the semidfinite program solver.} 
We used the  programming language Julia \cite{Julia}  utilizing the JuMP   package \cite{Jump} for coding\footnote{{Our code is available at https://github.com/JonasBritz/ImplementationsDPSHierarchy.git.}} the problem formulation, and Hypatia \cite{coey2022solving}  as the semidefinite programming solver.

As expected, the generic implementation is the slowest and the CLDUI implementation is the fastest. For instance, for $t=3$, the runtimes are roughly  2.5 minutes, 10 seconds and 1 second in the generic, LDOI and CLDUI cases and, for $t=5$, roughly 30 seconds and 21 minutes for the CLDUI and LDOI cases.
 Note that the faster runtimes in the CLDUI regime allow to test membership in $\DPS^{(t)}_4$ up to level $t=7$, while one can only test up to level $t=5$ in the LDOI regime, and up to level $t=3$ in the generic regime. 
%at higher levels than in the LDOI regime. 
%Therefore, it is worthwhile to  use the CLDUI relaxation, in particular at the levels $t=6,7$ for which the LDOI implementation becomes unpractical to run repeatedly.
Hence, this also indicates that the CLDUI implementation can be used as a useful necessary condition for testing existence of a $\DPS^{(t)}$-certificate for generic $\rho_{AB}$, by applying it to its projection $\Pi_{\CLDUI}(\rho_{AB})$, since one can then go to higher levels $t$.
 
 \bigskip
\begin{figure}[h]
    \centering
    \includegraphics[width=\linewidth]{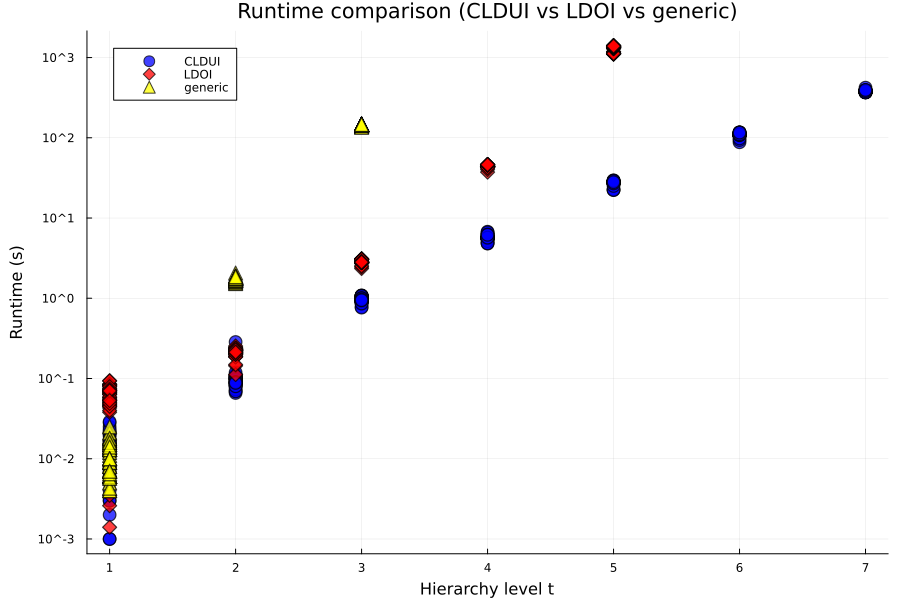}
    \caption{Graphic displaying the runtimes of the DPS hierarchy in three regimes: for CLDUI, LDOI and generic bipartite states}
    %three different implementation methods of the DPS hierarchy: CLDUI, LDOI and the generic version}
    \label{fig:runtimes}
\end{figure}

%\includegraphics[width=\linewidth]{} 
%\label{sec:implementation}

\section{Concluding remarks}\label{sec:conclusion}

In this paper we explore how structural properties of quantum bipartite states  can be exploited to design more efficient variations of the  semidefinite   hierarchy $\DPS^{(t)}_n$ for testing separability and detecting quantum entanglement.   
On the one hand, we consider bipartite (CLDUI, LDUI, LDOI) states with diagonal unitary invariance.
 These states admit  a    block diagonal sparsity pattern that can be transferred to any given level of the DPS hierarchy. By applying block diagonalization and, in addition, reformulating the DPS hierarchy in the moment formalism, one obtains substantial matrix size reduction in the SDP models,
 so that higher levels can be computed in practice. 
 
On the other hand, we consider Bose symmetric   bipartite states, i.e., states that are symmetric in their two registers. This symmetry property can again be transferred to any given level of the DPS hierarchy. This leads to the stronger symmetry-adapted $\tilDPS^{(t)}_n$ hierarchy. We characterize its  dual, in terms of existence of a r-sum-of-squares representation for an associated complex  polynomial. This characterization mirrors the known characterization of the dual of $\DPS^{(t)}_n$. In addition, it is   closely related to the sum-of-squares characterization of the approximation cone $\MK^{(t)}_n$ for $\COP_n$.

The completely positive cone $\CP_n$ and its dual $\COP_n$ play an important role within both settings. Indeed, separable bipartite states with CLDUI, LDUI (LDOI) sparsity structure can be captured by  pairwise (triplewise)  completely positive matrix pairs (triples), leading to the cones $\PCP_n,\TCP_n$ that generalize $\CP_n$. On the dual side, there is a close relationship between   $\COP_n$ and the dual of the cone of Bose symmetric separable states, which extends  to their respective conic approximations by $\MK^{(t)}_n$ and \smash{$(\tilDPS^{(t)}_n)^*$}.  

Hence, there are interesting connections among several topics, ranging from quantum information to semidefinite optimization, matrix cones, and polynomial optimization. 
We now conclude with mentioning some open  questions.

\medskip
The symmetry-adapted DPS hierarchy applies to Bose symmetric states, so we have the inclusion 
 $\tilDPS^{(t)}_n\subseteq \DPS^{(t)}_n\cap \BS(\C^n\ot \C^n)$ (the right most inclusion in (\ref{eq:DPS-BS-inclusion})). Equality holds for $t=1$ (since no state extension takes place). A natural question, already posed in \cite[Question 8.20]{GNP_2025},  is whether equality holds for larger $t\ge 2$. 
We expect the answer to be negative, yet we could not find a counterexample.  Let us briefly sketch our numerical approach for trying to  find a counterexample.

Let $n=5$ and $t=2$. For $C\in \MK^{(1)}_5\setminus \MK^{(0)}_5$,  consider the   semidefinite program
$$p^*_C:=\min\{\langle C, X\rangle: X\in \MS^5,\  \rho_{(X,X)}^{T_B} \in \DPS^{(2)}_5\}
$$
(which can be solved using the LDUI sparsity reduction). Note that, if we find a feasible $X$ such that $\langle C,X\rangle <0$, then $X\not\in (\MK^{(1)}_5)^*$ and thus, by Theorem~\ref{theo:DPS-CP}, $\rho_{(X,X)}^{T_B}\in \DPS^{(2)}_5\setminus \tilDPS^{(2)}_5$ while, by construction, $\rho_{(X,X)}^{T_B}\in\BS(\C^5\ot\C^5)$. Hence,  $p_C^*<0$ would imply strict inclusion  $\tilDPS^{(2)}_5\subset \DPS^{(2)}_5\cap \BS(\C^5\ot \C^5)$. The extreme rays of $\COP_5$ are fully characterized in \cite{Hildebrand-COP5}: those that do not belong to $\MK^{(0)}_5$ are the Horn matrix $H$ together with a class of parametrized matrices $T(\psi)$ (for $\psi\in \R^5$, $\psi>0$ and $\sum_{i=1}^5\psi_i<\pi$) and their positive diagonal scalings $DHD$, $DT(\psi)D$ for any positive diagonal matrix $D$. We computed the above SDP  for the following matrices $C=H, T(\psi)$ for thousands of random choices of $\psi$, and for hundreds of diagonal scalings $DHD$ of the Horn matrix, making sure to select $D$ such that $DHD\in \MK^{(1)}_5$ (following \cite[Theorem 4]{Laurent-Vargas_2023}).
The returned optimum value $p_C^*$ is (numerically) zero, so that no counterexample could be found.

Alternatively, one can try to show the following equality on the dual side: $\smash{(\tilDPS^{(2)}_n)^*}=(\DPS^{(2)}_n)^* + \Herm(S^2(\C^n))^\perp$.
That is, if  $M\in  {(\tilDPS^{(2)}_n)^*}$  (meaning 
$M\ot I_n= B+W_0+W_1$ with $B\in \Herm(S^3(\C^n))^\perp$, $W_0\succeq 0$,  $W_1^{T_{B[1]}}\succeq 0$), does there exist $M'\in (\DPS^{(2)}_n)^*$ (meaning 
$M'=B'+W_0'+W_1'+W_2'$ with $B'\in  \Herm(\C^n\ot S^2(\C^n))^\perp$, $W_0'\succeq 0$, 
$(W'_1)^{T_{B[1]}}\succeq 0$, $(W'_2)^{T_{B[2]}}\succeq 0$)  such that $M-M'\in \Herm(S^2(\C^n))^\perp$? 
(Here, we have used Theorems \ref{theo:FF} and \ref{theo:FF-BS}.)

\medskip
Consider now the inclusion $\SEPBS_n\subseteq \tilDPS^{(t)}_n$ (the left most inclusion in (\ref{eq:DPS-BS-inclusion})).
By  Theorem~\ref{theo:DPS-CP}, if a matrix $X$ belongs $(\MK^{(t)}_n)^*\setminus \CP_n$, then the associated Bose symmetric (LDUI) state $\rho_{(X,X)}^{T_B}$ belongs to ${\tilDPS^{(t+1)}_n}\setminus \SEPBS_n$ and thus the (CLDUI) state $\rho_{(X,X)}$ belongs to $\DPS^{(t+1)}_n\setminus \SEP_n$. Such  matrices $X$ do exist for any $n\ge 5$ and $t\ge 0$;
indeed, the inclusion $\MK^{(t)}_n\subset \COP_n$ is  shown to be strict in \cite{DDGH_2013}, which implies that also  $\CP_n\subset (\MK^{(t)}_n)^*$ is a strict inclusion. This   implies that   both inclusions $\SEPBS_n\subset {\tilDPS^{(t)}_n}$ and $\SEP_n\subset \DPS^{(t)}_n$ are strict  for any $n\ge 5$ and $t\ge 1$, as observed in \cite{GNP_2025}. In fact,  Fawzi \cite{Fawzi-SEP} shows that the separable cone $\SEP_n$ is not semidefinite representable for any $n\ge 3$, which implies that the inclusion $\SEP_n\subset \DPS^{(t)}_n$ is strict for any $n\ge 3$ and $t\ge 1$. In contrast, recall that $\SEP_2=\DPS^{(1)}_2$ (in fact this equality holds for bipartite states acting on $\C^3\ot \C^2$, see \cite{Woronowicz}). 

{Examples of matrices $A\in \COP_n\setminus \cup_{t\ge 0}\MK^{(t)}_n$ are constructed in \cite{Laurent-Vargas_2023} for any $n\ge 6$. Hence, in view of Theorem \ref{theo:DPS-KCOP}, any such matrix $A$ provides a LDUI instance {$M^{T_B}_{A,0}\in (\SEP^\BS_n)^*\setminus \cup_{t\ge 0} (\tilDPS^{(t)}_n)^*$}. In addition, it is shown in \cite[Theorem 8.5]{GNP_2025} how to use such matrices $A$ to construct  (C)LDUI matrices 
belonging to $\SEP_n^* \setminus \cup_{t\ge 0} (\DPS^{(t)}_n)^*$.
}

\medskip
As observed earlier, $\COP_4 = \MK^{(0)}_4$, so that the hierarchy collapses: $\MK^{(0)}_4=\MK^{(t)}_4$ for all $t\ge 1$. In \cite[Question 8.21]{GNP_2025} the authors ask whether the inclusion $\MK^{(t-1)}_n \subseteq \MK^{(t)}_n$ is strict for $n\ge 5$ and $t\ge 1$. Note that, in view of Theorem \ref{theo:DPS-CP}, a positive answer would  imply  {$\tilDPS^{(t)}_n\not=\tilDPS^{(t+1)}_n$.}
It is known that the inclusion $\MK^{(0)}_5\subset \MK^{(1)}_5$ is strict; for instance, the  Horn matrix belongs to  $\MK^{(1)}_5\setminus \MK^{(0)}_5$ \cite{Parrilo_2000}.
We are not aware of explicit examples of matrices lying in $\MK^{(t)}_n\setminus \MK^{(t-1)}_n$ for $t\ge 2$ and $n\ge 5$.
%\footnote{\ML{For us to discuss: the graph matrix $M_G=\alpha(G)(I+A_G)-J$ of the graph $G$, obtained by adding 9 isolated nodes to $C_5$, does not belong to $\MK^{(1)}_{14}$ - we showed this in \cite{Laurent-Vargas_2023}, so $\alpha(G)=2+9=11$ and $n=5+9=14$. I do not think that we checked whether it belongs to $\MK^{(2)}_{14}$. Maybe you can try to check it?}}

\medskip
Finally, let us mention again  the PPT$^2$ conjecture. 
It is   expected that   the conjecture   fails in general. However, {since it has been shown to hold when one state is LDOI and the other one is (C)LDUI \cite{Singh-Nechita-PPT2,Nechita-Park_2025},} it remains intriguing to see whether it holds when both states are  LDOI.
% states, which form a common generalization of CLDUI and LDUI states. 
% \label{sec:conclusion}

%\input{P-Temporary-Chapter-For-Examples.tex}

\bigskip\noindent
{\bf Acknowledgements.} 
We are grateful to Victor Magron for many helpful discussions and for bringing to our attention the preprint \cite{GNP_2025}.  We thank Aabhas Gulati, Ion Nechita and Sang-Jun Park for giving us further pointers to their recent works. We thank Sander Gribling for pointing out an error in Remark 2.7 of an earlier version of this work.

Our work has been supported by the European Union's
HORIZON-MSCA-2023-DN-JD programme under the Horizon Europe (HORIZON) Marie
Sk\l odowska-Curie Actions, grant agreement 101120296 (TENORS).
%\ML{Anything else to mention?}

\bibliographystyle{alpha}
\bibliography{ref.bib}

\appendix

\section{The PPT$^2$ conjecture}\label{A:PPT^2Conj}

We give here some background information about the PPT$^2$ conjecture.
Set $M_n(\C)=\C^{n\times n}$ and let $\MT_n(\C)$ denote the set of linear maps $\Phi: M_n(\C)\to M_n(\C)$. %, also known as {\em quantum channels}. 
Let $I_n, \sf T\in \MT_n(\C)$ denote, respectively,  the identity map and the transpose map. Given two maps $\Phi_1,\Phi_2\in\MT_n(\C)$, $\Phi_1\circ \Phi_2\in\MT_n(\C)$ denotes their  composition, defined by $\Phi_1\circ\Phi_2(X)=\Phi_1(\Phi_2(X))$ for $X\in M_n(\C)$. 
We group some definitions about linear maps in $\MT_n(\C)$ and refer, e.g., to \cite{SN-CLDUI,Singh-Nechita-PPT2} and further references therein for a detailed treatment.
%\footnote{\ML{Should we not assume that $\Phi$ maps Hermitian matrices to Hermitian matrices?}}

\begin{definition}
%    The set of all linear maps $\Phi:M_n(\C)\to M_n(\C)$ is denoted by $\mathcal{T}_n(\C)$. 
Consider a linear map   $\Phi\in\mathcal{T}_n(\C)$. % \ML{that preserves $\Herm^n$}. 
\\
$\bullet$ $\Phi$  is called {\em positive} if $\Phi(X)\succeq0$ for all $X\in\Herm(\C^n)_+$.\\
$\bullet$  $\Phi$ is called {\em completely positive} (CP)\footnote{This notion of CP linear map is classical in quantum information and should not be confused with the earlier notion of completely positive matrix.} if the map 
$I_r\ot\Phi:M_r(\C)\ot M_n(\C)\to M_r(\C)\ot M_n(\C)$   is positive for all $r\in\N$.\\
%, and $\Phi$ is called {\em completely copositive} (coCP)\footnote{These notions of CP and coCP maps  are classical in quantum information theory and should not be confused with the notions of completely positive and copositive matrices   considered earlier.} if 
%$\Phi\circ \sf T$ is completely positive.\\
%\footnote{in the context of maps $\circ$ denotes the usual concatenation of functions $f\circ g(x)=f(g(x))$}
 %T$ ($T$ denotes the the usual matrix transposition) is completely positive we call $\Phi$ completely copositive (coCP). 
 $\bullet$   $\Phi$ is  called  {\em PPT} if both $\Phi$ and $\Phi\circ \sf T$ are  completely positive.\\
 $\bullet$   $\Phi$ is called {\em entanglement breaking} if  $[I_n\ot \Phi](Y)\in\SEP_n$ for all $Y\in \Herm(\C^n\ot\C^n)_+$.
\end{definition}

Linear maps $\Phi\in\MT_n(\C)$ that are CP and preserve the trace (i.e., $\Tr(\Phi(X))=\Tr(X)$ for all $X\in M_n(\C)$) are known as {\em quantum channels}, they are used   in quantum information to  model the evolution of quantum systems. 
There is a  well-known correspondence between  maps $\Phi\in\MT_n(\C)$ and bipartite states in  $M_n(\C)\ot M_n(\C)$ via their   Choi (or Choi-Jamio\l kowski) matrix $J(\Phi)$.

\begin{definition}\label{def:Choimatrix}
For  $\Phi\in\mathcal{T}_n(\C^n)$,   its   Choi state $J(\Phi)\in M_n(\C)\ot M_n(\C)$ is defined by 
$$J(\Phi)=\sum_{r,s=1}^n\Phi(e_r e_s^*)\ot e_r e_s^* \quad \text{ i.e., }\quad 
(J(\Phi))_{ij,kl}= \Phi(e_je_l^*)_{ik} \text{ for } i,j,k,l\in [n].$$
Conversely, any state $\rho\in M_n(\C)\ot M_n(\C)$ corresponds to a linear map $\Phi_\rho \in\MT_n(\C)$ via the above relationship, given by $\Phi_\rho(e_je_l^*)_{ik}=\rho_{ij,kl}$ for all $i,j,kl\in [n]$, so that $J(\Phi_\rho)=\rho$.
\end{definition}

\noindent
Note that $J(\Phi)$ is Hermitian if and only if $\Phi$ preserves Hermitian matrices (i.e., $\Phi(X)$ is Hermitian whenever $X$ is Hermitian).  Moreover, the following links can be shown (see, e.g., \cite{SN-CLDUI, Singh-Nechita-PPT2}).

\begin{lemma}\label{lem:Phi-Choi}
    Let $\Phi\in\mathcal{T}_n(\C)$ and $J(\Phi)$ its associated Choi state. Then, 
     \begin{align*}
     \Phi \text{ is CP} & \Longleftrightarrow J(\Phi)\succeq 0,\\
     \Phi\circ \sf T \text{ is CP} &\Longleftrightarrow J(\Phi)^{T_B}\succeq 0,\\
        \Phi \text{ is PPT} &\Longleftrightarrow J(\Phi)\in\DPS^{(1)}_n,\\
        \Phi \text{ is entanglement breaking} &\Longleftrightarrow J(\Phi)\in\SEP_n.
    \end{align*}
\end{lemma}

\noindent
Hence, both $\Phi$ and $\Phi\circ \sf T$ are CP if and only if $J(\Phi)\in \DPS^{(1)}_n$. Another well-known fact is that any CP map $\Phi$ admits a Kraus decomposition: $\Phi(X)=\sum_{j =1}^k W_j X W_j^*$ for some $W_j\in M_n(\C)$ (and $k\le n^2$). Christandl posed the following conjecture in 2012, see also \cite{CMHW-2018}.

\begin{conjecture}[PPT$^2$ \text{conjecture}]\label{conj:PPT2}
    The composition $\Phi_1\circ \Phi_2$ of two arbitrary PPT linear maps $\Phi_1,\Phi_2$  is entanglement breaking.
\end{conjecture}

\begin{conjecture}[Restatement in state language]\label{conj:PPT2-states}
    Let $\rho_1,\rho_2 \in\Herm(\C^n\ot \C^n)$ and let $\Phi_{\rho_1},\Phi_{\rho_2}$ be the corresponding linear maps. If $\rho_1,\rho_2\in \DPS^{(1)}_n$, then   $J(\Phi_{\rho_1}\circ\Phi_{\rho_2})\in\SEP_n$.
\end{conjecture}

%\begin{remark}
%    The conjecture is only interesting when $\rho_1,\rho_2\not\in\SEP_n$. Indeed, if one of them is separable, then the corresponding linear map is entanglement breaking and the composition of two linear maps is entanglement breaking whenever at least one of them  is entanglement breaking.
%\end{remark}

As we now see, the conjecture holds if one of the two maps $\Phi_1, \Phi_2$ is entanglement breaking and the other one is CP.

\begin{lemma}\label{lem:conj-holds}
Assume $\Phi_1$ is CP and $\Phi_2$ is entanglement breaking, or vice versa, then their composition $\Phi_1\circ \Phi_2$ is entanglement breaking.
\end{lemma}

\begin{proof}
    The proof relies on the following (easy to check) identity: For $\Psi_1,\Psi_2,\Gamma_1,\Gamma_2\in\MT_n(\C)$,
     $$(\Psi_1\circ\Psi_2)\ot(\Gamma_1\circ\Gamma_2)=(\Psi_1\ot\Gamma_1)\circ(\Psi_2\ot\Gamma_2).$$
  Therefore, since $I_n=I_n\circ I_n$, we get 
  $$I_n\ot(\Phi_1\circ\Phi_2)=(I_n\ot\Phi_1)\circ(I_n\ot\Phi_2).$$
 Assume first that $\Phi_1$ is CP and $\Phi_2$ is entanglement breaking.   By assumption, $I_n\ot\Phi_2$ maps any $Y\in\Herm(\C^n\ot\C^n)_+$ to a separable state and separability is clearly preserved by $I_n\ot\Phi_1$. Now, assume $\Phi_2$ is CP and $\Phi_1$ is entanglement breaking. Then, 
 % $I_n\ot\Phi_2$ preserves membership in $\Herm(\C^n\ot\C^n)_+$ and consequently $I_n\ot\Phi_1$ would map to a separable state.
$I_n\ot \Phi_2$ maps  any $Y\in\Herm(\C^n\ot\C^n)_+$ to a positive semidefinite state, which is thereafter mapped by $I_n\ot\Phi_1$ to a separable state. 
  \end{proof}

Linear maps $\Phi\in \MT_n(\C)$ whose Choi matrix $J(\Phi)$ has some diagonal unitary invariance are investigated in \cite{SN-CLDUI,Singh-Nechita-PPT2}. In particular, $J(\Phi)\in \LDUI_n$ (resp., $J(\Phi)\in \CLDUI_n$) if and only if $\Phi(UXU^*) = U^*\Phi(X)U$ (resp., $\Phi(UXU^*)=U\Phi(X)U^*$) for all $U\in\MDU_n$ (in which case $\Phi$ is dubbed a DUC or CDUC map). The PPT$^2$ conjecture holds when restricting to such maps. 

\begin{theorem}\cite[Theorem 4.5]{Singh-Nechita-PPT2}\label{theo:PPT2CLDUI}  The PPT$^2$ conjecture holds when restricting to DUC and CDUC linear maps $\Phi_1,\Phi_2$. That is, if  $J(\Phi_1),J(\Phi_2))\in \CLDUI_n \cup \LDUI_n$ and $\Phi_1,\Phi_2$ are PPT,  then $\Phi_1\circ \Phi_2$ is entanglement breaking. 
%if $\Phi_1,\Phi_2$ are PPT
\end{theorem}

This motivates investigating (DOC) linear maps $\Phi$ whose Choi state $J(\Phi)$ belongs to $\LDOI_n$, in which case the conjecture has only been answered partially. The following results are shown in \cite[Corollary 6.4, Theorem 6.5]{Nechita-Park_2025} and generalize the result from Theorem \ref{theo:PPT2CLDUI}.

\begin{theorem}
    The PPT$^2$ conjecture holds when restricting to one DUC or CDUC linear map $\Phi_1$ and one DOC linear map $\Phi_2$. That is, if  $J(\Phi_1))\in \CLDUI_n \cup \LDUI_n$,  $J(\Phi_2)\in\LDOI_n$ and $\Phi_1,\Phi_2$ are PPT, then $\Phi_1\circ \Phi_2$ and $\Phi_2\circ \Phi_1$ are entanglement breaking.
    % if $\Phi_1,\Phi_2$ are PPT. 
    Furthermore, the composition of two random DOC maps generically satisfies the PPT$^2$ conjecture.
\end{theorem}

In what follows,  for the convenience of the reader, we explain  how to compute the DOC map of a LDOI state, i.e., the map $\Phi$ whose Choi matrix $J(\Phi)$ is a given LDOI state (see also  \cite[Lemma~9.3]{SN-CLDUI}).
%In the next lemma we compute \textcolor{red}{the corresponding DOC map of a LDOI state, i.e. the linear map $\Phi$, such that the Choi state $J(\Phi)$ coincides with the LDOI state}
%(extending the treatment in  \cite{Singh-Nechita-PPT2} for DUC/CDUC maps
For a matrix $Y\in M_n(\C)$, let $Y^0:=Y-\Diag(\diag(Y))$ denote the matrix obtained from $Y$ by setting to 0 its diagonal entries.

\begin{lemma}\label{lem:computePhi}
    Given $(X,Y,Z)\in\R^{n\times n}\times\Herm^n\times\Herm^n$ with $\diag(X)=\diag(Y)=\diag(Z)$, the  linear map corresponding  to $\rho_{(X,Y,Z)}$, denoted as $\Phi_{(X,Y,Z)}:=\Phi_{\rho_{(X,Y,Z)}}$ for short,    is defined by 
    $$\Phi_{(X,Y,Z)}(M)=\Diag(X\diag(M))+Y^{0}\circ M + Z^{0}\circ M^T\ \text { for } M\in M_n(\C).$$
\end{lemma}

\begin{proof}
    We compute the Choi state $J(\Phi_{(X,Y,Z)})=\sum_{i,j=1}^n\Phi_{(X,Y,Z)}(e_i e_j^*)\ot e_ie_j^*$ and show that it coincides with $$\rho_{(X,Y,Z)} =\sum_{i,j=1}^n X_{ij} e_ie_i^*\ot e_je_j^* +\sum_{i\ne j\in [n]} Y_{ij} e_ie_j^*\ot e_ie_j^*+\sum_{i\ne j\in [n]} Z_{ij} e_ie_j^* \ot e_je_i^*.$$
    For this we compute 
    \begin{align*}
   & \Phi_{(X,Y,Z)}(e_j e_j^*) = \Diag(Xe_j) = \Diag((X_{ij})_{i\in[n]}) = \sum_{i=1}^nX_{ij}e_ie_i^* \ &\text{ for } j\in [n],\\
  &\Phi_{(X,Y,Z)}(e_i e_j^*)=Y_{ij}e_ie_j^* + Z_{ji}e_je_i^* \ & \text{ for } i\ne j\in [n].
  \end{align*}
    Putting both computations together yields $\sum_{i,j=1}^n\Phi_{(X,Y,Z)}(e_i e_j^*)\ot e_ie_j^*=\rho_{(X,Y,Z)}$, as desired.
\end{proof}

\begin{definition}\label{def:DOC123}
    Let $(X_1,Y_2,Z_2),(X_2,Y_2,Z_2)\in\R^{n\times n}\times\Herm^n\times\Herm^n$ such that, for $k=1,2$,  $\diag(X_k)=\diag(Y_k)=\diag(Z_k)$.
    % and $\diag(X_2)=\diag(Y_2)=\diag(Z_2)$. 
    We define their composition `$\bullet$' by
    $$(X_1,Y_2,Z_2)\bullet(X_2,Y_2,Z_2)=(X_1X_2,Y_1\circ Y_2+Z_1\circ Z_2^T+DY_3,Y_1\circ Z_2+Z_1\circ Y_2^T+DZ_3),$$
   setting $DY_3:=\Diag(X_1X_2-Y_1\circ Y_2-Z_1\circ Z_2^T)$ and $DZ_3:=\Diag(X_1X_2-Y_1\circ Z_2-Z_1\circ Y_2^T)$ (that are used to ensure that the three arguments in the resulting matrix triplet share the same diagonal).
\end{definition}

\begin{lemma}\label{lem:Phi123}
    Let $(X_1,Y_2,Z_2),(X_2,Y_2,Z_2)\in\R^{n\times n}\times\Herm^n\times\Herm^n$ such that, for $k=1,2$, $\diag(X_k)=\diag(Y_k)=\diag(Z_k)$, and  $(X_3,Y_3,Z_3):=(X_1,Y_2,Z_2)\bullet(X_2,Y_2,Z_2)$ (as in Definition~\ref{def:DOC123}). Then,  we have
       $$\Phi_{\rho_{(X_3,Y_3,Z_3)}} = \Phi_{\rho_{(X_1,Y_1,Z_1)}}\circ\Phi_{\rho_{(X_2,Y_2,Z_2)}}.$$
\end{lemma}

\begin{proof}
Let $M\in M_n(\C)$. By Lemma \ref{lem:computePhi},  we have
\begin{align*}
& \Phi_{(X_1,Y_1,Z_1)}\circ \Phi_{(X_2,Y_2,Z_2)} (M) = \Phi_{(X_1,Y_1,Z_1)}(\Phi_{(X_2,Y_2,Z_2)}(M))\\
& = \underbrace{\Diag(X_1 \diag(\Phi_{(X_2,Y_2,Z_2)}(M)))}_{=A} + \underbrace{Y_1^0 \circ \Phi_{(X_2,Y_2,Z_2)}(M)}_{=B}
 + \underbrace{Z_1^0 \circ (\Phi_{(X_2,Y_2,Z_2)}(M))^T}_{=C}.
 \end{align*}
 %   \begin{align*}
%        \Phi_{(X_1,Y_1,Z_1)}(\Phi_{(X_2,Y_2,Z_2)}(M))&=\Phi_{(X_1,Y_1,Z_1)}(\Diag(X_2\diag(M))+Y_2^{0}\circ M + Z_2^{0}\circ M^T)\\
 %       &=A+B+C
 %   \end{align*}
%    where 
%    \begin{align*}
 %       A &= \Diag(X_1\diag(\Diag(X_2\diag(M))+Y_2^{0}\circ M + Z_2^{0}\circ M^T)),\\
 %       B &= Y_1^0\circ(\Diag(X_2\diag(M))+Y_2^{0}\circ M + Z_2^{0}\circ M^T),\\
 %       C &= Z_1^0\circ(\Diag(X_2\diag(M))+Y_2^{0}\circ M + Z_2^{0}\circ M^T)^T.
  %  \end{align*}
We compute each of the three summands $A,B,C$, using the fact (again, by Lemma \ref{lem:computePhi}) that 
$$\Phi_{(X_2,Y_2,Z_2)}(M)=\Diag(X_2\diag(M))+ Y_2^0\circ M+Z_2^0\circ M^T.$$
As $Y_2^0$ and $Z_2^0$ are traceless, the same holds for $Y_2^0\circ M+Z_2^0\circ M^T$ and thus  we have
    \begin{align*}
        A = \Diag(X_1\diag(\Diag(X_2\diag(M)))) = \Diag(X_1X_2\diag(M)).
    \end{align*}
 As $Y_1^0$ and $Z_1^0$ are traceless, we obtain
     \begin{align*}
        B = Y_1^0\circ(Y_2^{0}\circ M + Z_2^{0}\circ M^T) = Y_1^0\circ Y_2^0\circ M + Y_1^0\circ Z_2^0\circ M^T,\\
        C = Z_1^0\circ(Y_2^{0}\circ M + Z_2^{0}\circ M^T)^T = Z_1^0\circ (Y_2^0)^T\circ M^T + Z_1^0\circ (Z_2^0)^T\circ M.
    \end{align*}
    Putting everything together we obtain $$ A + B + C = \Diag(X_1X_2\diag(M)) + (Y_1^0\circ Y_2^0 + Z_1^0\circ (Z_2^0)^T)\circ M + (Y_1^0\circ Z_2^0 + Z_1^0\circ (Y_2^0)^T)\circ M^T,$$
which concludes the proof.
\end{proof}

Using Lemma \ref{lem:Phi123} we now see that Conjecture \ref{conj:LDOIstates} is indeed a reformulation of Conjecture \ref{conj:PPT2-states} for the case of LDOI states.

\end{document}